\documentclass[oneside,11pt]{book}
\usepackage{amsmath,amsfonts}
\usepackage{graphicx}
\usepackage{subfig}
\usepackage{verbatim}

\newtheorem{theorem}{Theorem}

\newtheorem{claim}[theorem]{Claim}

\newtheorem{condition}[theorem]{Condition}

\newtheorem{corollary}[theorem]{Corollary}

\newtheorem{definition}[theorem]{Definition}

\newtheorem{lemma}[theorem]{Lemma}

\newtheorem{remark}{Remark}

\newenvironment{proof}[1][Proof]{\noindent\textbf{#1.} }{\ \rule{0.5em}{0.5em}}
\newcommand{\sgn}{\mathop{\mathrm{sgn}}}

\newcommand{\Pp}{\mathbf{P}}
\newcommand{\E}{\mathbf{E}}
\newcommand{\eps}{\epsilon}

\newcommand{\R}{\mathbb{R}}
\newcommand{\wt}{\widetilde}
\newcommand{\tauh}{\hat{\tau_\epsilon}}
\newcommand{\tauw}{\widetilde{\tau}_\epsilon}
\newcommand{\aone}{{\alpha_1^-}}
\newcommand{\atwo}{{\alpha_2^-}}
\newcommand{\Wc}{\mathcal{W}}
\newcommand{\Span}{\mathop{\rm span}}
\newcommand{\one}{\mathbf{1}}

\newcommand{\Fc}{\mathcal{F}}

\newcommand{\N}{\mathbb{N}}
\newcommand{\Z}{\mathbb{Z}}

\newcommand{\ONE}{{\bf 1}}

\newcommand{\x}{X_\eps}

\newcommand{\sgm}{\tilde \sigma}

\renewcommand\Re{\operatorname{Re}}

\title{Scaling Limit for the Diffusion Exit Problem}
\author{Sergio Angel Almada Monter}

\begin{document}

\maketitle
\listoffigures
\tableofcontents

\chapter{Introduction}\label{ch: Intro}
In this thesis we study the so called exit problem~\cite[Section 4.3]{Freidlin--Wentzell-book} for small noise diffusion. This model belongs to the more general area of random perturbations of dynamical systems, which has been a very active area of research over the last 30 years~\cite{GentzBook},~\cite{Freidlin--Wentzell-book},~\cite{Vares}. The small noise diffusion framework has attracted  the interest of both the  pure and applied mathematics communities. From the mathematical standpoint  it is interesting
because this area has strong interactions with other important branches of mathematics such as probability theory, dynamical systems, or PDE. As regards applied mathematics,  the set of problems relating to small noise diffusion has found applications in climate modeling~\cite{Clima},~\cite{GentzClimate}, electrical engineering~\cite{Zei2},~\cite{Zei1},~\cite{Zei3},  finance~\cite{FinancePaper},~\cite{FinanceBook}, neural dynamics~\cite{rabinovich:188103},~\cite{Biology} among others~\cite{DemboZeitouni}. The main focus of the thesis, the exit problem, was originally motivated by applications on the reaction rate theory of chemical physics~\cite{Quimica}. Moreover, the work presented here, although purely theoretical, was motivated by neural dynamics~\cite{rabinovich:188103}. We proceed to describe the setting, in order to provide a more extensive background.

The setting of this problem is as follows. Given a smooth vector field $b: \R^d \to \R^d$ consider the It\^o equation driven by the $d$-dimensional standard Wiener process $W$:
\begin{align}
\label {eqn: Ito_additive} d\x (t) &= b( \x (t) ) dt + \eps \sigma( \x (t) ) dW(t), \\
\x (0 ) &= x_0. \notag
\end{align}
Here $\sigma:\R^d \to \R^{ d\times d}$ is a smooth uniformly non-degenerate matrix valued function. That is, the matrix $a = \sigma^T \sigma$ is uniformly positive definite. Under these assumptions we can ensure that equation~\eqref{eqn: Ito_additive} has a unique strong solution (see~\cite{Karatzas--Shreve} or~\cite{Protter} for all stochastic analysis references).

Given an initial condition $x_0 \in \R^d$ (or a set of initial conditions), the goal is to characterize some asymptotic properties of $\x$ as $\eps \to 0$. In particular we focus on the exit from a domain problem or exit problem for short. Consider a domain (open, bounded and connected) $D \subset \R^d$ with piecewise smooth boundary (at least $C^2$). The exit problem is the study of the time
\[
\tau_\eps^D (x)= \inf \{ t>0: \x(t) \in \partial D \},
\]
at which $\x$ exits, and the exit distribution $\Pp_{x_0} \{ \x ( \tau_\eps^D ) \in \cdot \} $. In this work we aim for a joint asymptotic result on the distribution of $(\tau_\eps^D, \x ( \tau_\eps^D ) )$ under certain assumptions for $b$.

As we said before, this problem had its origin in chemical physics: it is a glorified model for the speed at which chemical reactions take place. The first model of this kind was proposed by Kramers~\cite{Kramers}, we refer to~\cite{GentzQuimica} for a modern treatment.

From a pure mathematics perspective, the problem became of interest because it provides a framework to compute asymptotic (as $\eps \to 0$) properties for solutions of the Dirichlet problem
\begin{align} \label{eqn: intro-PDE}
\nabla u_\eps (x)  \cdot b (x) + \frac{ \eps^2 }{2} \Delta u_\eps (x)&=0, \quad x \in D ,\\
 u_\eps (x)&=g (x), \quad x \in \partial D.
\end{align}
Indeed, using a relaxed version of the Feynman-Kac theorem~\cite[Theorem 4.4.2]{Karatzas--Shreve}, the solution to this PDE can be written as the average $u_\eps(x)=\mathbf{E}_{x} g ( \x ( \tau_\eps^D) ) $, where $\x$ is the solution of~\eqref{eqn: Ito_additive} with $\sigma=\rm{Id}$. Solutions to other PDE's can be written as a similar average, but for this discussion we choose~\eqref{eqn: intro-PDE} since it is representative of the area. Although the asymptotic study of the function $u_\eps$ is now known to be strongly related to the solution to the exit problem, the first studies relied only on analytic non-probabilistic arguments.  For example, in~\cite{Devinatz},~\cite{Kamin1},~\cite{Kamin2} under some assumptions on the drift $b$, they were able to rigorously write the solution as a formal series in $\eps$. In~\cite{Schuss} it is proved, under the assumption that $b(0)=0$ and that $D$ is contained in the basin of attraction of $0$,  that as $\eps \to 0$, $u_\eps$ converges to a constant. This result is of great importance since it means that the system forgets its initial condition. However, this is not always true. For instance in~\cite{Levinson} using functional analysis and dynamical systems  for the Levinson case (see section~\ref{sec: intro_FW}) it is proved that the limiting function at $x$ depends on the orbit of $x$ under the action of the flow generated by $b$.  In~\cite{Kifer0} and~\cite{Kifer} a combined approach involving probabilistic and purely analytic arguments was put into practice. Very general limit theorems were obtained, but strong restrictions on the non-linearity of the drift $b$ were imposed.

	The standard mindset in tackling this problem from the probabilistic point of view is to think of the SDE that defines $\x$ as a (random) singular perturbation of the system $\dot{x}=b(x)$. In this context it is natural to expect that the methods used to study the exit problem lie in the intersection between probability theory and dynamical systems. Freidlin and Wentzell~\cite{Freidlin--Wentzell-book},~\cite{Vares} were the ones who put together a general theory in this direction. They based their theory on the Large Deviation principle for $\x$.  This result was then used as the building block in constructing what today is known as  the Freidlin-Wentzell theory. The core of this theory strongly relates the exit behavior of $\x$ to the properties of the vector field $b$ by providing two elements:
\begin{enumerate}
\item It defines a function $V:D \times \partial D \to [0,\infty ]$, known as quasi-potential, that characterizes the exit distribution. Indeed, the exit distribution of $\x$ is asymptotically concentrated on the set $V^*_{\x(0)}$ of minimizers of $V(\x(0), \cdot )$. Moreover, under the assumption that $\sigma$ is uniformly non-degenerate, $V$, and hence $V^*_{ \x (0) }$, depends mostly on $b$. For example it can be shown that in the case of $b=\nabla \phi$, $V$ is proportional to $\phi$. This is the reason why the  $V$  function is called quasi-potential.
\item It shows that in the case in which the domain is contained in the basin of attraction of an equilibrium , $\eps^{-2} \log \tau_\eps^D$ converges in probability to the minimum of $V( \x(0), \cdot )$. This result is of vital importance since it gives a hierarchy of transition for the case in which there are several equilibria. See~\cite[Section 6.5]{Freidlin--Wentzell-book} for more background on this particular direction.
\end{enumerate}
The theory came to light with a series of papers beginning with~\cite{FWPaper1} and~\cite{FWPaper2} until the Russian edition of the book~\cite{Freidlin--Wentzell-book} appeared.  See~\cite{KoralovSPA} and ~\cite{KoralovPTRF} for a modern version of the theory, and~\cite{SPDE1},~\cite{Cerra1} and references therein for a stochastic partial differential equations version of the theory. In Section~\ref {sec: intro_FW} we give a brief review of the Freidlin-Wentzell theory necessary to understand the motivation of this work.
	
In contrast with  the Freidlin-Wentzell theory, that mostly relies on the large deviation principle, a modern trend relying on a path-wise approach has emerged in the last years. As a consequence, more detailed phenomena can be captured. That is the case, for example, in~\cite{nhn} in which a heteroclinic network is considered or in~\cite{GentzP1} in which a bifurcation problem is studied. The monograph~\cite{GentzBook} contains several examples in this direction together with applications.

In this thesis, a modern and more complete treatment of two cases is developed: the saddle case in which the vector field has a unique saddle point and the Levinson case in which the deterministic dynamics escape from the domain in a finite time. The two results when combined complete the treatment of the case in which the underlying dynamics admit a heteroclinic network as studied in~\cite{nhn}. We also provide a $1$-dimensional example that explains how to obtain a correction to the exit time for a diffusion conditioned to exit through an unlikely exit point. The approach presented here relies heavily on the underlying dynamical structure, and combines techniques from differential equations, bifurcation theory and martingale theory.

The rest of this chapter is organized as follows. In Section~\ref{sec: intro_FW} we give a brief introduction to Freidlin-Wentzell theory. In Section~\ref{sec: saddle} we study the exit problem in the case the system $\dot{x}=b(x)$ has a saddle point. In Section~\ref{sec: unstable} we study the escape when it takes a finite time in the so called Levinson setting. Applications of these two cases are presented in Sections~\ref{sec: 1d-time-reversal} and~\ref{sec: Intro-PlanarHetero}. The general setting and a brief description of the chapters is given in Section~\ref{Sec: Intro-Notation}.

\section{Background and Motivation}\label{sec: intro_FW}

In this section we gather the background tools that will allow us to explain where our results stand with respect to the classical Freidlin-Wentzell theory.

Consider $\x$, the strong solution to the SDE~\eqref{eqn: Ito_additive}. If seen as a random perturbation, the equation for $\x$ suggests that the process should behave like the flow generated by $b$:
\begin{align}
\label {eqn: det_flow} \frac{d}{dt}S^t x = b(S^tx ), \quad
S^tx = x.
\end{align}
Indeed, through a standard martingale argument, it is easy to see that for any $\delta>0$ there are constants $C_{T,\delta}^{ (1) }$ and $C_{T,\delta}^{(2)}$ such that
\begin{equation} \label{eqn: estimate_basic}
\sup_{ x \in \R^d} \Pp_x \left \{  \sup_{ t \leq T} | \x (t) - S^t x| > \delta \right \} \leq C_{T,\delta}^{ (1) } e^{ - C_{T,\delta}^{ (2) } \eps^{-2} }.
\end{equation}
This inequality can be used to show that $\Pp_x^\eps$, the law of $\x$ on $C([0,T];\R^d)$ conditioned to $\x(0)=x$, converges weakly (on the space $C([0,T];\R^d)$) to the measure concentrated on the orbit of $x$. See~\cite{Blagoveschenskii:MR0139204} for a series expansion in $\eps$ and~\cite{Cas93} for a series expansion of more general stochastic flows.

The question now is to find the optimal constant $C_{T,\delta}^{(2)}$ in~\eqref{eqn: estimate_basic}, or more generally to find a large deviation principle~\cite{DemboZeitouni},~\cite{DenHollanderLDP} or Appendix~\ref{ch: appendix}:
\begin{theorem} [Freidlin-Wentzell~\cite{Freidlin--Wentzell-book} ] \label{thm: LDP}
Let $H_{0,T}^1$ be the space of all absolutely continuous functions from $[0,T]$ to $\R^d$ with square integrable derivatives. Define the functional $I_T^x$ by
\begin{equation}  \label{eqn: intro-LDP_I}
I_T^x (\varphi)= \frac{1}{2} \int_0^T \langle \stackrel{\cdot}{\varphi} (s)-b(\varphi(s) ), a^{-1} (\varphi (s) ) ( \stackrel{\cdot}{\varphi} (s)-b(\varphi(s) ) ) \rangle ds,
\end{equation}
if $\varphi \in H_{0,T}^1$ and $\varphi(0)=x$, and $\infty$ otherwise. Here $b$ is the drift in~\eqref{eqn: Ito_additive} and $a=\sigma^T \sigma$, with $\sigma$ the diffusion matrix in~\eqref{eqn: Ito_additive}.

Then for each $x \in \R^d$ and $T>0$ the family $(\Pp_x^\eps)_{\eps>0}$ satisfies a Large Deviation Principle on $C([0,T];\R^d)$ equipped with uniform norm at rate $\eps^2$ with good rate function $I_T^x$.
\end{theorem}
See~\cite{BeC96} and~\cite{NualartLDP} for more large deviations results. In order to have this thesis as self contained as possible, we give a large deviation overview on Appendix~\ref{ch: appendix} .

Informally, Theorem~\ref{thm: LDP} says that if $A \subset C([0,T];\R^d)$ then
\[
\Pp_x \left\{   \x \in A \right\} \asymp e^{- \eps^{-2} \inf_{\varphi \in A} I_T^x (\varphi)}.
\]
Intuitively, due to~\eqref{eqn: intro-LDP_I}, this result suggest that $I_T^x$ serves as a measure on how costly (in terms of probability) is for the system $\x$ not to follow the deterministic trajectory. This interpretation is essential when solving problems that require non-compact time frames, in particular, when studying the exit problem described above.

Regarding $I_T^x$ as a cost function, it make sense to introduce
\begin{equation} \label{eqn: quasipotential}
V(x,y) = \inf_{T>0} \left \{ I_T^x ( \varphi) : \varphi (T)=y, \varphi( [0,T] ) \subset D \cup \partial D \right \}
\end{equation}
as the cost to go from $x$ to $y$ inside $D$. The function $V: D \times  \partial D \to [0, \infty]$, known as the quasipotential, plays an important role on the exit problem we described:
\begin{theorem} [Freidlin-Wentzell~\cite{Freidlin--Wentzell-book}]
Suppose $\x (0)=x_0$ and let \[z=\inf_{y\in \partial D } V(x_0,y).\] Then for every closed set $N \subset \partial D$ such that $\inf_{y\in N } V(x_0,y) > z$,
\[
\lim_{\eps \to 0} \Pp_{x_0} \{ \x (\tau_\eps) \in N \}=0.
\]
\end{theorem}

The necessary observation derived from this theorem is that, in the limit, the exit occurs in a neighborhood of the set of minimizers of the quasipotential.  The limitation is that when there are several minimizers, the result doesn't provide any distinction between them. For example, suppose the phase portrait of $S$ is as in Figure~\ref{fig: heteroclinic} and $D$ is a rectangular region as in Figure~\ref{fig: heteroclinic}. Then the set of minimizers consists of $3$ points: $q_1$, $q_2$ and $q_3$. Freidlin-Wentzell theory ensures that, asymptotically,  the exit occurs outside this set with exponentially small probability, but doesn't distinguish between the minimizer. For example, the theory is not able to establish if it is more likely to exit on a neighborhood of $q_1$ or in a neighborhood of $q_2$.   Bakhtin~\cite{nhn} started a theory that will allow us to compute the probability of exiting close to each of the minimizers. To complete this theory is part of the motivation for this work.

\begin{figure}
\centering
\includegraphics[width=4in]{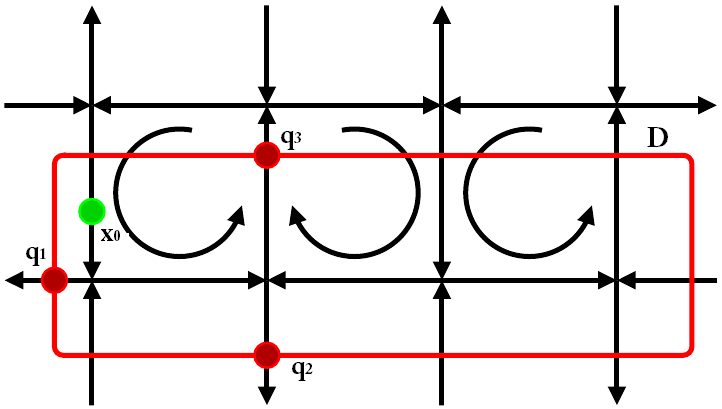}
\caption{Example with several minimizers}
\label{fig: heteroclinic}
\end{figure}

This work (see also ~\cite{NormalForm}) drops some technical assumptions needed in~\cite{nhn} in the planar case. The program makes intensive use of normal form theory and provides small noise estimates for non-linear diffusion. With this work, the asymmetric behavior found in~\cite{nhn} is extended to arbitrary Hamiltonian systems on the plane. Moreover, the present text also (see ~\cite{Levinson}) provides exact scaling corrections when the flow $S$ exits the domain in a finite time. This is a step towards to both Freidlin-Wentzell theory and Bakhtin's heteroclinic result, since extension to the case in which $S$ is asymptotically stable can be carried out by time reversing. In this direction, a $1$ dimensional example is presented in Chapter~\ref{ch: levinson}. This is an open and promising future research area, since it may provide scaling limits for exit points under very general assumptions on $b$.

\section{Escape from a Saddle} \label{sec: saddle}

In this section, we assume that the system has a unique critical point and that point is a saddle. Without loss of generality, suppose that the critical point is the origin; that is, we are assuming that $0 \in \R^d$ is the only point $x \in D\cup \partial D$ such that $b(x)=0$, and the matrix $A=\nabla b(0)$ has spectrum bounded away from zero with at least one eigenvalue with positive real part and one eigenvalue with negative real part. In other words, there is a pair of integers $\nu, \mu \geq 1$ such that the eigenvalues $\lambda_1,...,\lambda_d$ of $A$ satisfy
\[
\Re \lambda_1 = ...= \Re \lambda_\nu > \Re \lambda_{\nu+1} \geq ...\geq \Re \lambda_\mu >0 > \Re \lambda_{\mu+1} \geq ... \geq \Re \lambda_d.
\]
Under this assumptions, it is well known~\cite[Theorem 3.2.1]{Wiggins} that $\bar{D}=D \cup \partial D$ can be decomposed as  $\bar{D} = \{0\}\cup \Wc^u \cup \Wc^c \cup \Wc^s$, where
\begin{align*}
\Wc^u &=\{ x\in \bar{D}:  \lim_{t \to -\infty } S^tx = 0 \text{, and for some } s\geq 0, S^{(-\infty,s)}x \subset D \text{ and } S^{(s,\infty)}x\cap \bar {D} = \emptyset  \}, \\
\Wc^c &=\{ x \in \bar{D}: S^{(s_1,s_2)}x \subset D \text{ and } S^{[s_1,s_2]^c}x \cap \bar{D} = \emptyset, s_1\leq0 \text{ and } s_2\geq0  \} ,
\end{align*}
and,
\[
\Wc^s =\{ x\in \bar{D}:  \lim_{t \to \infty } S^tx = 0 \text{, and for some } s \leq 0, S^{(-\infty,s)}x \cap \bar {D} = \emptyset \text{ and } S^{(s,\infty)}x \subset D \}.
\]
Here, for an interval $A \subset \R$, $S^Ax$ denotes the set
\[
S^A x = \{ S^t x: t\in A \}.
\]
We are ready to state the theorem in~\cite{Kifer} concerning the exit time $\tau_\eps$:
\begin{theorem} \label{thm: kifer}
 If $x \in D \cap \Wc^s$, then
 \[
 -\frac{ \tau_\eps} { \log \eps } \stackrel { \Pp } { \longrightarrow} \frac{1}{\Re \lambda_1} , \quad \eps \to 0.
 \]
Consider the (deterministic) time
\[
T(x)= \inf \{ t>0: S^t x \in \partial D \},
\]
then, if $x \in ( \Wc^c \cup \Wc^u) \cap D$,
\[
\tau_\eps \stackrel { \Pp } { \longrightarrow} T(x), \quad \eps \to 0.
\]
\end{theorem}

In order to state the corresponding theorem for the exit distribution, denote $\Gamma_{\max}$ the generalized eigenspace of $A$ which corresponds to $\lambda_1,...,\lambda_\nu$. The Hadamard--Perron theorem~\cite[Section 2.7]{Perko},~\cite[Theorem 3.2.1]{Wiggins} states that there is a $\nu$-dimensional $S^t$-invariant submanifold $W_{max}$ tangent to $\Gamma_{\max}$ at the origin. Note that the intersection $Q_{\max}=W_{\max} \cap \partial D$ is not empty. Moreover, in the case of $\nu > 1$, $Q_{\max}$ is a $\nu-1$-dimensional manifold, while for $\nu=1$ consists of two points: $Q_{\max}=\{ q_-, q_+ \}$. The result in~\cite{Kifer} reads:

\begin{theorem} \label{thm: kifer_space}
If $x \in D \cap \Wc^s$ and $\nu>1$, then for any relatively open subset $G \subset \partial D$ such that $Q_{\max} \subset G$, it holds that
\[
\lim_{\eps \to 0} \Pp_x \{ \x ( \tau_\eps ) \in G\} =1.
\]

If $\nu=1$ then the measure $\Pp_\eps^x ( \cdot) = \Pp\{ \x (\tau_\eps) \in \cdot | \x (0)=x\}$ converges weakly to the measure $\frac{1}{2} \delta_{q_-} +\frac{1}{2} \delta_{q_+}$, where $\delta_z$ is the probability measure concentrated at $z$.

If $x \in ( \Wc^c \cup \Wc^u) \cap D$, $\Pp_\eps^x ( \cdot)$ converges weakly to the measure $\delta_{S^{T(x)} x }$.
\end{theorem}

In the work by Day~\cite{DaySPA} a refinement to the theorem about the exit time is given in the $2-$dimensional situation. He proved that $\lambda_1 \tau_\eps$ can be written as a sum between $-\log \eps$ and a tight correction. He gave a precise asymptotic description for the distribution of the correction:
\begin{theorem} \label{thm: Day}If $d=2$ and $\x (0 ) \in \mathcal{W}^s \cap D$, then
\[
\lambda_1 \tau_\eps^D + \log \eps \to \mathcal{K} + C_\nu
\]
in distribution. Here $\mathcal{K}$ and $C_\nu$ are independent random variables. Moreover, $\mathcal{K}$ has a density with respect to Lebesgue measure given by
\[
d\mathcal{K} = \frac{2}{\sqrt{ \pi } } e^{ - ( x + e ^{ -2x } ) } dx,
\]
and $C_\nu$ is a Bernoulli random variable with $\Pp \{ C_\nu = C_\pm  \} = 1/2$, where $C_\pm$ are a constants depending only on $b$ and $\sigma$.
\end{theorem}

This theorem complements previous work by Mikami~\cite{Mikami} which established the decay in the distribution of $- \frac{\lambda_1}{\log \eps} \tau_\eps^D$:
\begin{theorem} \label{thm: Mikami}For an arbitrary $d \geq 1$ and $\x (0 ) \in \mathcal{W}^s \cap D $ for $T\in(0,1)$ it holds that
\[
\lim_{ \eps \to 0 } \frac{1}{\log \eps} \log \left (  - \log \Pp \left\{ - \frac{\lambda_1}{\log \eps} \tau_\eps^D < T \right \} \right )=T-1,
\]
while for $T>1$,
\[
\lim_{ \eps \to 0 } \frac{1}{\log \eps} \Pp \left\{ - \frac{\lambda_1}{\log \eps} \tau_\eps^D > T \right \} = (T-1)/2.
\]
\end{theorem}

Theorem~\ref{thm: Mikami} is a good refinement of Theorem~\ref{thm: kifer}, but it doesn't prove that the distribution of the difference $\lambda_1 \tau_\eps^D + \log\eps$ is tight as in Theorem~\ref{thm: Day}. The tails established in Theorem~\ref{thm: Mikami} are consistent with the tails of Theorem~\ref{thm: Day}.

Bakhtin~\cite{Bakhtin-SPA} gave a refinement of the last theorems. He proved that in any dimension the family of random variables $\tau_\eps + \lambda_1^{-1} \log \eps$ converges in distribution and he identified the limit. He also improved the result about the convergence on the exit distribution in Theorem~\ref{thm: kifer_space}. In the case $\nu=1$, he showed that the factor $1/2$ in the limiting distribution $\frac{1}{2} \delta_{q_+} + \frac{1}{2} \delta_{q_-}$ comes from the symmetry of a certain Gaussian random variable.
\begin{theorem}\label{thm: BakhtinSPA}
In the case $\nu=1$, parametrize the manifold $W_{\max}$ as a $C^1$-curve $\gamma$ that can be locally represented as a graph over the (one dimensional) space $\Gamma_{\max}$.
Let the times $t(\pm \delta, q_\pm )$ be the time it takes to the deterministic flow to go from $\gamma(\pm \delta)$ to $q_\pm$:
\[
t(\pm \delta, q_\pm ) = T ( \gamma ( \pm \delta ) ),
\]
with $T(x)$ defined as in Theorem~\ref{thm: kifer}.

Then, the numbers
\[
h_{\pm} = \lim_{\delta \to 0} \left( \frac{\log \delta } {\lambda_1} + t( \pm \delta, q_\pm) \right)
\]
are well defined and, in the case where $\sigma= I$, and $x \in D \cap \Wc^s$,  there is a Gaussian random variable $\mathcal{N}$ and a number $\kappa=\kappa(x)>0$ such that :
\begin{enumerate}
\item As  $\eps \to 0$ the convergence
\[
\x (\tau_\eps) \stackrel{\Pp}{\longrightarrow} q_+ \ONE_{ \{ \mathcal{N} > 0 \} } + q_- \ONE_{ \{ \mathcal{N} < 0 \} },
\]
and
\[
\tau_\eps + \frac{1}{\lambda_1} \log \eps \stackrel{\Pp}{\longrightarrow}  h_+ \ONE_{\{ \mathcal{N} > 0 \}} + h_- \ONE_{ \{ \mathcal{N} < 0 \} } - \frac{1}{\lambda_1} \log ( \kappa \mathcal{N} )
\] hold in probability.
\item As $\eps \to 0$ the distribution of the random vector $( \x (\tau_\eps), \tau_\eps + \frac{1}{\lambda_1} \log \eps)$ converges weakly to the measure
\[
\frac{1}{2} \delta_{q_+} \times \mu_{h_+,\omega} + \frac{1}{2} \delta_{q_-} \times \mu_{h_-,\omega},
\]
where $\mu_{h_\pm,\omega}$ is the distribution of
\[
h_\pm -  \frac{1}{\lambda_1} \log ( \kappa \mathcal{N} ).
\]
\end{enumerate}
\end{theorem}

The proof of Theorems~\ref{thm: Day} and~\ref{thm: BakhtinSPA} is based on the study of the linear system and then an approximation to the non-linear one. The steps followed for a prototypical $2-$dimensional system is presented in the next section. Part of this thesis focus is on adapting this methodology to the non-linear case.
\subsection{A $2$-dimensional linear example} \label{subsec: sketch}

For a fixed $\delta>0$, consider the domain $D=(-\delta,\delta) \times (-\delta,\delta)\subset \R^2$. Given two positive numbers  $\lambda_\pm >0$, we sketch the solution to the exit problem from $D$ for the diffusion $\x=(x_\eps^1,x_\eps^2)$ given by
\begin{align*}
d\x (t) &= { \rm{diag} }( \lambda_+, - \lambda_- ) \x (t) dt + \eps dW(t),\\
\x(0) &= ( 0, x_0 ) \in D.
\end{align*}
Here, for a column vector $v=(v_1,v_2)$, $v_1$ is the first coordinate and $v_2$ is the second one.

Using It\^o's formula~\cite[Theorem 3.3.3]{Karatzas--Shreve} in each coordinate we write the Duhamel principle for $x_\eps^1$ and $x_\eps^2$ as
\begin{align} \label{eqn: Duh_Intro1}
x_\eps^1 (t) &= \eps e^{ \lambda_+ t } \int_0^t e^{-\lambda_+ s} dW_1 (s), \\ \label{eqn: Duh_Intro2}
x_\eps^2 (t) &= e^{-\lambda_- t } x_0 +  \eps  \int_0^t e^{-\lambda_- ( t- s) } dW_2 (s).
\end{align}
These two identities are the main ingredient in this development.

We will show that $\x$ exits $D$ along $(\delta,0)$ or $(-\delta,0)$. Consider the time at which $x_\eps^1$ exits the interval $(-\delta,\delta)$:
\[
\tau_\eps^\delta= \inf\{t>0: |x_\eps^1 (t) | = \delta \}.
\]
The program is to compute $\tau_\eps^\delta$ based on the path-wise properties of $\x$. Use the identity obtained for $\tau_\eps^\delta$ to characterize $\x(\tau_\eps^\delta)$ from which we will deduce that $\Pp \{ \tau_\eps^\delta = \tau_\eps^D \} \to 1$ as $\eps \to 0$. Hence, we can obtain the limiting behavior of $(\tau_\eps^{D},\x)$ based on the identities we have obtained for $ (\tau_\eps^\delta,\x(\tau_\eps^\delta) )$. This will establish a result in the spirit of Theorem~\ref{thm: BakhtinSPA}.

First, note that $\tau_\eps < \infty$ with probability $1$. This is a classical fact that we prove in Appendix~\ref{ch: app-sec} for completeness. On the set $\{ \tau_\eps^\delta< \infty \}$, define the random variable $\mathcal{N}_\eps$ by
\[
\mathcal{N}_\eps =  \int_0^{\tau_\eps^\delta} e^{-\lambda_+ s} dW_1 (s).
\]
An application of Duhamel's principle~\eqref{eqn: Duh_Intro1} for $x_\eps^1$ and the definition of $\tau_\eps^\delta$ establishes the equality $\delta = \eps e^{ \lambda_+ \tau_\eps^\delta }\left|\mathcal{N}_\eps \right|$ with probability $1$.
This identity implies that, with probability $1$,
\begin{equation} \label{eqn: time_eqn}
\tau_\eps^\delta = -  \frac{1} {\lambda_+}  \log \eps + \frac{1}{\lambda_+} \log \left( \frac{\delta}{\mathcal{N}_\eps} \right).
\end{equation}
Using equality~\eqref{eqn: time_eqn} together with~\eqref{eqn: Duh_Intro1} and~\eqref{eqn: Duh_Intro2} it holds that
\begin{equation} \label{eqn: x1_equal}
x_\eps^1 ( \tau_\eps^\delta  ) = \delta \sgn \left( \mathcal{N}_\eps\right ),
\end{equation}
and
\begin{align}
x_\eps^2 (\tau_\eps^\delta )& =\eps^{ \lambda_- / \lambda_+ } x_0  \left (  \frac{    \left| \mathcal{N}_\eps \right|}{\delta}\right ) ^{ \lambda_- / \lambda_+ } + \eps \int_0^{\tau_\eps^\delta } e^{ -\lambda_- ( t-s) } dW_2 (s),  \label{eqn: asymmetry_linear}
\end{align}
both with probability $1$. Hence, if we can establish tightness for the distribution of the family of random variables $( \mathcal{N}_\eps )_{\eps >0 }$, the fact that $\Pp \{ \tau_\eps^\delta = \tau_\eps^D\}$ converges to $1$ would be a consequence of~\eqref{eqn: x1_equal},~\eqref{eqn: asymmetry_linear} and the tightness of the distribution of the stochastic integral in~\eqref{eqn: asymmetry_linear}.

In order to get the tightness result, we need to analyze the time $\tau_\eps^\delta$ without any reference to~\eqref{eqn: time_eqn}.
We can show two properties (see Appendix~\ref{ch: app-sec} for their proofs):
\begin{enumerate}
\item For every $\delta>0$, $\tau_\eps^\delta \to \infty $ in probability as $\eps \to 0$.
\item As a consequence to the last point, as $\eps \to 0$,
\[
\mathcal{N}_\eps \stackrel{\Pp}{\longrightarrow} \int_0^\infty e^{-\lambda_+ s} dW_1 (s).
\]
\end{enumerate}
Let $\mathcal{N}$ be the limit Gaussian random variable of $(\mathcal{N}_\eps)_{\eps>0}$ in the second observation above. Then, we have proved the following lemma:
\begin{lemma} \label{lemma: intro_x1_conv}
\[
\tau_\eps^\delta+ \frac{1} {\lambda_+}  \log \eps \stackrel{\Pp}{\longrightarrow} \frac{1}{\lambda_+} \log \left( \frac{\delta}{| \mathcal{N} |} \right ), \quad \eps \to 0.
\]
\end{lemma}
We apply this lemma to the exit distribution of $\x$. Before that, let us denote $\mathcal{N}_-$ a zero mean Gaussian random variable with variance $(2 \lambda_- ) ^{-1}$ independent of $\mathcal{N}$. It is possible to prove that
\[
\int_0^{ \tau_\eps^D } e^{ - \lambda_- (t -s )} dW(s) \to  \mathcal{N}_-, \quad \eps \to 0,
\]
in distribution. This convergence combined with the convergence of $\tau_\eps^\delta$ and $\mathcal{N}_\eps$ used in~\eqref{eqn: x1_equal} and~\eqref{eqn: asymmetry_linear} implies that on the set $\{ \tau_\eps^D = \tau_\eps^\delta \}$,
 \begin{equation} \label{eqn: intro_exit}
\x (\tau_\eps^D ) = \delta ( \sgn \mathcal{N}_\eps, 0 ) + \eps^{(  \lambda_- / \lambda_+ ) \wedge 1 } (0, \xi_\eps).
\end{equation}
Here $(\xi_\eps)_{\eps>0}$ is a family of random variables that satisfies
\begin{equation} \label{eqn: intro_assy}
\xi_\eps \to \left(  \frac{|\mathcal{N}|}{\delta} \right)^{ \lambda_-/\lambda_+ } x_0 \ONE_{\{ \lambda_- \leq \lambda_+ \}} + \mathcal{N}_- \ONE_{ \{ \lambda_- \geq \lambda_+ \}}, \quad \eps \to 0,
\end{equation}
in distribution. Moreover, it can be shown that when $ \lambda_- < \lambda_+$ this convergence holds in probability.

Hence Theorem~\ref{thm: BakhtinSPA} holds with $\gamma (t) = (t,0)$ up to time re-parametrization.

\subsection{Analysis and generalization} \label{subsec: Intro_analysis}

The simplified argument of last section not only recovers, for this simple linear case, Theorem~\ref{thm: BakhtinSPA} but also provides more information. From~\eqref{eqn: asymmetry_linear} we can see that when $\lambda_- \leq \lambda_+ $ the exit distribution has a bias in the stable direction. In this case, the second coordinate of the exit distribution is not centered. Indeed, from~\eqref{eqn: intro_exit} and ~\eqref{eqn: intro_assy} we can identify 3 cases:
\begin{enumerate}
\item When $\lambda_- > \lambda_+$, the second coordinate of $\x$ converges to $\mathcal{N}_-$. Hence, the exit has a centered distribution in the stable direction. We refer to this case as the symmetric case and is illustrated in Figure ~\ref{fig: sym}.
\item When $\lambda_- = \lambda_+$, the second coordinate of $\x$ converges to $\left(  \frac{|\mathcal{N}|}{\delta} \right)^{ \lambda_-/\lambda_+ } x_0 + \mathcal{N}_-$. Hence, the exit has a bias in the stable direction due to the initial condition $x_0$. We refer to this case as the asymmetric case.
\item When $\lambda_- < \lambda_+$, the second coordinate of $\x$ converges to $\left(  \frac{|\mathcal{N}|}{\delta} \right)^{ \lambda_-/\lambda_+ } x_0$. Hence, the exit has a very strong bias in the stable direction due to the initial condition $x_0$. We refer to this case as the strongly asymmetric case and is illustrated in Figure ~\ref{fig: assim}
\end{enumerate}

The consequences of such an asymmetry, when there is one, have been explored in~\cite{nhn}, and it turned out to be an important improvement to Freidlin-Wentzell theory for a particular situation. We will summarize this improvement in Section~\ref{sec: Intro-PlanarHetero}. For now let us comment about how general the argument of last Section~\ref{subsec: sketch} really is.

\begin{figure}
\centering
\includegraphics[width=3in]{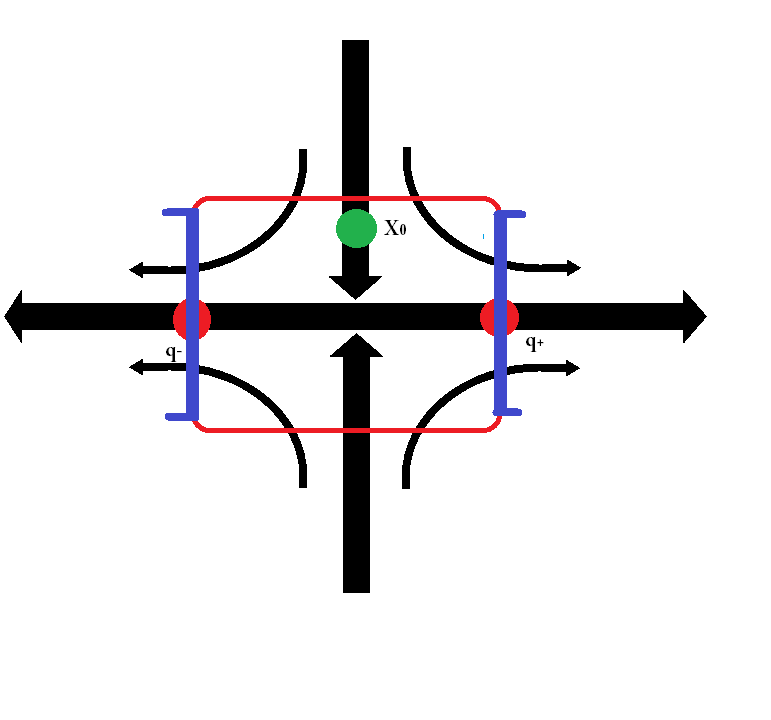}
\caption{Symmetric Case.}
\label{fig: sym}
\end{figure}

The immediate limitation of the argument presented in Section~\ref{subsec: sketch} is that it is mostly based on explicit representations for the solution of $\x$. In~\cite{Bakhtin-SPA},~\cite{nhn} and~\cite{DaySPA} a linear approximation to the original process $\x$ is made. The non-optimal feature of this procedure is that we lose all the identities, and we have just approximations. This is not acceptable if we are interested on computing the properties of the aforementioned asymmetry.

\begin{figure}
\centering
\includegraphics[width=3in]{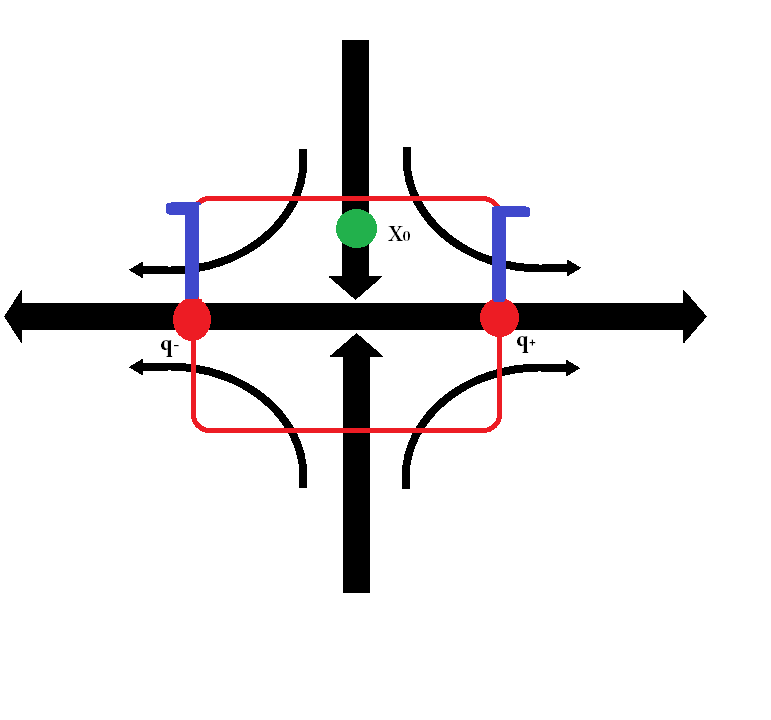}
\caption{ Strongly Asymmetric Case.}
\label{fig: assim}
\end{figure}

A similar but structurally different argument is presented here in Chapter~\ref{ch:  Saddle} ( see~\cite{nhn} also) in which a change of variable is introduced to linearize the system locally. Consider the process $Y_\eps (t) = f  (\x (t) )$ and for the moment assume $f$ to be smooth. Then, the new process $Y_\eps$ solves the SDE:
\[
dY_{\epsilon } = \left( \nabla f^{-1}(Y _{\epsilon }) \right)^{-1} b(f^{-1} (Y_{\epsilon
}))dt+\eps \sgm (Y_{\epsilon })dW +\epsilon ^{2} \Psi (Y_\eps)dt,
\] for some smooth (depending on the function $\sigma$) functions $\Psi: \R^d \to \R^d$ and $\sgm: \R^d \to \R^d \times \R^d$. If we can choose the function $f$ to be such that
\begin{equation} \label{eqn: conjugation_eqn}
\left( \nabla f^{-1}(y) \right)^{-1} b(f^{-1} (y))=Ay,
\end{equation}
then the proof of the theorem will follow almost identically as the case in Section~\ref{subsec: sketch} (only with multiplicative noise instead of additive one) giving the asymmetry for the general case. The existence of such a transformation belongs to the study of conjugation in Ordinary Differential Equations.  The main result of the latter theory (the so called Hartman-Grobman Theorem~\cite[Theorem IX.7.1]{Hartman},~\cite[Section 2.8]{Perko}) guarantees the existence of a homomorphism $f$ that solves ~\eqref{eqn: conjugation_eqn}. This is not enough for our purposes since we need $f$ to be $C^2$ in order to use It\^o's formula.  In~\cite{nhn} is assumed that the transformation $f$ exist and is $C^2$. In order to become aware on how restrictive (if restrictive at all) is this hypothesis, we need to study transformations $f$ that satisfy a relaxed (in a sense explained below) version of~\eqref{eqn: conjugation_eqn}. Such transformations are the main subject of study in normal form Theory~\cite{MR1290117},~\cite{Katok--Hasselblatt},~\cite{Perko},~\cite{Wiggins}. 

We will summarize the main ideas in normal form theory. Following ~\cite{IIlyashenko}, we call a set of complex numbers $\lambda=(\lambda_1, \lambda_2,...,\lambda_d)$ non-resonant if there are no
integral relations between them of the form $\lambda_j=\alpha \cdot \lambda$, where $\alpha =(\alpha _1,...,\alpha_d) \in \mathbb{Z}_{+}^{d}$
is a multi-index with $|\alpha|=\alpha_1 + ...+\alpha_d \geq 2$. Otherwise, we say that it is resonant. Moreover, a resonant $\lambda$ is said to be one-resonant
if all the resonance relations for $\lambda$ follow from a single resonance relation.  A monomial $x^\alpha e_j=x_1^{\alpha _1 }... x_d ^{\alpha _d } e_j$
is called a resonant monomial of order~$R$ if~$\alpha \cdot \lambda = \lambda_j$ and $|\alpha|=R$. Normal form theory asserts (see \cite{MR1290117}, ~\cite{IIlyashenko})
that for any pair of integers $R\geq1$ and $k \geq 1$, there are two neighborhoods of the origin $\Omega_f$ and $\Omega_g$  and a $C^k$-diffeomorphism
$f:\Omega_f\to \Omega_g$ with inverse  $g:\Omega_g\to \Omega_f$
such that
\begin{equation}
\left( \nabla f^{-1} (y) \right)^{-1}b( f^{-1} (y) )=Ay + P(y) + \mathcal{R}(y), \quad y \in \Omega_g \label{eqn: Pre-NormalForm}
\end{equation}
where $P$ is a polynomial containing only resonant monomials of order at most~$R$ and $\mathcal{R}(\zeta)=O(|\zeta|^{R+1})$.
 If $\lambda$ is non-resonant, then $f$ can be chosen so that both $P$ and $\mathcal{R}$ in~\eqref{eqn: Pre-NormalForm} are identically zero.
Moreover, due to~\cite[Theorem~3,Section~2]{IIlyashenko},  if $\lambda$ is one-resonant then $f$ can be chosen so that $\mathcal{R}$ in~\eqref{eqn: Pre-NormalForm} is identically zero.

The result in~\cite{nhn} only includes the non-resonant case. When applied to heteroclinic networks (network of saddles interconnected to each others) this assumption impose a restriction by requiring each critical point to be non-resonant. In particular, typical Hamiltonian systems (that usually present heteroclinic structures) have resonant relations due to the symplectic structure~\cite{MR1290117}.

In this work, we give a complete solution to the $2-$dimensional case with no assumptions about resonance, and with random initial conditions. The theorem (see~\cite{NormalForm} also) informally reads as
\begin{theorem} \label{thm: MioNon}
Let $\x$ be the solution of equation~\eqref{eqn: Ito_additive} with initial condition given by
\[
\x (0) = x_0 + \eps^\alpha \xi_\eps,
\]
where $\alpha \in (0,1)$ and $(\xi_\eps)_{\eps > 0}$ is a family of random vectors that converges weakly to the random vector $\xi_0$. We assume that $x_0 \in \Wc^s$ and $\xi_0$ is such that
\[
\Pp \{ b(x_0) || \xi_0 \}=0,
\]
where $ || $ means collinearity of vectors.

Then, there is a family of random vectors $(\phi _{\epsilon })_{\eps>0}$, a family of random variables $(\psi_{\epsilon })_{\eps>0}$,
and the number
\begin{equation}
\beta=\left\{
\begin{array}{cc}
1, & \alpha \lambda _- \geq \lambda_+ \\
\alpha \frac{\lambda_-}{\lambda _+}, & \alpha \lambda _- <\lambda _+ \\
\end{array}%
\right.
\end{equation}
such that%
\begin{equation*}
X_\eps (\tau _\eps^{D})=q_{\sgn(\psi_{\epsilon })}+\epsilon
^{\beta}\phi _{\epsilon }.
\end{equation*}%
The random vector
\[\Theta_\eps=\left(\psi_\eps, \phi_\eps, \tau _{\epsilon }^{D}+\frac{\alpha }{\lambda _{+}}\ln \epsilon\right)\]
converges in distribution as $\epsilon\to 0$ to the limit $\Theta_0$ that can be identified. The exit distribution exhibits the behavior presented at the beginning of Section~\ref{subsec: Intro_analysis}.
\end{theorem}

The proof of this theorem is divided in two steps. One, is the study of the diffusion $\x$ when is close to the origin. The second, is the study of the diffusion $\x$ far from the origin. Here, the meaning of close and far relies on whether or not the system can be conjugated to its Normal Form.

By the study of $\x$ close to the critical point, we mean $\x \in B$, where $B$ is a neighborhood of the origin in which normal form conjugation is valid. In this case, the analysis has two parts. The first part is when the diffusion starts along the stable direction. To study this part, we study the diffusion until the the projection of $\x$ along the stable direction dominates over the noise level. We achieve this by posing the problem as an exit problem from the strip $[-\eps^{\bar{\alpha} }, \eps^{ \bar{\alpha} }] \times B$, for some $\bar{\alpha} \in (0, \alpha)$. For the second part we study the exit problem from $B$ with the initial condition being the exit distribution obtained in the first part. In order for this program to be successful, we require very precise path-wise expressions for the diffusion. By using the general form of the resonances we are able to preserve the essence of the argument in Section~\ref{subsec: sketch}.

The analysis of the system far from the origin is used twice. The first time is when the system starts along the stable manifold $\Wc^s$.  The second time is when $\x$ is about to exit the domain and it probably has some bias. Our study is based on a series expansion in powers of $\eps>0$. This expansion is inspired by~\cite{Blagoveschenskii:MR0139204}, but it requires additional geometric arguments. These results are of independent interest, so we start a new section to describe them.

\section{Levinson Case.} \label{sec: unstable}

This section is devoted to the Levinson case. We will first review the history of the problem and then outline our contribution.

Given an initial condition $x_0 \in D$, Levinson condition is a hypothesis associated to the flow $S$ and the domain $D$. This case was originally formulated in~\cite{Levinson:MR0033433} with a PDE flavor, we state the condition as presented by Freidlin~\cite[Chapter 2]{Freidlin-lectures:MR1399081}.

\begin{condition}[Levinsion] \label{cond: levinson}
The flow $S$ satisfies the Levinson condition at $x_0 \in D$ with respect to $D$ if the following holds:
\begin{enumerate}
 \item The exit time
 \[
T(x_0)= \inf \{ t>0: S^t x_0 \in \partial D \} ,
\]
is finite.
 \item The flow $S^tx_0$ leaves the domain immediately after $T(x)$. That is, there is a $\delta>0$ such that $S^{T(x_0) + s} \not \in D \cup \partial D$ for all $s \in (0,\delta)$.
\end{enumerate}
We say that the domain $D\subset \R^d$ satisfies Levinson condition if properties 1 and 2 are satisfied at each $x\in D$.
\end{condition}
As mentioned in the introduction, Levinson~\cite{Levinson:MR0033433},~\cite{LevinsonLectures} was originally interested in studying the behavior of the solution of the PDE:
\begin{align}
\nabla u_\eps (x)  \cdot b (x) + \frac{ \eps^2 }{2} \Delta u_\eps (x)&=0, \quad x \in D , \notag \\
 u_\eps (x)&=g (x), \quad x \in \partial D. \label{eqn: intro-Lev-PDE}
\end{align}
The claim is that the solution to this PDE has to converge to the solution of the unperturbed PDE:
\begin{theorem}[~\cite{Levinson:MR0033433}]\label{thm: intro-lev-levinson}
Under the Levinson condition~\ref{cond: levinson}, there is a unique (maybe generalized) solution to both, the perturbed problem~\eqref{eqn: intro-Lev-PDE} and the unperturbed problem
\begin{align*}\
\nabla u_0 (x)  \cdot b (x)&=0, \quad x \in D ,\\
 u_0 (x)&=g (x), \quad x \in \partial D.
\end{align*}
Let $u_0: \R^d \to \R^d$ be a solution to the unperturbed problem. Then, if $g:\partial D \to \R^d$ is smooth,
\[
u_\eps ( x )= u_0 (x) + \eps v_\eps (x),
\]
where $v_\eps(x)$ is a locally bounded function for each $\eps>0$.
\end{theorem}
The proof of this theorem is based on a series expansion of $u_\eps$ along the characteristics of~\eqref{eqn: intro-Lev-PDE}. In order to prove the result with this idea, several technical challenges had to be overcome in~\cite{Levinson:MR0033433}. Contrastingly, as pointed out in~\cite{Freidlin-lectures:MR1399081}, the probabilistic approach here is simpler and cleaner. Indeed, once we know inequality~\eqref{eqn: estimate_basic}, the convergence in Theorem~\ref{thm: intro-lev-levinson} is an immediate consequence of the stochastic representation $u_\eps(x)=\mathbf{E}_{x} g ( \x ( \tau_\eps^D) ) $, where $\x$ solves~\eqref{eqn: Ito_additive} with $\sigma=\rm{Id}$. To get the exact behavior a series expansion of the processes $\x$ has to be made as in~\cite{Blagoveschenskii:MR0139204} and~\cite[Chapter 2]{Freidlin--Wentzell-book}.

In this thesis we develop a path-wise approach to this problem. We give a geometrical characterization of the exit point $\x ( \tau_\eps)$, and joint properties of $(\x ( \tau_\eps),\tau_\eps ) $ are obtained. We start by obtaining a generalization of the series expansion given in~\cite{Blagoveschenskii:MR0139204}. This result serves as backbone of our proof and has independent interest on itself.

In order to present our result, we need further notation. Write $b$ as
\begin{equation*}
b(x)=b(y)+\nabla b(y)(x-y)+Q_1(y,x-y),\quad x,y\in\R^d, 
\end{equation*}
where
\begin{equation*}
|Q_1(u,v)|\leq K  |v|^2 
\end{equation*}
 for some constant $K>0$ and any $u,v \in \R^d$.

Denote by $\Phi_x(t)$ the linearization of $S$ along the
orbit of $x$:
\begin{equation}
\frac{d}{dt}\Phi_x(t)=A(t)\Phi_x(t)\text{, \ }\Phi_x(0)=I, \label{eqn: intro-lev-Phi_def}
\end{equation}
where $A(t)=\nabla b(S^tx)$ and $I$ is the identity matrix. We can state our first lemma:
\begin{lemma} \label{lemma: important}
Consider the initial value problem
\begin{align}
\label{eqn: intro-Mod1}
dX_\eps(t)& = \left( b( X_\eps(t) )+\eps^{\alpha_1} \Psi_\eps(X_\eps(t)) \right)dt + \eps \sigma(X_\eps(t))dW\\
\label{eqn: intro-Mod2}
X_{\epsilon }(0) & =x_{0}+\epsilon ^{\alpha_2 }\xi _{\epsilon },\quad\eps>0.
\end{align}
where, for each $\eps$, $\Psi_\eps$ is a deterministic Lipschitz vector field on $\R^d$  converging uniformly to
a limiting Lipschitz vector field $\Psi_0$. Both $\alpha_1$ and $\alpha_2$ are positive scaling exponents. The family of random variables $(\xi_\eps)_{\eps>0}$ converges, as before, to $\xi_0$ in distribution as $\eps \to 0$.

Let
\begin{align*}
\phi_\eps (t) &=\eps^{\alpha_2-\alpha}\Phi_{x_0}(t) \xi_\eps+  \eps^{\alpha_1-\alpha}\Phi_{x_0}(t) \int_0^t \Phi_{x_0} (s)^{-1} \Psi_0( S^s x_0) ds  \\
&\quad + \eps^{1-\alpha} \Phi_{x_0}(t) \int_0^t \Phi_{x_0} (s)^{-1}\sigma( S^s x_0) dW(s),
\end{align*} and
\begin{align*}
\phi_0 (t) &= \mathbf{1}_{\{\alpha_2= \alpha\}}  \Phi_{x_0}(t) \xi_0 + \mathbf{1}_{\{\alpha_1= \alpha\}}\Phi_{x_0}(t) \int_0^t \Phi_{x_0} (s)^{-1} \Psi_0( S^s x) ds \\
& \quad +\mathbf{1}_{\{1=\alpha\}} \Phi_{x_0}(t)\int_{0}^{t}\Phi_{x_0}^{-1}(s)\sigma(S^s x_0)dW(s), \quad t>0.
\end{align*}
Then,
\[
X_\eps(t)=S^tx_0 + \eps^\alpha \varphi_\eps (t)
\]
holds almost surely for every $t>0$,
where $\varphi_\eps(t)=\phi_\eps(t)+ r_\eps(t)$, and $r_\eps$ converges to 0 uniformly over compact time intervals in probability. Moreover, for any $T>0$, $\phi_\eps \to \phi_0$,
in distribution in $C[0,T]$ equipped with uniform norm.
\end{lemma}

The reason to consider~\eqref{eqn: intro-Mod1} instead of the more standard~\eqref{eqn: Ito_additive} will become evident in Section~\ref{sec: 1d-time-reversal}. For now, let us observe that, although inequality~\ref{eqn: estimate_basic} may not hold, it is still true that
\[
\sup_{ t \leq T} | \x (t) - S^t x_0 | \stackrel{ \Pp } { \longrightarrow} 0, \quad \eps \to 0.
\]
Hence, under the Levinson Condition~\ref{cond: levinson} the exit of $\x$ from $D$ will occur on a finite (still random) time and very close to the deterministic exit $z=S^{T(x_0)}x_0$. A better understanding of this convergence is the main result in this section.

The main theorem provides a scaling limit to the distribution of $(\tau_\eps, \x ( \tau_\eps) )$. In order to understand the theorem, we regard the boundary of $D$ as an hypersurface $M$ in $\R^d$. In general, for an hypersurface $M$ in $\R^d$, denote the tangent space of $M$ at the point $z \in M$ as $T_z M$. Further, we denote the (algebraic) projection onto $\Span(b(z))$ as $\pi_b:\R^d \to \R $ , and the (geometric) projection onto $T_zM$ along $\Span(b(z))$ as $\pi_M :\R^d \to T_zM$.  In other words, for any vector $v\in \R^d$, $\pi_b v\in \R$ and $\pi_M v\in T_zM$ are the unique number and vector such that
\begin{equation} \label{eqn: intro-lev}
 v=\pi_b v \cdot b(z)+ \pi_M v.
\end{equation}
With this notation in mind, we are ready to state the theorem:

\begin{theorem} \label{thm: Intro-Main-Levinso} Let $M$ be an hypersurface in $\R^d$. Let $\x$ be the solution of~\eqref{eqn: intro-Mod1} with initial condition~\eqref{eqn: intro-Mod2}. Consider $\tau_\eps$ and $T(x_0)$ the exit time from $M$ of $\x$ and $S$ respectively. If $\alpha=\alpha_1 \wedge \alpha_2\wedge 1$ and $z=S^{T(x_0)} x_0$, then
\begin{equation}
 \eps^{-\alpha}(\tau_\eps-T, X_\eps(\tau_\eps)-z) \to (-\pi_b \phi_0(T), \pi_M \phi_0(T)),
\end{equation}
in distribution. Here $\pi_b$ and $\pi_M$ are as in~\eqref{eqn: intro-lev}.
\end{theorem}

\section{Applications} \label{sec: intro-app}

\subsection{Conditioned diffusions in 1 dimension} \label{sec: 1d-time-reversal}

Throughout this section, we restrict ourselves to the $1$-dimensional situation. In particular, let, for each $\eps>0$, $X_\eps$ be a weak solution of the following ($1$ dimensional) SDE:
\begin{align*}
 dX_\eps(t)&=b(X_\eps(t))dt+\eps\sigma(X_\eps(t)) dW(t),\\
 X_\eps(0)&=x_0,
\end{align*}
where $b$ and $\sigma$ are $C^1$ functions on $\R$, such that $b(x)<0$ and $\sigma(x)\ne 0$ for all $x$ in an interval $[a_1,a_2]$
containing $x_0$. We want to study the exit of such an interval $D=[a_1,a_2]$, that is
\[
 \tau_\eps=\inf\{t\ge0:\ X_\eps(t)=a_1\ \text{\rm or}\ a_2\}.
\]
Let $B_\eps=\{X_\eps(\tau_\eps)=a_2\}$, and note that, since $b<0$, $\lim_{\eps\to0}\Pp(B_\eps)=0$. More precise estimates
on the asymptotic behavior of $\Pp(B_\eps)$ can be obtained in terms of large deviations. However, our interest is to study the process
$X_\eps$ conditioned on the rare event $B_\eps$.

In this case $T(x_0)$, the time it takes for the flow $S$ generated by $-b$ starting at $x_0$ to reach $a_2$, is given by
\[
T(x_0)=-\int_{x_0}^{a_2}\frac{1}{b(x)}dx.
\]
It is known from our basic Lemma~\ref{lemma: important} that $\tau_\eps \to T(x_0)$ as $\eps \to 0$ in probability. But the correction was not known so far.

The idea is to condition the diffusion to the event $B_\eps$ and note that this conditioned process solves a martingale problem (hence is a diffusion) and the result from Section~\ref{sec: unstable} are applicable. Hence, we have the lemma:

\begin{lemma} Conditioned on $B_\eps$, the process $X_\eps$ is a diffusion with the same
diffusion coefficient as the unconditioned process, and with the
drift coefficient given by
\[
b_\eps(x)= b(x)+\eps^2\sigma^2(x)\frac{h_\eps(x)}{\int_{a_1}^x h_\eps(y)dy},
\]
where
\begin{equation*}
 h_\eps(x)=\exp\left\{-\frac{2}{\eps^2}\int_{a_1}^x \frac{b(y)}{\sigma^2(y)}dy\right\}.
\end{equation*}
\end{lemma}

With this lemma and the help of an analogy of Laplace's method the main theorem in this direction is:

\begin{theorem} Conditioned on $B_\eps$, the distribution of $\eps^{-1}(\tau_\eps-T(x_0))$ converges weakly to a centered
Gaussian
distribution with variance
\[
- \int_{x_0}^{a_2} \frac{ \sigma^2 (y) }{ b^3 (y) } dy.
\]
\end{theorem}

The result is of relevance not only because of the correction itself, but also, because is the first step of analysis for diffusions conditioned on rare events. Such a tool may lead to a general theory of correction in small noise systems.

\subsection{Planar Heteroclinic Networks} \label{sec: Intro-PlanarHetero}

A further application of our results is to the theory of Noisy Heteroclinic Networks first proposed in~\cite{nhn}. Our presentation applies only to the $2$-dimensional situation. See~\cite{BakhtinDS} for a survey in this direction.

In Section~\ref{sec: app-hetero-intuitive} we give an intuitive presentation of the argument in~\cite{nhn}. The general theory developed in~\cite{nhn} is presented in Section~\ref{sec:Hetero-result}. In subsection~\ref{sec: Hetero_Current} relations of this result to the current text are highlighted. In this section, we also introduce the idea of a random Poincar\'e map.

\subsubsection{Intuitive argument}\label{sec: app-hetero-intuitive}
We study the exit problem of the diffusion~\eqref{eqn: Ito_additive} from a domain $D$. Consider the vector field $b:\R^2 \to \R^2$ which has finite set of critical points $\mathcal{Z}= (\zeta_k)_{k=1}^N$ inside $\bar{D}$.  We assume that $S$ admits an heteroclinic network.

A heteroclinic network for the flow $S$ is an invariant that contains at most countable number of saddles connected with each other. For simplicity, in this section we suppose that $S$ admits an heteroclinic network with a finite set of critical points $\mathcal{Z}= (\zeta_k)_{k=1}^N$ inside $\bar{D}$. Precisely, we assume the following:
\begin{enumerate}
\item Each critical point $\zeta_k$ is a saddle point of the flow $S$. That is, $b(\zeta_k)=0$ and the matrix $A_k = \nabla b ( \zeta_k)$ has two eigenvalues: $\lambda_k^+>0$ and $-\lambda_k^-<0$.
\item The flow $S$ generated by $b$ admits an heteroclinic structure in $\bar{D}$. We give the technical description of this assumption. For each critical point $z_k \in \mathcal{Z}$, let $\mathcal{W}_k^s$ be the $1$-dimensional locally stable manifold and $\mathcal{W}_k^u$ the $1$-dimensional locally unstable manifold. Take a $\delta > 0$ small enough so that normal form conjugation and Hadamard--Perron invariant manifold theorem holds for a ball $B_k=B_\delta( \zeta_k )$ of radius $\delta>0$ centered at each critical point . Denote $\{q_k^+, q_k^- \} =\mathcal{W}_k^u \cap B_\delta ( \zeta_k )$. The hypothesis that $b$ admits an heteroclinic structure means that for each integer $1 \leq k \leq N$, there is an integer $n^\pm_k \in \{ 1,...,N\}$ such that
\[
\lim_{t \to \infty } S^t q_k^ \pm = \zeta_{n^\pm_k}.
\]
\item All non degeneracy assumption made in Section~\ref{sec: saddle} hold for each critical point.
\end{enumerate}

Suppose the starting point for the diffusion $\x(0)$ is a deterministic point in $\mathcal{W}_1^s$.  Theorem~\ref{thm: MioNon} implies that with high probability the diffusion will exit $B_1$ approximately along
$q_1^+$ or $q_1^-$ with equal probability. Moreover, Theorem~\ref{thm: MioNon} tells us how to compute the scaling exponent of the additive exit correction term and the asymptotic distribution of this correction term.

For this discussion, we suppose that the diffusion exits $B_1$ asymptotically close to $q_1^+$. The exit from $B_1$ will now be the initial condition in Lemma~\ref{lemma: important}. Applying this lemma (with $\Psi_\eps \equiv 0$) for sufficiently large $T$, we can derive the asymptotic representation of the entrance distribution for $B_{n_1^+}$, which satisfies the properties imposed to the initial condition in Theorem~\ref{thm: MioNon}. Observe that Lemma~\ref{lemma: important} also implies that the scaling exponent of the additive correction in the entrance distribution for $B_{n_1^+}$ is the same as in the exit distribution for $B_1$. Moreover, this lemma implies that the asymptotic distribution of the additive correction term has in the entrance to $B_{n_1^+}$ is the evolved (under the linearization of $S$) version of the asymptotic distribution of the correction term at the exit from $B_1$. In particular, any bias on the exit of $B_1$ gets translated, by the linearization of the flow, to the entrance of $B_{n_1^+}$. Let $i_1=n_1^+$. Then Theorem~\ref{thm: MioNon} applies again to derive that asymptotically the exit distribution from $B_{n_1^+}$ is concentrated mostly along $q_{i_1^+}$ or $q_{i_1^-}$, but with possible unequal probability. We can proceed like this iteratively along any sequence of saddle points $z_1, z_{i_1},...,z_{i_r}$ such that for any $j$, $i_{j+1}=n^+ (i _j)$ or $i_{j+1}=n^- (i _j)$.

The result of this procedure allows us to conclude that the system evolves in a Markov fashion (choosing the next saddle with probability $1/2$ independently of the history of the process) until it meets a saddle point at which the exit distribution becomes asymmetric. After that the choice of the two heteroclinic connections is not Markov anymore. The choice of the two heteroclinic connections may become Markov again if the system meets a saddle in which the symmetry is reestablished. We will illustrate how this phenomenon affects the exit distribution.

\begin{figure}[lineheight]
\centering
\includegraphics[width=0.9\textwidth]{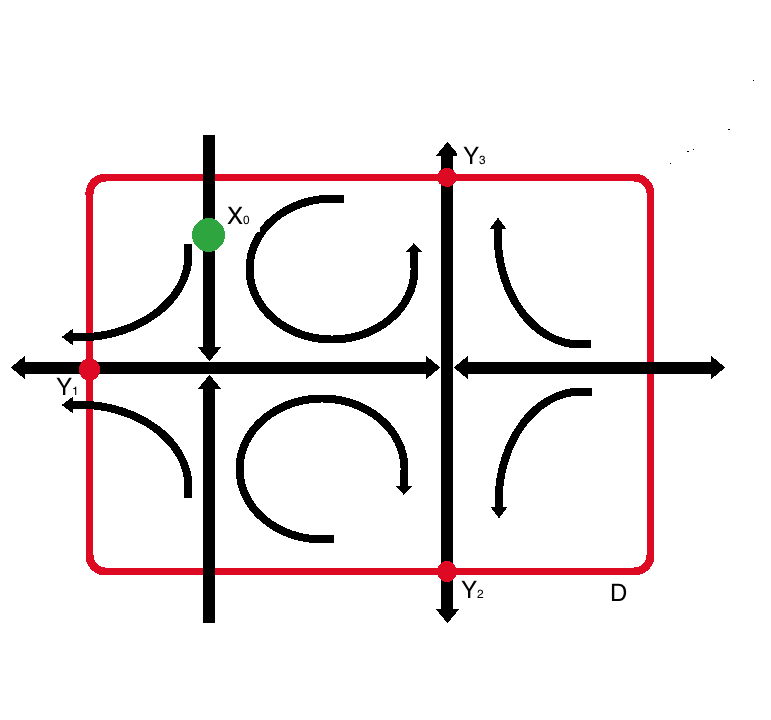}
\caption{Heteroclinic network example}
\label{fig: heteroclinicEx}
\end{figure}

\subsubsection{Planar heteroclinic network with two nodes} \label{sec: two-nodes}

Let us give a concrete example. Consider the system in which $b$ has two critical points $\{\zeta_1,\zeta_2\}$ and phase space of the flow $S$ is as depicted in Figure~\ref{fig: heteroclinicEx}, with $D$ being a rectangle around the two critical points. Let $\x (0)=x_0$ be on the locally stable manifold $\mathcal{W}^s_1$ of $\zeta_1$. Consider $\{q_k^+, q_k^- \} =\mathcal{W}_k^u \cap B_\delta ( \zeta_k )$, where $B_\delta ( \zeta_k )$ is a ball of sufficiently small radius around $\zeta_k$ and $\mathcal{W}_k^u$ the locally unstable manifold of $\zeta_k$. There is an orbit connecting $\zeta_1$ with $\zeta_2$:
\[
\lim_{ t\ \to \infty} S^t q_1^- = \zeta_2.
\]
Recall from Section~\ref{sec: intro_FW} that as $\eps \to $ the exit distribution concentrates near the minimizers of $V(x,\cdot)$ over the boundary of $\partial D$. For a heteroclinic network this means that the exit concentrates at all points in the boundary that can be reached from $x$ along a sequence of heteroclinic connections. In this case, these points are 
\begin{align*}
y_1 =\lim_{ t\ \to \infty} S^t q_1^+, \quad
y_2 =\lim_{ t\ \to \infty} S^t q_2^+,
\end{align*}
and $y_3 =\lim_{ t \to \infty} S^t q_2^-$. Then, the exit measure will weakly converge, as $\eps \to 0$, to
\begin{equation} \label{eqn: hetero-limit-distr}
p_1 \delta_{y_1}+ p_2 \delta_{y_2} + p_3 \delta_{y_3},
\end{equation}
where $p_1,p_2$ and $p_3$ are positive numbers that sum up to $1$. In the spirit of Theorem~\ref{thm: MioNon} a direct application of our results imply a scaling limit to this convergence. 
\begin{lemma}
Consider the system just described. There is a family of random variables $(\theta_\eps)_{\eps>0}$, a family of random vectors $( \phi_\eps)_{\eps>0}$ and the random variable $\alpha_0 $ such that $\Pp \{ \theta_\eps \in \{1,2,3\} \}=\Pp \{ \alpha_0 \in (0,1] \}=1$, and
\[
\x ( \tau_\eps^D ) = y_{\theta_\eps} + \eps^{\alpha_0} \phi_\eps,
\]
for every $\eps>0$. The random vector $(\theta_\eps,\phi_\eps)$ converges in distribution to $(\theta_0, \phi_0)$. 
\end{lemma}
The random vector $(\theta_0, \phi_0)$ in principle can be obtained explicitly. It is clear from~\eqref{eqn: hetero-limit-distr} that $p_i=\Pp \{ \theta_0 = i\}$, $i=1,2,3$.

Let $\{ \lambda_+, - \lambda_- \}$ be the set of eigenvalues of $\nabla b ( \zeta_1 )$ and $\{ \mu_+, - \mu_- \}$ the set of eigenvalues of $\nabla b ( \zeta_2 )$ (see Figure~\ref{fig: heteroclinicEx}). Using the three observations made at the beginning of Section~\ref{subsec: Intro_analysis} several cases can be considered (see~\cite{BakhtinDS} for further explanation):
\begin{itemize}
\item If $\mu_+ < \mu_-$ , and $ \lambda_+ < \lambda_-$ then the system is symmetric with $p_1=1/2$, $p_2=p_3=1/4$, and $\alpha_0=1 $. Here symmetric means that the random vector $\phi_0$ has no bias along the direction of $\mathcal{W}_1^s$ if $\theta_0=1$ or along the direction $\mathcal{W}^u_1$ if $\theta_0 \in \{2,3\}$. Asymmetric means that there is a bias in any of the aforementioned cases.
\item If $\mu_+ \lambda_+ < \mu_- \lambda_-$, and $\lambda_+ \geq \lambda_-$, the system is symmetric if $\theta_0 \in \{ 2,3\}$, strongly asymmetric if $\lambda_+ < \lambda_-$ and $\theta_0=1$, and asymmetric if $\lambda_+ = \lambda_-$ and $\theta_0=1$. Moreover, $p_1=1/2$, $p_2=0,p_3=1/2$, when $\lambda_- < \lambda_+$, and $p_1=1/2, p_2=0, p_3=1/2$, when $\lambda_- = \lambda_+$. The random variable $\alpha_0$ is given by 
\[
\alpha_0= \frac{\lambda_-}{\lambda_+ } \delta_{ \{1\} } ( \theta_0 ) +   \delta_{ \{2,3\} } ( \theta_0 ).
\]
\item If $\mu_+  \lambda_+ > \mu_- \lambda_-$, and $\lambda_+ > \lambda_-$, the system is strongly asymmetric and $p_1=1/2$, $p_2=0, p_3=1/2$, and 
\[
\alpha_0=\frac{\lambda_-}{\lambda_+ } \delta_{ \{1\} } ( \theta_0 ) +   \mu_- \lambda_- / (\mu_+ \lambda_+) \delta_{ \{2,3\} } ( \theta_0 ) .
\]
\item If $\mu_+ = \mu_- $ and $\lambda_+ = \lambda_-$, the system is asymmetric and $p_1=1/2$, $p_2\in (0,p_3), p_3<1/2$, and $\alpha_0=1$.
\item If $\mu_+  > \mu_- $, and $\lambda_+ = \lambda_-$, the system is asymmetric if $\theta_0=1$ and strongly asymmetric otherwise. Moreover, $p_1=1/2$, $p_2\in (0,p_3), p_3<1/2$, and 
\[
\alpha_0=\delta_{ \{1\} } ( \theta_0 ) + ( \mu_- / \mu_+ ) \delta_{ \{2,3\} } ( \theta_0 ). 
\] 
\item If $\mu_+  = \mu_- $, and $\lambda_+ > \lambda_-$, the system is strongly asymmetric if $\theta_0 =1$, and asymmetric otherwise. Moreover, $p_1=1/2$, $p_2=0, p_3=1/2$, and $\alpha_0=\lambda_- / \lambda_+$.
\end{itemize}

A formalization of this argument based on a weak convergence result is done in~\cite{nhn}. In such, the limiting behavior of the rescaled process
\[
Z_\eps(t) = \x ( t \log ( \eps^{-1} ) )
\]
is obtained. Notice how this rescaled process instantaneously jumps along saddles. Hence if a weak convergence result has to be established, we need to introduce a new topology. Indeed, the standard Skorokhod topology does not allow to capture the curves along which the jumps are made. We state the weak convergence result in the next section.

\subsubsection{Weak convergence result}\label{sec:Hetero-result}
In order to present the weak convergence result for the rescaled version of $\x$, we need to introduce a new topology.

Consider all paths $\gamma:[0,1] \to [0,\infty) \times \R^2$ such that the first coordinate $\gamma^0$ is nondecreasing. Equip the space of paths with the equivalence relation $\sim$, where $\gamma_1 \sim \gamma_2$ if and only if there is a path $\gamma^*$ and non-decreasing surjective functions $\lambda_1,\lambda_2 : [0,1] \to [0,1]$ such that $\gamma_i = \gamma^* \circ \lambda_i$. The set of curves $\mathbf{X}$ is the quotient of the space of paths with the equivalence relation $\sim$. Actually the set $\mathbf{X}$ can be regarded as a Polish space:
\begin{lemma}[~\cite{nhn}]
$\mathbf{X}$ can be made into a metric Polish space with distance function
\[
 \rho ( \Gamma_1, \Gamma_2 ) = \inf_{ \gamma_1 \in \Gamma_1 , \gamma_2 \in \Gamma_2 } \sup_{ s \in [0,1] } | \gamma_1 (s) - \gamma_2(s) |.
\]
\end{lemma}
Refer to~\cite{nhn} for more information about this sapce.

In order to state the result, we give a non-technical introduction to the notion of entrance-exit maps introduced in~\cite{nhn}. Let $\mathcal{P}$ be the set of probability measures in $\R^2$, define $\rm{out}=(0,\infty) \times [0,1] \times \R^2 \times (0,1] \times \mathcal{P}$ and
\begin{align*}
\rm{Out}_k=\{ \left( (t_-,p_-,x_-,\beta_-,F_-),(t_+,p_+,x_+,\beta_+,F_+)  \right) \in \rm{out}^2: \\
  t_-=t_+ , x_\pm=q_k^\pm, p_-+p_+=1, \beta_-=\beta_+ \}.
\end{align*}
Then we have the following definitions.
\begin{definition}
For each $k$, an entrance-exit map is a map \
\[
\Psi_k : \{q_k^+,q_k^-\} \times (0,1] \times \mathcal{P} \to \rm{ Out}_k,
\]
where the domain of $\Psi_k$ satisfies some regularity assumptions(see~\cite[page 10]{nhn}~)~. We denote $\Psi_k=( \Psi_k ^+, \Psi_k ^- )$.
\end{definition}

\begin{definition}
Suppose $x_0 \in \mathcal{W}^s_k$, for some $1 \leq k \leq N$. The sequence $\mathbf{z}=(\theta_0, z_{i_1},...,\theta_{r-1},z_{i_r},\theta_r)$ is admissible for $x_0$ (refered as $x_0$-admissible) if
\begin{enumerate}
 \item $\theta_0$ is the orbit of $x_0$ with $S^t x_0 \to z_{i_1} $, as $t \to \infty$;
 \item for each $j\in\{1,..,r \}$,  $\theta_j$ is either the orbit of $q_{i_j}^+$ or the orbit $q_{i_j}^-$;
 \item for each $j\in\{1,..,r \}$, $i_{j+1}$ is either $n_{i_j}^+$ or $n_{i_j}^-$ according to whether $theta_j$ is the  orbit of $q_{i_j}^+$ or $q_{i_j}^-$.
\end{enumerate}
\end{definition}
With each admissible sequence $\mathbf{z}$ we associate the sequence
\[
 \eta( \mathbf{z} ) = ( ( x_0',\alpha_0,\mu_0), ( t_1,p_1,x_1,\alpha_1,\mu_1),...,(t_r,p_r,x_r,\alpha_r,\mu_r) ),
\]
where $x_0'=S^{t'(x_0) }x_0, \alpha_0=1$,
\[
 t'(x_0)=\inf \{ t>0: S^t x_0 \in B_\delta(\zeta_{i_1} ) \},
\]
and the rest of the entries are given by
\[
(t_j,p_j,x_j,\alpha_j,\mu_j)=\left \{ \begin{tabular} { l r }
                                        $\Psi_{i_j}^+ ( x_{j-1},\alpha_{j-1},\mu_{j-1})$, & $i_j=n_{i_j}^+$ \\
					$\Psi_{i_j}^- ( x_{j-1},\alpha_{j-1},\mu_{j-1})$, & $i_j=n_{i_j}^-$ \\
                                      \end{tabular}
 \right. .
\]

To each admissible sequence $\mathbf{z}$ we can associate a piecewise constant curve $\Gamma ( \mathbf{z} )$ by identifying it with the path of curves such that spend time $t_j$ at the point $x_j$ and jump to the next point along the path $\theta_j$. Also we can associate probabilities through the relationship
\[
 \pi ( \mathbf{z} ) = p_1 ... p_r.
\]

Note how the set of all admissible sequences for $x_0$ has the structure of a binary tree partially ordered by inclusion. We say that a set of admissible sequences $L$ of $x_0$ is free if no two sequences of $L$ are comparable with respect to this partial order. Additionally, if any sequence not in $L$ is comparable to one sequence from $L$ then $L$ is called complete. It is clear that for any free set $\pi ( L) := \sum_{\mathbf{z} \in L } \pi ( \mathbf{z} ) \leq 1 $, while for a complete set $\pi ( L) =1$.

The main theorem is:
\begin{theorem} \label{thm: nhn}
Suppose that $\x (0)=x_0$ is in the heteroclinic invariant. For each $\eps > 0$, define the process $Z_\eps(t) = \x ( t | \log ( \eps ) | )$. Then, for any conservative set $L$ of $x_0$-admissible sequences, there is a family of stopping times $(T_\eps)_{\eps>0}$ such that the distribution of the graph $\Gamma_{ Z_\eps(t), t < T_\eps }$ converges weakly in $( \mathbb{X}, \rho)$ to the measure $M_{x_0, L}$ concentrated on the set
\[
\{ \Gamma ( \mathbf{z} ): \mathbf{z} \in L \}
\]
 and satisfying $M_{\x (0), L } \{ \Gamma ( \mathbf{z} ) \} = \pi( \mathbf{z} ) $.
\end{theorem}

\subsubsection{Contributions made in the case $S$ admits a heteroclinic network} \label{sec: Hetero_Current}
In this section we outline our contribution for the case in which $S$ admits an heteroclinic network.

The iteration procedure described in Section~\ref{sec: app-hetero-intuitive} was first proposed in~\cite{nhn}. It is the central idea in proving the main results in~\cite{nhn}. This iteration is carried out in~\cite{nhn} by using an equivalent version of Theorem~\ref{thm: MioNon}. This version was proved under the hypothesis that the non-linear system can be locally conjugated to a linear system by a $C^2$ transformation. That is, in~\cite{nhn} Theorem~\ref{thm: nhn} is proved under the following hypothesis:
\begin{condition} \label{con: intro_C2}
At each critical point $\zeta \in \mathcal{Z}$ there are non-resonant conditions.
\end{condition} 
As discussed in Section~\ref{sec: saddle} this is in general not the case, and examples of saddle point that do not satisfy this condition are known~\cite{MeyerExamples}.  
In this work, we completely remove condition~\ref{con: intro_C2} in the $2$-dimensional situation.

On the other hand, observe that the iteration procedure is Section~\ref{sec: app-hetero-intuitive} is based on the computation of a random map. This map is such that, for a domain $V$, to any given initial distribution of the diffusion $\x$, it gives the exit distribution of $\x$ from $V$. We call this map a random Poincar\'e map. For $V \subset \R^2$, let $\Pi_{V}$ be the set of probability measures with support on $V$. Then, the random Poincar\'e map for $D$, $\Upsilon_D: \Pi_{D} \to \Pi_{\partial D}$ is the (deterministic) map such that $\Upsilon_D \mathbf{Q}=\Pp_{\mathbf{Q}} \{ \x ( \tau_\eps^D \}$, where $\Pp_{\mathbf{Q}}$ is the original probability measure conditioned on $\x(0)$ being distributed as $\mathbf{Q}$.  The iteration in Section~\ref{sec: app-hetero-intuitive} illustrates the use of this map. Notice that this methodology applies regardless the type of equilibria that the system exhibits. We base our proof of Theorems~\ref{thm: MioNon} and~\ref{thm: Intro-Main-Levinso} on a similar idea. Hence, it is worth to study small noise perturbations with this direction in mind. 

As an example of a Poincar\'e map, consider our example in Section~\ref{sec: two-nodes}. The exit distribution is the composition of the Poincar\'e maps $\Upsilon_{D_5} \circ ...\circ \Upsilon_{D_1} \delta_{x_0}$, where $D_i$ are ilustrated in Figures~\ref{fig:animals} and~\ref{fig:animals3}, and we are conditioning on exit along $y_3$.

\begin{figure}
  \centering
 \includegraphics[width=1.0\textwidth]{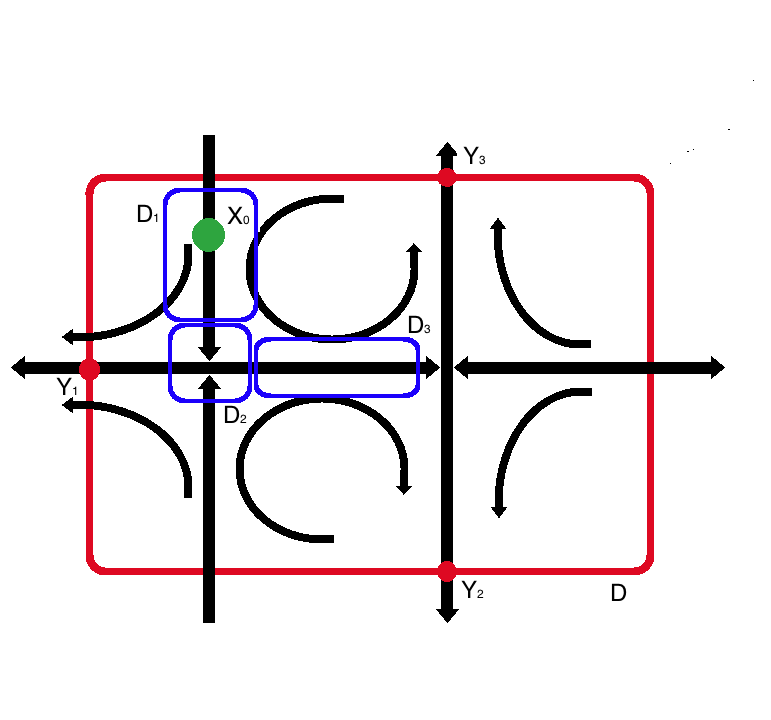} \\         
  \caption{Illustration of the domains $D_1$, $D_2$, and $D_3$ used to compute the Poincar\'e maps in the case of a heteroclinic network with 2 nodes conditioned on exit along $y_3$: escape from the first saddle.}
  \label{fig:animals}
\end{figure}

\begin{figure}
  \centering
  \includegraphics[width=1.0\textwidth]{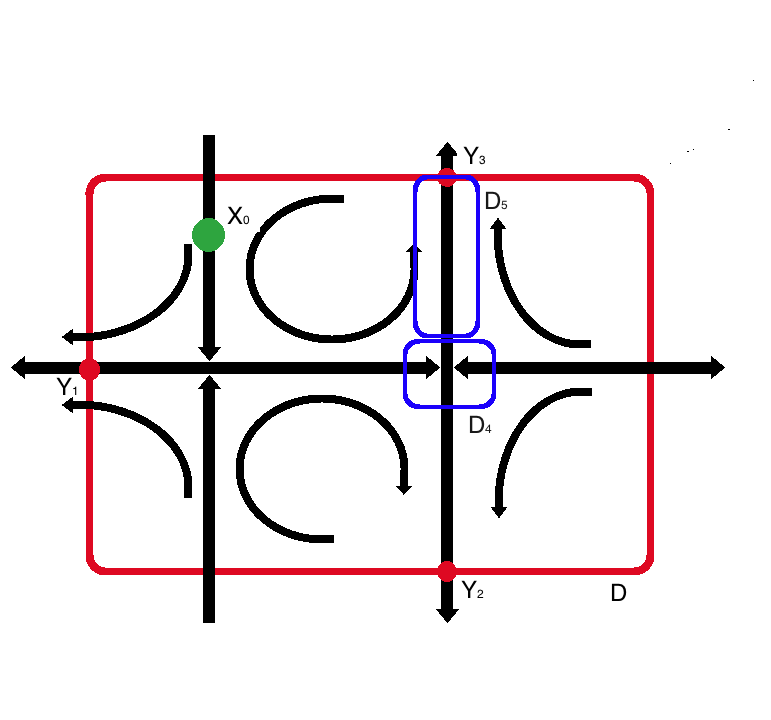}
  \caption{Illustration of the domains $D_3$ and $D_4$ used to compute the Poincar\'e maps in the case of a heteroclinic network with 2 nodes condition on exit along $y_3$: escape from the second saddle.}
  \label{fig:animals2}
\end{figure}

\section{General Setting}
\label{Sec: Intro-Notation}

The objective of this section is to establish the general setting and notation, although each chapter has the necessary modifications and additions to the following.

Let $(\Omega, \mathcal{F}, \Pp)$ be a complete probability space (every subset of every measurable null set is measurable) and $W$ be a $d$-dimensional standard Brownian Motion on it. Let $( \mathcal{F}^W_t )_{t>0}$ be the filtration generated by $W$ which satisfies the usual hypothesis~\cite[Section I.5]{Protter} . We assume that $(\Omega, \mathcal{F}, \Pp)$ is rich enough to accommodate a family of random vectors $(\xi_\eps)_{\eps\geq 0}$ in $\R^d$ such that the sigma algebra generated by $\xi_\eps$ is independent of $\mathcal{F}_\infty^W$ for each $\eps>0$. For each $\eps>0$, we consider the left continuous filtration
\[
\mathcal{G}^\eps_t=\sigma(\xi_\eps)\vee \mathcal{F}_t^W,
\]
as well as the collection of null sets
\[
N_0=\{ Z \subset \Omega : \exists G \in \mathcal{G}^\eps_\infty \text{ with } Z \subset G \text{ and } \Pp\{G\}=0 \}.
\]
Let us create the augmented filtration $\mathcal{F}_t^\eps=\sigma(\mathcal{G}_t^\eps \cup N_0)$ for $t \in [0,\infty)$, and $\mathcal{F}^\eps_\infty=\sigma( \cup_{ t \geq 0 } \mathcal{F}_t^\eps )$. It can be shown that $W$ is a brownian motion with respect to $( \mathcal{F}^\eps_t)_{t\geq0}$, the path of $W$ is independent of $\xi_\eps$ and $( \mathcal{F}^\eps_t)_{t\geq0}$ satisfies the usual hypothesis, for every $\eps>0$.

Throughout the text, we suppose that the family of random variables $(\xi_\eps)_{\eps\geq 0}$ satisfies $\xi_\eps \to \xi_0$ in distribution, and that $\xi_\eps$ has a finite second moment for each $\eps>0$.

Consider a $C^\infty$-smooth vector field~$b$ on $\R^d$ and a $C^2$-smooth matrix valued function $\sigma:\R^d \to \R^{d \times d}$. Consider the It\^o stochastic differential equation
\begin{equation}
dX_\eps = b( X_\eps )dt + \eps \sigma(X_\eps)dW \label{eqn: PrincipalEquation}
\end{equation}
equipped with initial condition
\begin{equation} \label{eqn: intro_initial_principal}
\x(0)=x_0 + \eps^\alpha \xi_\eps,
\end{equation}
where $\alpha \in (0,1]$. Hypothesis regarding the point $x_0 \in \R^d$ will be given in each chapter. We assume that both $b$ and $\sigma$ are uniformly Lipschitz and bounded, i.e.,\ there is a constant $L>0$ such that
\begin{align*}
|\sigma(x)-\sigma(y)| \vee |b(x)-b(y)|& \leq L|x-y|,\quad x,y\in \R^2, \\
|\sigma(x)| \vee |b(x)|& \leq L,\quad x\in \R^2,
\end{align*}
where $|\cdot|$ denotes the Euclidean norm for vectors and Hilbert--Schmidt norm for matrices. Further assume that the matrix function $a=\sigma \sigma^*$ is uniformly positive definite. These conditions can be weakened, but we prefer this setting to avoid multiple localization procedures throughout the text. These assumptions imply~\cite[Theorems 5.2.5 and 5.2.9]{Karatzas--Shreve} that equation~\eqref{eqn: PrincipalEquation} has a strong solution with strong uniqueness on the filtered probability space $(\Omega, \mathcal{F}, \Pp , ( \mathcal{F}^\eps_t)_{t\geq 0} )$ with initial condition~\eqref{eqn: intro_initial_principal} for each $\eps>0$. Let us recall the definition of strong uniqueness and strong solution for completeness.
\begin{definition}
A strong solution of the stochastic differential equation~\eqref{eqn: PrincipalEquation} with initial condition~\eqref{eqn: intro_initial_principal} on the filtered probability space $(\Omega, \mathcal{F}, \Pp , ( \mathcal{F}^\eps_t)_{t\geq 0} )$ is a process $\x=\{ \x(t); 0 \leq t < \infty \}$ with continuous sample paths and with the following properties:
\begin{enumerate}
\item $\x$ is adapted to the filtration $( \mathcal{F}^\eps_t)_{t\geq 0}$,
\item $\Pp \{ \x (0 ) = x_0 + \eps^\alpha \xi_\eps \}=1$,
\item $\Pp \{ \int_0^t (|b_i ( \x (s) )| + \sigma_{i,j} ( \x (s) )^2 )ds < \infty \} =1$ for every $1 \leq i,j \leq d$ and $ t\geq 0$,
\item the integral version of~\eqref{eqn: PrincipalEquation}
\[
\x(t)=\x(0)+\int_0^t b( \x (s) )ds + \eps \int_0^t \sigma (\x(s))dW(s); 0 \leq t < \infty,
\]
holds with probability $1$.
\end{enumerate}

Given two strong solutions $\x$ and $\widetilde{X}_\eps$ of~\eqref{eqn: PrincipalEquation} with initial condition~\eqref{eqn: intro_initial_principal} relative to the same brownian motion $W$. Then, we say that strong uniqueness holds whenever $\Pp \{ \x (t) =  \widetilde{X}_\eps (t) ; 0 \leq t < \infty \}=1$.

\end{definition}
For a general background on stochastic differential equations see, for example,~\cite[Chapter 5]{Karatzas--Shreve}.

The flow generated by $b$ is denoted by $S=(S^tx)_{(t,x)\in\R \times \R^d}$. That is, $S^t x$ satisfies
\begin{equation*}
\frac{d}{dt}S^{t}x=b(S^{t}x),\quad S^{0}x=x.
\end{equation*}
The linearization of $S$ along the orbit of $x$ is denoted by $\Phi_x(t)$:
\begin{equation}
\frac{d}{dt}\Phi_x(t)=A(t)\Phi_x(t)\text{, \ }\Phi_x(0)=I, \label{eqn: intro-Phi_def}
\end{equation}
where $A(t)=\nabla b(S^tx)$ and $I$ is the identity matrix. Here $\nabla$ is the derivative operator, that is, for a differentiable vector field $h:\R^d \to \R^d$, $\nabla h$ is the $\R^{d \times d}$ matrix derivative of $h$.

Throughout the text $D$ is a domain (open, connected and bounded) in $\R^d$ with piecewise $C^2$ boundary. 

The exit problem for the diffusion process $X_\eps$ in $D$ is studied. We are interested in the joint asymptotic properties (as $\eps \to 0$) of $(\x ( \tau_\eps^D ), \tau_\eps^D )$, where $\tau _{\epsilon }^{D}$ is the
stopping time defined by
\begin{equation*}
\tau _{\epsilon }^{D}=\tau _{\epsilon }^{D}(x_{0})=\inf \{t>0:X_{\epsilon
}(t)\in \partial D\}.
\end{equation*}.

Specific hypotheses on the vector field will be given in each chapter. On the other hand, the abstract formulation will not be given in each chapter, instead we assume this technical formulation to hold throughout the text.

\subsection{Organization of the Text}

The organization of the rest of the text closely mimics the presentation given in this chapter.

In Chapter~\ref{ch: Saddle} the planar (i.e. $d=2$) exit problem is studied under the assumption that $S$ has a unique saddle at the origin. That is, $0 \in \R^2$ is the only critical point $b(0)=0$ and the eigenvalues $\lambda_+$, $- \lambda_-$ of the of the matrix $\nabla b(0)$ are such that $\lambda_\pm >0$.  The exit problem is studied conditioned that the process $\x$ starts on the stable manifold $\mathcal{W}^s$ of $0$.

In Chapter~\ref{ch: levinson} the Levinson case in arbitrary dimensions is considered. We also proved Lemma~\ref{lemma: important} stated in this chapter, and use it intesively in the proof of the Levinson case. In this section, we also study the $1$-dimensional example discussed in Section~\ref{sec: 1d-time-reversal} of this chapter.

In Chapter~\ref{ch: conclusion} we present a short survey on how the techniques in this text can be extended. Several open problems are also discussed.

\chapter{Saddle Point}\label{ch: Saddle}
In this chapter we study the stochastic process $\x$ when the underlying deterministic system $S$ has a unique saddle point. 

In Section~\ref{Sec: Notation_Saddle} we introduce the setting, which relies on the setting presented in Section~~\ref{Sec: Intro-Notation} of Chapter~\ref{ch: Intro}. In
Section~\ref{sec:statement_main} we state the main theorem and split the proof into several parts. 
In Section~\ref{sec: Normal_Forms} we introduce a simplifying change of coordinates in a small neigborhood of the saddle point.  The analysis of the transformed process in Section~\ref{sec: Main}
is based upon two results. Their proofs are given in Sections~\ref{sec:close_to_saddle} 
and~\ref{sec:unstable_manifold}.

\section{Setting}
\label{Sec: Notation_Saddle}
For this chapter we consider the general formulation made in Section~\ref{Sec: Intro-Notation} of Chapter~\ref{ch: Intro}, except that we restrict ourselves to the $2-$dimensional situation. The process $\x$ is the strong solution of~\eqref{eqn: PrincipalEquation}, under the assumptions made on the $C^\infty$-smooth vector field~$b$, the $C^2$-smooth matrix valued function $\sigma:\R^2 \to \R^{2 \times 2}$ and the standard $2$-dimensional Wiener process $W$ in Section~\ref{Sec: Intro-Notation} of Chapter~\ref{ch: Intro}.

We will study the exit problem from the domain $D \subset \R^2$ with piecewise $C^2$ boundary.
Assume that the origin $0$ belongs to $D$ and it is a unique fixed point for $S$ in $\bar D$, 
or, equivalently, a unique critical point for $b$ in $\bar D$. Therefore, 
\begin{equation*}
b(x)=Ax+Q(x),
\end{equation*}%
where $A=\nabla b(0)$ and $Q$ is the non-linear
part of the vector field satisfying $|Q(x)|=O(|x|^2)$, $x\to 0$. 

Suppose that $0$ is a hyperbolic critical point, i.e.\ the matrix $A$ has two eigenvalues $\lambda_+$ and $-\lambda_-$ satisfying $-\lambda_- < 0 < \lambda_+$. Without loss of generality, we suppose
that the canonical vectors are the eigenvectors for the matrix, so that $A=\mathop{\mathrm{diag}}(\lambda_+, -\lambda_-)$.

For an interval $J \subset \R$, let $S^J x$ denote the set
\[
S^J x = \{ S^t x: t\in J \}.
\]
According to the Hadamard--Perron Theorem (see e.g.~\cite[Section~2.7]{Perko}),
the curves $\mathcal{W}^{s}$ and $\mathcal{W}^{u}$ defined via
\begin{align*}
\Wc^u &=\{ x\in \bar{D}:  \lim_{t \to -\infty } S^tx = 0 \text{, and for some } s\geq 0, S^{(-\infty,s)}x \subset D \text{ and } S^{(s,\infty)}x\cap \bar {D} = \emptyset  \},
\end{align*}
and,
\[
\Wc^s =\{ x\in \bar{D}:  \lim_{t \to \infty } S^tx = 0 \text{, and for some } s \leq 0, S^{(-\infty,s)}x \cap \bar {D} = \emptyset \text{ and } S^{(s,\infty)}x \subset D \}.
\]
are smooth, invariant under $S$ and tangent to $e_2$ and, respectively, to $e_1$ at $0$. The curve
$\Wc^{s}$ is called the stable manifold of $0$, and $\Wc^{u}$ is called the unstable manifold of $0$.

We assume that $\mathcal{W}^u$
intersects $\partial D$ transversally at points $q_{+}$ and $q_{-}$ such that
the segment of $\mathcal{W}^u$ connecting $q_{-}$ and $q_{+}$ lies entirely inside $D$ and 
contains~$0$.

We fix a point $x_{0}\in\mathcal{W}^s\cap D$ and
equip~\eqref{eqn: PrincipalEquation} with the initial condition
\begin{equation}
X_{\epsilon }(0)=x_{0}+\epsilon ^{\alpha }\xi _{\epsilon },\quad\eps>0,
\label{eq:initial_condition}
\end{equation}%
where $\alpha \in (0,1]$ is fixed, and $(\xi _{\epsilon })_{\epsilon >0}$ is
a family of random vectors independent of $W$, such that for some random
vector $\xi _{0}$, $\xi _{\epsilon }\rightarrow \xi _{0}$ as $\epsilon 
\rightarrow 0$ in distribution. 

If $\alpha\ne1$, then we impose a further technical
condition
\begin{equation}
 \Pp\{ \xi _0\parallel b(x_0)\} = 0,  \label{eqn: intial_hyp}
\end{equation}
where $\parallel$ denotes collinearity of two vectors.

\section{Main Result.}\label{sec:statement_main}
 The main result of the present chapter is the following:

\begin{theorem}
\label{Thm: Main} In the setting described above, 
there is a family of random vectors $(\phi _{\epsilon })_{\eps>0}$, a family of random variables $(\psi_{\epsilon })_{\eps>0}$,
and a number 
\begin{equation}
\beta=\left\{ 
\begin{array}{cc}
1, & \alpha \lambda _- \geq \lambda_+ \\ 
\alpha \frac{\lambda_-}{\lambda _+}, & \alpha \lambda _- <\lambda _+ \\
\end{array}%
\right.
\label{eq:alpha_0}
\end{equation}
such that%
\begin{equation*}
X_\eps (\tau _\eps^{D})=q_{\sgn(\psi_{\epsilon })}+\epsilon
^{\beta}\phi _{\epsilon }.
\end{equation*}%
The random vector 
\[\Theta_\eps=\left(\psi_\eps, \phi_\eps, \tau _{\epsilon }^{D}+\frac{\alpha }{\lambda _{+}}\ln \epsilon\right)\]
converges in distribution as $\epsilon\to 0$.
\end{theorem}

The distribution of $\psi_{\epsilon }$,$\phi _{\epsilon }$, and the distributional limit of $\Theta_\eps$
will be described precisely.

The proof of Theorem~\ref{Thm: Main} has essentially three parts involving the analysis of diffusion (i) along $\mathcal{W}^s$;
(ii) in a small neighborhood of the origin; (iii) along~$\mathcal{W}^u$.  

In order to study the first part, we need to introduce $\Phi _{x}(t)$ as the linearization of $S$ along the
orbit of $x\in\R^2$, i.e.\ we define $\Phi_{x}(t)$ to be the
solution to the matrix ODE
\begin{equation*}
\frac{d}{dt}\Phi _{x}(t)=A(t)\Phi _{x}(t)\text{, \ }\Phi _{x}(0)=I,
\end{equation*}
where $A(t)=\nabla b(S^tx)$.
We have the following theorem:

\begin{theorem} 
\label{thm: far}
Let $x \in \R^2$ and $\left( \xi_\eps \right)_{\eps>0}$ be a family of random vectors independent of $W$ and convergent in distribution, as $\eps \to 0$, to $\xi_0$. Suppose
$\alpha\in(0,1]$ and let $X_\eps$ be 
the solution of the SDE~\eqref {eqn: PrincipalEquation} with initial condition $X_{\epsilon }(0)=x+\epsilon ^{\alpha }\xi_{\epsilon }$. Then, for every $T>0$, the following representation holds true:
\begin{equation*}
X_{\epsilon }(T)=S^{T}x+\epsilon ^{\alpha }\bar{\xi}_{\epsilon },\quad \eps>0,
\end{equation*}%
where 
\begin{equation*}
\bar{\xi}_{\epsilon }\overset{\mathop{Law}}{\longrightarrow }\bar{\xi}_{0},\quad \eps\to 0,
\end{equation*}%
with%
\begin{equation*}
\bar{\xi}_{0}=\Phi _x(T)\xi_{0}+\mathbf{1}_{\{\alpha =1\}}N,
\end{equation*}%
$N$ being a Gaussian vector:
\begin{equation*}
N=\Phi _x (T)\int_{0}^{T}\Phi _x (s)^{-1}\sigma(S^s x)dW(s).
\end{equation*}
If $\alpha=1$ or assumption~\eqref{eqn: intial_hyp} holds, then $\Pp\{\bar \xi_0\parallel b({S^Tx})\}= 0$.
\end{theorem}

\medskip

The second part of the analysis is the core of the chapter. Theorem~\ref{Thm: main_small} below
describes the behavior of the process in a small neighborhood $U$
of the origin. Notice that
since $x_0\in \Wc^s$, one can choose~$T$ large enough to ensure that that $S^Tx_0\in \Wc^s\cap U$. Therefore,
 the conditions of the following result are met if we use the terminal distribution of Theorem~\ref{thm: far} (applied to
the initial data given by~\eqref{eq:initial_condition}) as the initial
distribution.

\begin{theorem}\label{Thm: main_small}
There are two neighborhoods of the origin $U\subset U'\subset D $, two positive numbers $\delta<~\delta'$, 
and $C^2$ diffeomorphism $f:U'\to (-\delta',\delta')^2$, such that $f(U)=(-\delta,\delta)^2$
and the following property holds:
 
Suppose $x\in \Wc^s\cap U$, and $(\xi_\eps)_{\eps>0}$ is a family of random variables independent of $W$ and convergent in distribution, as $\eps\to0$, to $\xi_0$, where $\xi_0$ satisfies~\eqref{eqn: intial_hyp} with respect to $x$. Assume that
$\alpha\in(0,1]$ and that $X_\eps$ solves~\eqref{eqn: PrincipalEquation} with initial condition~
\begin{equation}
X_\eps(0)=x+\eps^\alpha\xi_\eps,
\label{eq:small_entrance_scaling}
\end{equation} 
where $\xi_\eps$ satisfies condition~\eqref{eqn: intial_hyp} with respect to $x$.

There is also a family of random vectors $(\phi'_\eps)_{\eps>0}$, and a family of random variables $(\psi'_\eps)_{\eps>0},$ such that
\begin{equation*}
X_{\epsilon }(\tau _{\epsilon }^{U}) = g( \sgn(\psi'_\eps)\delta e_1)+\epsilon ^{\beta} \phi' _\eps,
\end{equation*}
where $g=f^{-1}$, $\beta$ is defined in~\eqref{eq:alpha_0}, and 
the random vector 
\[\Theta'_\eps=\left(\psi'_\eps, \phi'_\eps, \tau_{\epsilon }^{U}+\frac{\alpha }{\lambda _{+}}\ln \epsilon\right)\]
converges in distribution as $\epsilon\to 0$.
\end{theorem}

The notation for $\Theta'_\eps$ and its components is chosen to match the notation involved in the
statement of Theorem~\ref{Thm: Main}. 
Random elements $\psi' _{\epsilon }$,$\phi'_{\epsilon }$ and the distributional limit of $\Theta'_\eps$
will be described precisely, see~\eqref{eq:limiting_distr_for_Theta_prime}. 
Obviously, the symmetry or asymmetry in the limiting distribution of $\psi'_\eps$ results in the symmetric or asymmetric choice of exit direction so that the exits in the positive
and negative directions are equiprobable or not. On the other hand, the limiting distribution of $\phi'_\eps$ determining the asymptotics of the exit point can also be symmetric
or asymmetric which results in the corresponding features of the random choice of the exit  direction at the next saddle point visited by the diffusion. 

In Section~\ref{sec: Main} we prove Theorem~\ref{Thm: main_small} using the approach based on normal forms.

The last part of the analysis is devoted to the exit from $D$ along $\Wc^u$. We need the following statement which is a specific case of the main result of Chapter~\ref{ch: levinson}.

\begin{theorem}\label{Thm:Transversal_exit} In the setting of Theorem~\ref{thm: far}, assume additionally that (i) $q=S^Tx\in \partial D$; (ii) there is no
$t\in[0,T)$ with $S^tx\in\partial D$;  (iii) $b(q)$ is tranversal (i.e.\ not tangent) to~$\partial D$ at~$q$. Then
\begin{equation}
\tau_\eps^D\stackrel{\Pp}{\to} T,\quad \eps\to 0, 
\end{equation}
and
\begin{equation}
\eps^{-\alpha} (X_\eps(\tau_\eps^D)-q) \stackrel{Law}{\to}\pi \bar \xi_0 ,\quad \eps\to0,
\end{equation}
where $\pi$ denotes the projection along $b(q)$ onto the tangent line to $\partial D$ at $q$. 
\end{theorem}

Now Theorem~\ref{Thm: Main} follows from the consecutive application of Theorems~\ref{thm: far} through~\ref{Thm:Transversal_exit} and with 
the help of the strong Markov property. In fact, in this chain of theorems, the conclusion of 
Theorem~\ref{thm: far} ensures that the conditions of Theorem~\ref{Thm: main_small} hold, and
the conclusion of the latter ensures that the conditions of Theorem~\ref{Thm:Transversal_exit} hold.
Notice that
the total time needed to exit~$D$ equals the sum of times described in the three theorems.
Notice also that at each step we can compute the limiting initial and terminal distributions explicitly. Theorems~\ref{thm: far} and~\ref{Thm:Transversal_exit} contain the respective formulas 
in their formulations, and the explicit limiting distribution for $\Theta'_\eps$ of Theorem~\ref{Thm: main_small} is computed in~\eqref{eq:limiting_distr_for_Theta_prime}.

\section {Simplifying change of coordinates}
\label{sec: Normal_Forms}

\subsection{Smooth Transformation and Normal Forms} 

In this section we give a brief review of the theory of Normal Forms. In particular, we focus on the neighborhood of a saddle point for the deterministic flow $S$. 

The idea is to find a local change of variables $\theta:\R^2 \to \R^2$ such that $z(t)=\theta ( S^t x )$ satisfies $\dot{z}=Az$ with the appropriate initial condition. First, note that $z$ satisfies the equation
\begin{align*}
\frac{d}{dt}z(t) &= \nabla  \theta ( S^t x  ) b ( S^t x ) \\
&= \nabla  \theta ( z(t) )^{-1} b ( \theta^{-1} ( z(t) ) ), \quad z(0)=\theta(x).
\end{align*}
Hence, the goal is to find a transformation $\theta:\R^2 \to \R^2$ that leaves $\nabla  \theta ( z )^{-1} b ( \theta^{-1} ( z ) )$ as simple as possible (ideally equal to $Az$).

We start with some notions. For a multi-index $\alpha =(\alpha _{1},\alpha
_{2})\in \mathbb{Z}_{+}^{2}$ and a base $\{e_{1},e_{2}\}$ of $\mathbb{R}^{2}$%
 (as a vector space over $\mathbb{R}$) we denote the monomial $x^{\alpha
}e_{i}=x_{1}^{\alpha _{1}}x_{2}^{\alpha _{2}}e_{i}$.
\begin{definition}
For a non-negative integer $r$,
the space of linear combinations (over $\mathbb{R}$) of monomials $x^{\alpha
}e_{i}$ with $|\alpha |=\alpha _{1}+\alpha _{2}=r$, is called the space of Homogenous Polynomials in $2$ variables of degree $r$. This space is denoted as $\mathcal{H}_{r}$. In other words, $
\mathcal{H}_{r}$ is,
\begin{equation*}
\mathcal{H}_{r}=\rm{span}_{\mathbb{R}}\left\{ x^{\alpha }e_{j}:\alpha \in \mathbb{%
Z}_{+}^{2},|\alpha |=r \text{ and } 1\leq j\leq 2\right\} .
\end{equation*}%
\end{definition}
It is easy to see that $\mathcal{H}_{r}$ is isomorphic (as a vector space
over the real numbers) to $\mathbb{R}^{2(r+1)}$.

Using this notation, use Taylor's classical theorem to decompose the function $b:\mathbb{R}^{2}\rightarrow \mathbb{R}^{2}$  as%
\begin{equation}
b(z)=Az+b_{2}(z)+...+b_{R}(z),  \label{eqn: b_expansion}
\end{equation}%
with $b_{i}\in \mathcal{H}_{i}$ for $1\leq i \leq R$, and $b_R (x) = O ( |x|^R )$ as $|x| \to 0$. 

Suppose that $z=\theta _{k}(\zeta )$, where 
$\theta _{k}$ is the near identity transformation 
\begin{equation} \label{eqn: theta}
\theta _{k}(\zeta )=\zeta+h_{k}(\zeta ), \quad h_{k}\in \mathcal{H}_{k}, \quad k \geq 2.
\end{equation}%
Note that $\theta _{k}$ is a topological diffeomorphism in a small open
neighborhood of the origin $\Omega _{k}$. Throughout we restrict the analysis inside $\Omega_k$. A Taylor approximation shows that the inverse of $\theta_k$ satisfies
\begin{align} \notag
\theta_k^{-1} ( \zeta) &= \zeta - h_k ( \zeta ) + O ( |\zeta|^{ 2 k -1 } ) \\
&=  \zeta - h_k ( \zeta ) + O ( |\zeta|^{ k + 1 } ). \label{eqn: thetaInverse} 
\end{align}
Further application of Taylor's approximation together with the condition that $k\geq2$, imply that for any $\zeta \in \Omega _{k}$,
\begin{align*}
\nabla \theta _{k}(\zeta )^{-1} &=I-\nabla h_{k}(\zeta )+O(|\zeta |^{2 ( k-1 )}) \\
&= I-\nabla h_{k}(\zeta )+O(|\zeta |^{k}).
\end{align*}
Also, from~\eqref{eqn: thetaInverse}, we obtain that for any $i=1,...,R-1$, 
\[
b_{i}(\theta^{-1} _{k}(\zeta )) =b_{i}(\zeta )+O(|\zeta |^{k+1}).
\]
Using~\eqref{eqn: b_expansion} and this bounds, we get that 
\begin{eqnarray*}
(\nabla\theta _{k}(\zeta ))^{-1}b(\theta _{k}^{-1}(\zeta )) &=&A\zeta +b_{2}(\zeta
)+...+b_{k-1}(\zeta ) \\
&&+(b_{k}(\zeta )-\mathcal{L}_{A}^{k}h_{k}(\zeta ))+O(|\zeta |^{k+1}),
\end{eqnarray*}%
where we defined the operator $\mathcal{L}_{A}^{k}:\mathcal{H}%
_{k}\rightarrow \mathcal{H}_{k}$ by 
\begin{equation} \label{eqn: L_def}
\mathcal{L}_{A}^{k}h(\zeta )=h(\zeta )A\zeta -A\nabla h(\zeta ).
\end{equation}
It is clear that the following theorem holds:
\begin{theorem} \label{thm: Normal}
Let $\mathcal{ R } ( \mathcal{L}_{A}^{k} ) \subset \mathcal{H}_k$ be the range of the operator $\mathcal{L}_{A}^{k}:\mathcal{H}%
_{k}\to \mathcal{H}_{k}$. Take $\mathcal{ I }_k \subset \R^2$ be any subspace such that $\mathcal{H}_k=\mathcal{ R } ( \mathcal{L}_{A}^{k} ) \oplus \mathcal{ I }_k$. Then, there is a sequence of near identity transformations of the form~\eqref{eqn: theta} and nested neighborhoods of the origin $\Omega_{k+1} \subset \Omega_{k}$, such that $z(t)= \theta_r \circ \dots \circ \theta_2 (S^t x) $ satisfies
\[
\frac{d}{dt} z (t) = Az(t) + b_2 ( z(t) ) + \dots + b_r ( z(t) ) + O ( |z|^{ r+1 } ),
\]
inside $\Omega_r$, and $b_k \in \mathcal{ I }_k$, $k=1,...,r$.

An equation written in this form is said to be in Normal Form up to order $r$. 
\end{theorem}

The idea is to characterize the image of the operator $\mathcal{L}_{A}^{k}$
and simplify each non-linear part of $b$, starting from $b_{2}$ and all the
way up to $b_{R}.$ In order to achieve this, we remark that, $x^{\alpha
}e_{j}$ is an eigenvector of $\mathcal{L}_{A}^{k}$ for any $\alpha \in 
\mathbb{Z}_{+}^{2}$: 
\begin{equation*}
\mathcal{L}_{A}^{k}x^{\alpha }e_{j}=(\lambda ^{T}\alpha -\lambda
_{j})x^{\alpha }e_{j}\text{,}
\end{equation*}%
for $\lambda =(\lambda _{+},\lambda _{-})$. This motivates the following definition:
\begin{definition} 
 A pair of complex numbers $\lambda=(\lambda_1, \lambda_2)$ is said to be non-resonant if there are no 
integral relations between them of the form $\lambda_j=\alpha \cdot \lambda$, where $\alpha =(\alpha _1,\alpha_2) \in \mathbb{Z}_{+}^{2}$ 
is a multi-index with $|\alpha|=\alpha_1 + \alpha_2 \geq 2$. Otherwise, we say that $\lambda=(\lambda_1, \lambda_2)$ is resonant. 

A resonant $\lambda$ is said to be one-resonant if all the resonance relations for $\lambda$ follow from a single resonance relation.  

A monomial $x^\alpha e_j=x_1^{\alpha _1 } x_2 ^{\alpha _2 } e_j$ is called a resonant monomial of order~$R$ if~$\alpha \cdot \lambda = \lambda_j$ and $|\alpha|=R$.
\end{definition}

In the spirit of Theorem~\ref{thm: Normal} it is clear (see~\cite{IIlyashenko},\cite{MR1290117}) that for any pair of integers $R\geq1$ and $k \geq 1$, there are two neighborhoods of the origin $\Omega_f$ and $\Omega_g$  and a $C^k$-diffeomorphism $f:\Omega_f\to \Omega_g$ with inverse  $g:\Omega_g\to \Omega_f$ such that  
\begin{equation}
\left( \nabla g (y) \right)^{-1}b( g (y) )=Ay + P(y) + \mathcal{R}(y), \quad y \in \Omega_g \label{eqn; Pre-NormalForm}
\end{equation} 
where $P$ is a polynomial containing only resonant monomials of order at most~$R$ and $\mathcal{R}(\zeta)=O(|\zeta|^{R+1})$. Moreover, the so called Poincar\'e theorem~\cite[Theorem 2.2.4]{MR1290117} asserts that if $\lambda$ is non-resonant, then $f$ can be chosen so that both $P$ and $\mathcal{R}$ in~\eqref{eqn; Pre-NormalForm} are identically zero.  
If $\lambda$ is one-resonant then~\cite[Theorem~3,Section~2]{IIlyashenko}  says that $f$ can be chosen so that $\mathcal{R}$ in~\eqref{eqn; Pre-NormalForm} is identically zero. 
More precisely:
\begin{lemma} \label{lemma: Ill} 
For any $k \geq 1$, there are two neighborhoods of the origin $\Omega_f$ and $\Omega_g$  and a $C^k$-diffeomorphism 
$f:\Omega_f\to \Omega_g$ with inverse  $g:\Omega_g\to \Omega_f$
such that  
\begin{equation}
\left( \nabla g(y) \right)^{-1}b( g (y) )=Ay + P(y), \quad y \in \Omega_g, \label{eqn: NormalForm}
\end{equation}
where $P$ is a polynomial that contains only resonant monomials.
\end{lemma}

This is the core result we use to study the stochastic case in the next section.

\subsection{Change of Variables in the Stochastic Case} \label{sec: saddle-changevar}
In this section we start analyzing the diffusion in the neighborhood of the saddle point. The first step is to
find a smooth coordinate change that would simplify the system. This can be done with the help of the theory of normal forms presented on the last section.

Let $g$ be a $C^\infty-$diffeomorphism of a neighborhood of the origin with inverse~$f$. When $X_\eps$ is close to the origin and belongs to the image of that neighborhood under $g$, we can  use It\^o's formula to see that $Y_\eps=f(X_\eps)$ satisfies 
\begin{align*}
dY_{\epsilon } &=\nabla f(X_{\epsilon })dX_{\epsilon }+\frac{1}{2}%
[\nabla f(X_{\epsilon }),X_{\epsilon }] \\
&=\nabla f(g (Y _{\epsilon }))b(g (Y_{\epsilon
}))dt+\epsilon \sgm (Y_{\epsilon })dW +\epsilon ^{2}\Psi (Y_{\epsilon })dt,
\end{align*}%
for some smooth function $\Psi :\mathbb{R}^{2}\rightarrow \mathbb{R}^{2}$ and $\sgm =\left ((\nabla f)\circ g \right) \sigma$. 
Here the square brackets mean quadratic covariation.
Since $\nabla f \circ g =(\nabla g)^{-1}$, we can rewrite the above SDE as in the deterministic case as
\begin{equation}
dY_{\epsilon }=\left( \left( \nabla g (Y_{\epsilon })\right)
^{-1}b(g (Y_{\epsilon }))+\epsilon ^{2}\Psi (Y_{\epsilon
})\right) dt+\epsilon \sgm (Y_{\epsilon })dW.  \label{eqn: Eqn_zeta}
\end{equation}

In order to simplify the drift term in this equation, we rely on Lemma~\ref{lemma: Ill}. First,
note that $(\lambda_+,-\lambda_-)$ is either non-resonant or one-resonant (resonant cases that are not one-resonant are possible in higher dimensions where
pairs of eigenvalues get replaced by vectors of eigenvalues). The non-resonant case (in any dimension) was studied in~\cite{nhn}. 
In this paper, we extend the analysis of~\cite{nhn} to the non-resonant case, i.e.\ the one-resonant case, given that we are working in 2 dimensions. 

To find all resonant monomials of a given order $r\ge 2$, we have to find all the integer solutions to the two $2 \times 2$ systems of equations:%
\begin{align*}
\alpha _{1}\lambda _{+}-\alpha _{2}\lambda _{-} &=\pm\lambda _{\pm}, \\
\alpha _{1}+\alpha _{2} &=r.
\end{align*}
Therefore, the power multi-indices of a resonant monomial of order $r$ has to coincide with one of the following:
\begin{eqnarray}
(\alpha _{1}^{+}(r),\alpha _{2}^{+}(r)) &=&\frac{1}{\lambda _{+}+\lambda _{-}%
}(\lambda _{+}+r\lambda _{-},(r-1)\lambda _{+}),
\label{eqn: alpha_plus} \\
(\alpha _{1}^{-}(r),\alpha _{2}^{-}(r)) &=&\frac{1}{\lambda _{+}+\lambda _{-}%
}((r-1)\lambda _{-},r\lambda _{+}+\lambda _{-}),  \label{eqn: alpha_minus}
\end{eqnarray}%
Let us make some elementary observations on integer solutions of these equations for $r\ge 2$.
\begin{enumerate}
\item None of the solution indices can be $0$.  Moreover, neither $\alpha_1^+(r)$ nor $\alpha_2^-(r)$ can be equal to $1$. 
\item As functions of $r$, $\alpha_i^\pm (r)$ are increasing.
\item Expressions~\eqref{eqn: alpha_plus} and~\eqref{eqn: alpha_minus} cannot be an integer for $r=2$. 
\item The term $P=(P_1,P_2)$ in~\eqref{eqn: NormalForm} satisfies $P_1(y)=O(y_1^2 |y_2|)$ and~$P_2(y)=O(|y_1| y_2^2)$. 
This observation is a consequence of observations 1 and 3 since they imply that resonant multi-indices have to satisfy $\alpha^+(r)\ge (2,1)$ and $\alpha^-(r)\ge (1,2)$ coordinatewise.
\item If at least one of the coordinates $y_1$ and $y_2$ is zero, then  $P(y_1,y_2)=0$. This is a direct consequence of the previous observation.
\end{enumerate}
Given all these considerations, the main theorem of this section is a simple consequence of Lemma~\ref{lemma: Ill}.
\begin{theorem}
\label{Lemma: Def_H1&H2} In the setting described in Section~\ref{Sec: Notation},
there is a number $\delta'>0$, a neighborhood of the origin $U'$, and a $C^2$-diffeomorphism $f:U'\to (-\delta',\delta')$
with inverse  $g:(-\delta',\delta')^2\to U'$ such that  the following property holds.

If $X_\eps(0)\in U$, then the stochastic process $Y_\eps=(Y_{\eps,1},Y_{\eps,2})$ given by 
\[
Y_\eps(t)=f (X_\eps(t\wedge \tau_\eps^{U}))
\]
satisfies the following system of SDEs up to $\tau_\eps^{U}$ : 
\begin{align}
\label{eq:SDE_changed_coord1}
dY_{\epsilon,1 } &=\left(  \lambda _{+}Y_{\epsilon,1 }+H_{1}(Y_{\epsilon}, \eps) \right) dt+\epsilon \sgm_{1}(Y_{\epsilon})dW \\
\label{eq:SDE_changed_coord2}
dY_{\epsilon,2 } &=\left( -\lambda _{-}Y_{\epsilon,2 }+H_{2}(Y_{\epsilon}, \eps) \right) dt+\epsilon \sgm_{2}(Y_\epsilon)dW,
\end{align}
where $\sgm_i:(-\delta',\delta')^2 \to \R$ are $C^1$ functions for $i=1,2$. 
The functions $H_i$ are given by $H_i=\hat H_i + \eps^2 \Psi_i$, where 
$\Psi_i:(-\delta',\delta')^2 \to \R^2$ are continuous bounded functions, and  $\hat H_i: (-\delta',\delta')^2 \times [0,\infty)$ are polynomials, so that for some constant $K_{1}>0$ 
and for any $y\in (-\delta',\delta')^2$,
\begin{align*}
|\hat H_{1}(y)| &\leq K_{1}|y_1|^{\alpha_1^+}|y_2|^{\alpha_2^+}, \\
|\hat H_{2}(y)| &\leq K_{1}|y_1|^\aone |y_2|^\atwo.
\end{align*}
Here, the integer numbers $\alpha_i^\pm$, $i=1,2$, are such that $(\alpha_1^+,\alpha_2^+)$ is of the form ~\eqref{eqn: alpha_plus} 
for some choice of $r=r_1\ge 3$, and
and $(\alpha_1^-,\alpha_2^-)$ is of the form~\eqref{eqn: alpha_minus} for some choice $r=r_2\ge 3$. In particular,
\begin{align*}
|H_{1}(y,\eps)| &\leq K_{1}y_1^{2}|y_2|+K_{2}\epsilon ^{2},\\
|H_{2}(y,\eps)| &\leq K_{1}|y_1|y_2^{2}+K_{2}\epsilon ^{2},
\end{align*}
for some constants $K_{1}>0$ and $K_{2}>0$.
 
\end{theorem}

\section{Proof of Theorem~\ref{Thm: main_small}}
\label{sec: Main}
In this section we derive Theorem~\ref{Thm: main_small} from several auxiliary statements. Their proofs are postponed to later sections.

Theorem~\ref{Lemma: Def_H1&H2} allows to work with process $Y_\eps=f(X_\eps)$ instead of $X_\eps$ while $Y_\eps$ stays in $(-\delta',\delta')^2$

If we take $\delta\in(0,\delta')$, then for the initial conditions considered in Theorem~\ref{Thm: main_small} and given in~\eqref{eq:small_entrance_scaling},
\[
 \Pp\{X_\eps(0)\in U'\}\to 1,\quad\eps\to 0,
\]
i.e.,
\[
 \Pp\{Y_\eps(0)\in (-\delta',\delta')^2\}\to 1,\quad\eps\to 0.
\]
Moreover, denoting $f(x)$ by $y=(0,y_2)$ we can write
\[
 Y_\eps(0)=y+\eps^\alpha \chi_\eps = (\eps^\alpha \chi_{\eps,1},\ y_2+\eps^\alpha \chi_{\eps,2}),\quad \eps>0,
\]
where $\chi_\eps=(\chi_{\eps,1},\chi_{\eps,2})$ is a random vector convergent in distribution to $\chi_0=(\chi_{0,1},\chi_{0,2})=\nabla f(x)\xi_0$. Due to the hypothesis in Theorem~\ref{Thm: main_small}, we notice that the distribution of $\chi_{0,1}$ has no atom at $0$.

Let us take any $p \in (0,1)$ such that 
\begin{equation}
 1-\frac{\lambda _{+}}{\lambda _{-}}<p<  \frac{\lambda_{-}}{\lambda_{+}+\lambda_{-}},  
\label{eqn: p_prop}
\end{equation}
and define the following stopping time: 
\[
\hat\tau _{\epsilon}=\inf \{t:|Y_{\epsilon,1 }(t)|=\epsilon^{\alpha p}\}.
\]

Up to time $\hat \tau_\eps$, the process $X_\eps$ mostly evolves along the stable manifold~$\mathcal{W}^{s}$.
After $\hat \tau_\eps$, it evolves mostly along the unstable manifold $\mathcal{W}^{u}$. Process $Y_\eps$ evolves
accordingly, along the images of $\Wc^s$ and $\Wc^u$ coinciding with the coordinate axes.

Let us introduce random variables $\eta_\eps^{\pm}$ via
\begin{align*}
\eta _{\epsilon }^{+}&=\epsilon ^{-\alpha }e^{-\lambda _{+}\hat\tau
_{\epsilon }}Y_{\epsilon,1 }(\hat\tau_{\epsilon }),\\ 
\eta
_{\epsilon }^{-}&=\epsilon ^{-\alpha (1-p)\lambda _{-}/\lambda
_{+}}Y_{\epsilon,2 }(\hat\tau _{\epsilon }).
\end{align*}
Also we define the distribution of random vector $(\eta _{0}^{+ },\eta_0^-)$ via
\begin{eqnarray}
\eta _{0}^{+} &=&\chi_{0,1}+\mathbf{1}_{\{\alpha =1\}}N^{+}, \label{eq:eta_0_+}\\
\eta _{0}^{-} &=&|\eta _{0}^{+}|^{\lambda _{-}/\lambda _{+}}y_{2},\notag
\end{eqnarray}%
where
\begin{equation}
N^+=\int_{0}^{\infty}e^{-\lambda _-s}\sgm_{1}(0,e^{-\lambda _{-}s}y_{2})dW
\label{eq:N_0_+}
\end{equation}
is independent of $\chi_{0,1}$.

\begin{lemma}
\label{prop: eta_convergence} If the first inequality in~\eqref{eqn: p_prop} holds, then
\begin{equation}
\label{eq:exit_at_tauh_along_axis_1} 
\Pp\{Y_{\eps,1}(\tauh) = \eps^{\alpha p} \sgn \eta^+_\eps\}\to 1,\quad \eps\to 0.
\end{equation}
and
\begin{equation}
 \left(\eta_\eps^+,\eta_\eps^-, \hat \tau_\eps+\frac{\alpha }{\lambda _{+}}(1-p)\log \epsilon\right)\ 
\stackrel{Law}{\longrightarrow}\ \left(\eta_0^+,\eta_0^-, -\frac{1}{\lambda _{+}}\log |\eta _{0 }^{+}| \right),\quad \eps\to 0. 
\end{equation}
\end{lemma}

We prove this lemma in Section~\ref{sec:close_to_saddle}. Along with the strong Markov property, it allows to reduce the study of the evolution of $Y_\eps$ after $\tauh$ to studying
the solution of system~\eqref{eq:SDE_changed_coord1}--\eqref{eq:SDE_changed_coord2} with initial condition
\begin{equation}
\label{eq:restart_at_tauh}
Y_{\epsilon }(0)=(\epsilon^{\alpha p}\sgn\eta _{\epsilon }^{+},\epsilon ^{\alpha (1-p)\lambda
_{-}/\lambda _{+}}\eta _{\epsilon }^{-}),
\end{equation}
where 
\begin{equation}
\label{eq:condition_for_theorem_along_unstable}
(\eta^+_\eps,\eta^-_\eps)\stackrel{Law}{\longrightarrow}(\eta^+_0,\eta^-_0),\quad\eps\to 0.
\end{equation}
We denote
\begin{equation}
 \tau _{\epsilon }=\tau _{\epsilon }(\delta )=\inf \{t\ge 0:|Y_{\epsilon,1}(t)|=\delta\}. \label{eqn: tau_def}
\end{equation}
Our next goal is to describe the behavior of $Y(\tau_\eps)$. To that end, we introduce a random variable $\theta$ via
\begin{equation}
\label{eq:theta}
\theta\stackrel{Law}{=}\begin{cases}
N, & \alpha \lambda _{-}> \lambda _{+}, \\
\left( \frac{|\eta _{0}^{+}|}{\delta }\right) ^{\lambda _{-}/\lambda
_{+}}y_{2}+N, & \alpha \lambda _{-}= \lambda _{+}, \\ 
\left( \frac{|\eta _{0}^{+}|}{\delta }\right) ^{\lambda _{-}/\lambda
_{+}}y_{2}, & \alpha \lambda _{-}< \lambda _{+}.%
\end{cases}
\end{equation} 
where the distribution of $N$ conditioned on $\eta_0^+$, on $\{\sgn \eta_0^+=\pm1\}$ is
centered Gaussian with variance
\begin{equation*}
\sigma _{\pm}=\int_{-\infty}^{0}e^{2\lambda _{-}s}
\left|\sgm_2(\pm \delta e^{\lambda _{+}s},0)\right|^2ds.
\end{equation*}%
Let us also recall that $\beta$ is defined in~\eqref{eq:alpha_0}.

\begin{lemma}\label{Thm:after_tauh} Consider the solution to system~\eqref{eq:SDE_changed_coord1}--\eqref{eq:SDE_changed_coord2}
equipped with initial conditions~\eqref{eq:restart_at_tauh} satisfying~\eqref{eq:condition_for_theorem_along_unstable}.
If the second inequality in~\eqref{eqn: p_prop} holds, then
\begin{equation}
\label{eq:exit_through_unstabale}
 \Pp\{|Y_{\eps,1}(\tau_\eps)|=\delta\}\to 1,\quad \eps\to 0,
\end{equation}
\begin{equation}
 \tau_\eps+\frac{\alpha p}{\lambda_+}\log\eps\stackrel{\Pp}{\longrightarrow}
\frac{1}{\lambda_+}\log\delta,
\label{eq:tau_eps_convergence}
\end{equation}
\begin{equation}
\epsilon ^{-\beta}Y_{\epsilon,2}(\tau _{\epsilon })\overset{Law}{\longrightarrow }\theta.
\label{eq:asymptotics_for_Y_2_tau}
\end{equation}%
Moreover, if $\beta<1$, then the convergence in probability also holds.
\end{lemma}

A proof of this lemma is given in Section~\ref{sec:unstable_manifold}.

Now Theorem~\ref{Thm: main_small} follows from Lemmas~\ref{prop: eta_convergence} and~\ref{Thm:after_tauh}.
In fact, the strong Markov property and~\eqref{eq:exit_at_tauh_along_axis_1} imply
\[
\Pp\{\tau_\eps^U=\tauh+\tau_\eps(\delta)\}\to 1,\quad \eps\to 0,
\]
so that the asymptotics for $\tau_\eps^U$ is defined by that of $\tauh$ and $\tau_\eps(\delta)$. It is also clear that
one can set $\psi'_\eps=\eta^+_\eps$,
and $\phi'_\eps=\nabla g(\sgn(\eta^+_\eps)\delta e_1)Y_\eps(\tau_\eps)$, so that the limiting distribution of $\Theta'_\eps$ is given by
\begin{equation}
 \left(\eta^+_0,\ \nabla g(\sgn(\eta^+_0)\delta e_1)(\theta e_2),\ \frac{1}{\lambda_+}\log\frac{\delta}{|\eta_0^+|}\right),
\label{eq:limiting_distr_for_Theta_prime}
\end{equation}
where random variables $\eta_0^+$ and $\theta$ are defined in~\eqref{eq:eta_0_+} and~\eqref{eq:theta}

\section{Proof of Lemma~\ref{prop: eta_convergence} }\label{sec:close_to_saddle}

In this section we shall prove  Lemma~\ref{prop: eta_convergence} using several auxiliary lemmas. We start with some terminology.

\begin{definition}\rm
\label{def: Op} Given a family $(\xi_{\epsilon })_{\epsilon >0}$ of random 
variables or random vectors and a function $h:(0,\infty )\rightarrow (0,\infty )$ we say that $%
\xi_{\epsilon }=O_{\mathbf{p}}(h(\epsilon ))$ if for some $\eps_0>0$
distributions of $\left( \xi_{\epsilon}/h(\epsilon ) \right )_{0< \epsilon < \eps_0}$, form a tight family,
i.e.,\ for
any $\delta>0$ there is a constant $K_\delta>0$ such that 
\begin{equation*}
\mathbf{P}\left\{ |\xi_{\epsilon }|>K_{\delta }h(\epsilon )\right\} <\delta ,\quad 0<\epsilon <\epsilon _{0}.
\end{equation*}
\end{definition}

\begin{definition} \rm
A family of random variables or random vectors~$(\xi_\eps)_{\eps>0}$ is called slowly growing as $\eps\to 0$ (or just slowly growing) if
$\xi_\eps=O_{\Pp}(\eps^{-r})$ for all $r>0$. 
\end{definition}

Our first lemma estimates the martingale component of the solution of SDEs~\eqref{eq:SDE_changed_coord1} and~\eqref{eq:SDE_changed_coord2}.
 Let us define
\begin{eqnarray*}
S_{\epsilon }^{+}(T) &=&\sup_{t\leq  T }\left\vert
\int_{0}^{t}e^{-\lambda _{+}s}\sgm _{1}(Y_{\epsilon}(s))dW(s)\right\vert,\quad T>0, \\
S_{\epsilon }^{-}(T) &=&\sup_{t\leq T }\left\vert
\int_{0}^{t}e^{-\lambda _{-}(t-s)}\sgm_{2}(Y_{\epsilon}(s))dW(s)\right\vert,\quad T>0.
\end{eqnarray*}%

\begin{lemma}
\label{lemma: stoch_est}
Suppose $(\tau_\eps)_{\eps>0}$ is a family of stopping times (w.r.t.\ the natural
filtration of $W$). Then  
\[
S_{\epsilon }^{+}(\tau_\eps )=O_{\mathbf{P}}(1).
\]
If additionally $(\tau_\eps)_{\eps>0}$ is slowly growing, then $S_{\epsilon }^{-}(\tau_\eps )$ is
also slowly growing.
\end{lemma}

\begin{proof}
Let us start with the proof for $S_\eps^+$. Use BDG inequality (see~\cite[Theorem 3.3.28]{Karatzas--Shreve}) and It\^o's isometry to see that for every constant $K>0$,
\begin{align*}
\Pp \left \{   S_\eps^+ ( \tau_\eps ) > K \right \} & \leq \frac{1}{K^2} \mathbf{E} S_\eps^+ ( \tau_\eps )  \\
& \leq \frac{C_1}{K^2} \mathbf{E} \int_0^{\tau_\eps} e^{ -  2 \lambda_+ s } \sgm _1 ( Y_\eps(s))ds.
\end{align*}
Since $Y_\eps (t) = f (\x (t \wedge \tau_\eps^U ) )$, the process $t \mapsto \sgm _1 ( Y_\eps(t))$ is almost surely bounded. Hence, integrability of the exponential $ t \mapsto e^{-2 \lambda_+ t}$ implies that for any $\delta>0$, there is a $K_\delta>0$ such that 
\[
 \sup_{ \eps > 0 } \Pp \left \{   S_\eps^+ ( \tau_\eps ) > K_\delta \right \} \leq \delta,
\]
proving the first part of the lemma.

For the second part, fix $\delta>0$ and $r>0$. For every $0 < \rho < 2r$, there is $K_\rho >0$ and $\eps_0>0$ such that 
\[
\sup_{0 < \eps < \eps_0 } \Pp \left \{   \eps^\rho \tau_\eps > K_\rho  \right \} < \delta/2.
\]
Then, for an arbitrary $K>0$, $0< \eps < \eps_0$ and $0 < \rho < 2r$, it holds that
\begin{align*}
 \Pp \left\{ \eps^r S_\eps^- (\tau_\eps) > K \right \} &\leq \Pp \left\{ \tau_\eps > \eps^{-\rho} K_\rho \right \} + \Pp \left\{ \eps^r S_\eps^- (\tau_\eps) > K, \tau_\eps \leq \eps^{-\rho} K_\rho \right\}\\
&\leq \delta/2 + \sum_{k=1}^{ \lceil K_\rho \eps^{ -\rho}  \rceil } \Pp \left \{ \eps^{r} \sup_{ (k-1) \leq t < k} \left |  \int_{0}^{t}e^{-\lambda _{-}(t-s)}\sgm_{2}(Y_{\epsilon}(s))dW(s) \right | > K \right \}.
\end{align*}
In order to bound each probability in the last sum, proceed as for the other case:
\begin{align*}
&\Pp \left \{ \eps^{r} \sup_{ (k-1) \leq t < k} \left |  \int_{0}^{t}e^{-\lambda _{-}(t-s)}\sgm_{2}(Y_{\epsilon}(s))dW(s) \right | > K \right \} \\
& \leq \Pp \left \{ \eps^{r} e^{ - (k-1) \lambda_- } \sup_{ 0 \leq t < k} \left |  \int_{0}^{t}e^{\lambda _- s}\sgm_{2}(Y_{\epsilon}(s))dW(s) \right | > K \right \} \\
& \leq \frac{\eps^{2r} e^{ - 2 (k-1) \lambda_- }}{K^2} \mathbf{E} \int_{0}^ {k } e^{2 \lambda _- s} | \sgm_{2}(Y_{\epsilon}(s)) | ^2ds \\
& \leq \frac{\eps^{2r} C_2}{K^2},
\end{align*}
for some constant $C_2 >0 $. Hence, there is a constant $C_3 >0$ such that 
\[
 \Pp \left\{ \eps^r S_\eps^- (\tau_\eps) > K \right \} \leq \delta/2 + \frac{C_3}{K^2} \eps^{ 2r - \rho },
\]
which implies the result and finishes the proof.
\end{proof}

\begin{lemma}
\label{thm: estimate_non_linear} Suppose $Y_\eps$ is the solution of equations
 (\ref{eq:SDE_changed_coord1})--(\ref{eq:SDE_changed_coord2}) with initial conditions given by%
\begin{equation}
Y_{\epsilon,1 }(0)=\epsilon ^{\alpha }\chi _{\epsilon,1 }\quad \text{\rm  and }\quad Y_{\epsilon,2}(0)
=y_{2}+\epsilon ^{\alpha }\chi _{\epsilon,2 },  \label{eqn: Initial}
\end{equation}%
where distributions of random variables $(\chi _{\epsilon,1 })_{\epsilon >0}$ and $(\chi _{\epsilon,2
})_{\epsilon >0}$ form tight families. 
Let us fix any $R>0$ and denote $l_\eps= \tau_\eps^U\wedge(- \frac{\alpha}{\lambda_+}\log\eps + R)$ for $\eps>0$. Then
\[
\sup_{t\leq l_\eps}e^{-\lambda
t}|Y_{\epsilon,1 }(t)| = O_\Pp (\eps ^\alpha),
\]  and the family
\[
\left ( \eps^{-\alpha }\sup_{t\leq l_\eps}|Y_{\epsilon,2 }(t)-e^{-\lambda t }(y_{2}+\epsilon ^{\alpha }\chi
_{\epsilon,2})| \right )_{\eps>0}
\] is slowly growing.
\end{lemma}

\begin{proof} The tightness property implies that without loss of generality we can assume that $|\chi_{\epsilon,1 }|,|\chi_{\epsilon,2 }|<C$ for some constant $C>0$ and every $\epsilon >0$.

Let us fix $\gamma >0$. We can use Lemma~\ref{lemma: stoch_est} to take $c=c(\gamma
/3)>0$ such that 
\[
\mathbf{P}\{S_{\epsilon }^{+ }(l_\eps)>c\}<\gamma /2,
\]
and 
\[
\Pp\{ S_\eps^-(l_\eps)>c\eps^{-q}\}<\gamma/2,
\]
where $q$ is an arbitrary number satisfying $0<q<\alpha$. Let us introduce a constant $K=(3c)\vee C$ and
stopping times 
\begin{align*}
\beta _{+} &=\inf \left\{ t\ge 0:e^{-\lambda _{+}t}|Y_{\epsilon,1 }(t)|\geq
2K\epsilon ^{\alpha }\right\} , \\
\beta _{-} &=\inf \left\{ t\ge 0:|Y_{\epsilon,2 }(t)-e^{-\lambda
_{-}t}(y_2+\epsilon ^{\alpha }\chi _{\epsilon,2 })|\geq 2K\epsilon ^{\alpha - q}\right\} , \\
\beta&=\beta _{+}\wedge \beta_- \wedge l_\eps.
\end{align*}
We start with an estimate for $Y_{\epsilon,1 }$. Duhamel's principle 
for \eqref{eq:SDE_changed_coord1}, Theorem~\ref{Lemma: Def_H1&H2} and Lemma~\ref{lemma: stoch_est} imply that the estimate%
\begin{align}
\sup_{t\leq \beta }e^{-\lambda _{+}t}|Y_{\epsilon,1 }(t)| &\leq \epsilon
^{\alpha }K+K_{1}\int_{0}^{\beta }e^{-\lambda _{+}s}Y_{\epsilon,1
}(s)^{2}|Y_{\epsilon,2 }(s)|ds+K_{2}\frac{\epsilon ^{2}}{\lambda _{+}} +\epsilon S_{\epsilon }^{+}(\beta )  \notag \\
&\leq \epsilon ^{\alpha }K+K_{1}\int_{0}^{\beta }e^{-\lambda
_{+}s}Y_{\epsilon,1 }(s)^{2}|Y_{\epsilon,2 }(s)|ds+K_{2}\frac{\epsilon ^{2}}{\lambda _{+}}  +\epsilon \frac{K}{3} \label{eqn: x_k_espression}
\end{align}%
holds with probability at least $1-\gamma /2$. We analyze each term in the
RHS of equation \eqref{eqn: x_k_espression}.

Let us start with the integral in \eqref{eqn: x_k_espression}. For $%
s\leq \beta $, we see that 
\begin{align}
Y_{\epsilon,1 }(s)^{2}|Y_{\epsilon,2 }(s)| &\leq 4K^{2}\epsilon ^{2\alpha
}e^{2\lambda _{+}s}\left( |Y_{\epsilon,2 }(s)-e^{-\lambda
_{-}s}(y_2+\epsilon ^{\alpha }\chi _{\epsilon,2 })|+e^{-\lambda
_{-}s}|y_2+\epsilon ^{\alpha }\chi _{\epsilon ,2}|\right)   \notag \\
&\leq 8K^{3}\epsilon ^{3\alpha - q } e^{2\lambda _{+}s}+4K^{2}\epsilon ^{2\alpha
}e^{(2\lambda _{+}-\lambda _{-})s}(|y_2|+\epsilon ^{\alpha }C).
\notag 
\end{align}
Therefore,
\begin{align} \notag
K_{1}\int_{0}^{\beta }e^{-\lambda _{+}s}Y_{\epsilon,1 }(s)^{2} |Y_{\epsilon,1 }(s)|ds 
&\leq \frac{8K^{3}K_{1}e^{\lambda_+R}}{\lambda _{+}}  \epsilon ^{2\alpha - q } \\ 
&+4K_{1}K^{2}  \epsilon ^{2\alpha }(|y_{2}|+\epsilon ^{\alpha} C)\int_0^\beta e^{(\lambda_+ - \lambda _{-}) s} ds  \notag \\
&\leq K\epsilon ^{\alpha }/12 +  5K_{1}K^{2}  \epsilon ^{2\alpha }|y_{2}| \int_0^\beta e^{(\lambda_+ - \lambda _{-}) s} ds \label{eqn; bound_x} 
\end{align}
for all $\eps >0$ small enough. Notice that this is a rough estimate, the constants on the r.h.s.\ are not optimal but sufficient for our purposes. This also applies to some
other estimates in this proof.

Let us estimate the integral on the r.h.s.\  of \eqref{eqn; bound_x}. When $\lambda_+ > \lambda_-$, the integral is bounded by
\[
\frac{1}{\lambda_+ - \lambda_-} e^{ (\lambda_+ - \lambda_- ) \beta} \leq 
\frac{e^{(\lambda_+ - \lambda_-)R}}{\lambda_+ - \lambda_-} \eps^{ - \alpha+\alpha \lambda_- / \lambda_+};
\]
if $\lambda_+ < \lambda_-$, then the integral on the r.h.s of~\eqref{eqn; bound_x}
is bounded by~$(\lambda_- - \lambda_+)^{-1}$; if $\lambda_+=\lambda_-$, then the integral is bounded by 
$2\alpha\lambda_+^{-1}|\log \eps|$. Hence, for some constant $K_{\lambda_+,\lambda_-}>0$ and $\epsilon >0$ small enough,
\begin{align} 
\notag
K_{1}\int_{0}^{\beta }e^{-\lambda _{+}s}Y_{\epsilon,1 }(s)^{2}|Y_{\epsilon,2 }(s)|ds 
&\leq K\epsilon ^{\alpha }/12 + K_{\lambda_+,\lambda_-} \eps^{2\alpha -\alpha(1-\lambda_-/\lambda_+)^+} |\log\eps| \\
&\leq K \eps^\alpha /6 \label{eqn: ineq_1}.
\end{align}%
Also, for $\epsilon >0$ small enough, 
\begin{equation}
K_{2}\epsilon ^{2}/\lambda _{+}+\eps K/3 <K\epsilon ^{\alpha }/2.  \label{eqn: ineq_2}
\end{equation}%
From~\eqref{eqn: x_k_espression}, \eqref{eqn: ineq_1} and \eqref{eqn: ineq_2} we get that for all $\eps>0$ small enough, the event   
\begin{equation*}
A=\left\{ \sup_{t\leq \beta }e^{-\lambda _{+}t}|Y_{\epsilon,1 }(t)|\leq
5K\epsilon ^{\alpha }/3\right\} 
\end{equation*}%
is such that $\Pp (A)>1-\gamma /2$.

Let us now consider $Y_{\epsilon,2 }(t)$ and denote 
\[
Z_{\epsilon }(t)=Y_{\epsilon,2 }(t)-e^{-\lambda_{-}t}(y_{2}+\epsilon ^{\alpha }\chi _{\epsilon,2 }).
\]
Duhamel's principle for $Y_{\eps,2}$, the definition of $\beta $, Theorem~\ref{Lemma: Def_H1&H2} and  Lemma~\ref{lemma: stoch_est}
imply that the inequalities
\begin{align}
\notag
\sup_{t\leq \beta }|Z_{\epsilon }(t)| &\leq K_1 \sup_{t \leq \beta}\int_{0}^{t}e^{-\lambda _{-}(t-s)}|Y_ {\eps,1}(s)|^\aone |Y_{\eps,2}(s)|^\atwo ds+K_2 \eps^{2} / \lambda_-  
+\epsilon S_{\epsilon }^{-}(\beta )  \\
&\leq K_1 \sup_{t \leq \beta}\int_{0}^{t}e^{-\lambda _{-}(t-s)}|Y_ {\eps,1}(s)|^\aone |Y_{\eps,2}(s)|^\atwo ds \notag \\
& \hspace{1.5 in} + \eps^{\alpha-q} \left(  K_2 \eps^{2-\alpha+q} / \lambda_- +\epsilon^{1-\alpha+q} S_{\epsilon }^{-}(\beta ) \right)
\notag \\
&\leq 2^\aone \eps^{\alpha \aone} K^\aone K_1 \sup_{t \leq \beta} e^{-\lambda_- t} \int_{0}^{t} e^{(\lambda _{-} + \aone \lambda_+ )s}|Y_{\eps,2}(s)|^\atwo ds + \eps^{\alpha-q}K/2  \label{eqn: y_k_espression}
\end{align}%
hold with probability at least $1-\gamma /2$ and for all $\epsilon >0$ small enough.
We analyze the integral term in \eqref{eqn: y_k_espression}. Note that, from the definition of $\beta$, and the inequality $(a+b)^r \leq 2^{r-1} ( a^r + b^r )$ we have that for any $t \leq \beta$ and any $\eps >0$ small enough,
\begin{align*}
|Y_{\eps,2}(t)|^\atwo &\leq 2^{\atwo-1} Z_\eps(t)^\atwo + 2^{\atwo -1 }e^{-\atwo \lambda_- t}|y_2 + \eps^\alpha \chi_{\eps,2}|^\atwo \\
& \leq 2^{2\atwo-1} K^\atwo \eps^{ (\alpha -q)\atwo } + 2^{2( \atwo -1)}e^{-\atwo \lambda_- t}|y_2|^\atwo \\
& \hspace{2.255 in}+2^{2( \atwo -1)}\eps^{\alpha \atwo }e^{-\atwo \lambda_- t}|\chi_{\eps,2}|^\atwo \\ 
&\leq \eps^{\atwo ( \alpha - q)} 2^{2 (\atwo -1)} \left(  2K^\atwo + \eps^{q \atwo} |\chi_{\eps,2}|^\atwo \right) + 2^{2(\atwo-1)} e^{- \atwo \lambda_- t} |y_2|^\atwo.
\end{align*}
Hence there is a constant $K_\alpha>0$ such that
\[
|Y_{\eps,2}(t) | ^\atwo \leq \eps^{\atwo ( \alpha - q)} K_\alpha + K_\alpha e^{-\atwo \lambda_- t},\quad t\leq \beta.
\] 
Using the last inequality, the definition of $\beta$, and the fact  $\aone \lambda _+ - (\atwo - 1)\lambda_-=0$ from Theorem~\ref{Lemma: Def_H1&H2}, we get
\begin{align}
\notag
\eps^{\alpha \aone} e^{-\lambda_- t} & \int_{0}^{t} e^{(\lambda _{-} + \aone \lambda_+ )s}| Y_{\eps,2}(s)|^\aone ds \\
\notag &\leq \eps^{\alpha( \aone + \atwo)} e^{\lambda_+ \aone \beta} \frac{K_\alpha \eps^{-q\atwo}}{\lambda_- + \aone \lambda_+}
 + K_\alpha \eps^{\alpha \aone }\int_{0}^{t} e^{(\aone \lambda _+ - (\atwo - 1)\lambda_- )s}ds \\
& \leq \eps^{ (\alpha - q) \atwo} \frac{K_\alpha e^{\lambda_+\alpha_1^- R}}{\lambda_- + \aone \lambda_+} + K_\alpha \eps^{\alpha \aone } \beta.
\label{eqn: integral_y_1}
\end{align}
Again, from Theorem~\ref{Lemma: Def_H1&H2} we know that $\aone \geq 1$ and $\atwo \geq 2$ which together with~\eqref{eqn: integral_y_1} imply that for all $\eps >0$ small enough
\begin{equation}
2^\aone \eps^{\alpha \aone} K^\aone K_1 \sup_{t \leq \beta} e^{-\lambda_- t} \int_{0}^{t} e^{(\lambda _{-} + \aone \lambda_+ )s}|Y_{\eps,2}(s)|^\atwo ds \leq K \eps^{\alpha - q} /6.
\label{eqn: y_integral_final}
\end{equation}
Using~\eqref{eqn: y_integral_final} and~\eqref{eqn: y_k_espression} we conclude that
the event 
\begin{equation*}
B=\left\{ \sup_{t\leq \beta }|Y_{\epsilon,2 }(t)-e^{-\lambda
_{-}t}(y_{2}+\epsilon ^{\alpha }\chi _{\epsilon,2 })|\leq 2K\epsilon ^{\alpha-q
}/3\right\} 
\end{equation*}%
is such that $\mathbf{P}(B) \geq 1-\gamma /2,$ for all $\epsilon >0$ small
enough.

The proof will be complete once we show that $\beta =l_\eps$ with probability at least $1-\gamma $. The latter is a consequence of the following
chain of inequalities that hold for all $\eps>0$ small enough: 
\begin{align*}
\mathbf{P}\{\beta _{+}\wedge \beta _{-} \leq l_\eps \}&\leq \mathbf{P}\left( \{\beta _{+}\wedge \beta _{-}\leq l_\eps\} \cap A\cap B \right)+\mathbf{P}(A^{c})+\mathbf{P}(B^{c}) \\
&\leq \mathbf{P}\left( \{\beta _{+}\wedge \beta _{-}\leq l_\eps \} \cap A\cap B\right)+\gamma \\
&\leq \mathbf{P}\left(  \{\beta _{+}\leq \beta _{-}\wedge l_\eps\} \cap A\right)+\mathbf{P}\left( \{\beta _{-}\leq \beta _{+}\wedge l_\eps\} \cap B \right)+\gamma \\
&=\mathbf{P}\{2\leq 5/3\}+\mathbf{P}\{2\leq 2/3\}+\gamma =\gamma .
\end{align*} 
\end{proof}

\medskip

Let us now analyze the evolution of the process $Y_\eps$ up to time $\tauh \wedge \tau_\eps^U$. We start with an application of Duhamel's principle: 
\begin{align}
Y_{\epsilon,1 }(t) &=e^{\lambda _{+}t}Y_{\epsilon,1 }(0)+\int_{0}^{t}e^{\lambda
_{+}(t-s)}H_{1}(Y_{\epsilon }(s),\eps)ds+\epsilon e^{\lambda _{+}t}%
{N}_{\epsilon }^{+}(t),  \label{eqn: x_duhamel} \\
Y_{\epsilon,2 }(t) &=e^{-\lambda _{-}t}Y_{\epsilon,2}(0)+\int_{0}^{t}e^{-\lambda _{-}(t-s)} H_{2}(Y_{\epsilon
}(s),\eps)ds+\epsilon {N}_{\epsilon }^{-}(t),  \label{eqn: y_duhamel}
\end{align}%
where ${N}_{\epsilon }^{\pm}(t)$ are defined by
\begin{align}
{N}_{\epsilon }^{+}(t) &=\int_{0}^{t}e^{-\lambda _{+}s}\sgm
_{1}(Y_{\epsilon }(s))dW(s),\notag \\
{N}_{\epsilon }^{-}(t) &=\int_{0}^{t}e^{-\lambda _{-}(t-s)}\sgm
_{2}(Y_{\epsilon }(s))dW(s). \label{eqn: def_N-}
\end{align}

\begin{lemma}
\label{lemma: y_convergence}
\[
\sup_{t \leq \hat{\tau_\eps} }
|Y_{\eps,2}(t) - e^{ -\lambda_- t} y_2|=O_{\Pp}(\eps^{\alpha p}).
\] 
\end{lemma}
\begin{proof}
Duhamel's principle, Theorem~\ref{Lemma: Def_H1&H2},  and the definition of~$\tauh$
imply that for some $K>0$,
\begin{align*}
|Y_{\eps,2} (t) - e^{ -\lambda_- t} y_2| &\leq \eps^\alpha |\chi_{\eps,2}| +\int_0^t e^{- \lambda_- (t - s)} \left(K_1|Y_{\eps,1} (s)|Y_{\eps,2}^2 (s) + K_2\eps^2 \right) ds + \eps S_\eps^- (t)\\
&\leq \eps^\alpha |\chi_{\eps,2}| + K \eps^{ \alpha p} + \eps^{\alpha p} \left( \eps^{1-\alpha p } S_\eps^- (\tauh) \right)
\end{align*}
for any $t \in (0, \tauh )$. The result follows since by Lemma~\ref{lemma: stoch_est}
the r.h.s.\ is $O_\Pp (\eps^{\alpha p})$
\end{proof}

As a simple corollary of this lemma, the first statement in Theorem~\ref{prop: eta_convergence} follows:
\begin{corollary} As ${\eps \to 0}$,
\label{cor: tau_eps<infty}
\[
 \Pp \{  \tau_\eps^U < \tauh \}\to 0.
\]In particular,~\eqref{eq:exit_at_tauh_along_axis_1}  holds true.
\end{corollary}

\begin{lemma}
\label{lemma: ito_convergence}
Let
\[
{N}_{0}^{+}(t) =\int_{0}^{t}e^{-\lambda _-s}\sgm_{1}(0,e^{-\lambda _{-}s}y_{2})dW.
\]
Then
\[
\sup_{t \leq \tauh} |{N}_\eps^+ (t) - {N}_0^+ (t) | \overset{L^2} {\longrightarrow}0,
 \quad \eps \to 0.
\]
\end{lemma}  
\begin{proof} BDG inequality implies that for some constants $C_1,C_2>0$,
\begin{align}\notag 
\mathbf{E}\sup_{t\leq \tauh}|{N}_{\epsilon
}^{+}(t)-{N}_{0}^{+}(t)|^{2} &\leq C_{1}\mathbf{E}\int_{0}^{\hat%
\tau _{\epsilon }}e^{-2\lambda _{+}s}|\sgm_1 (Y_{\epsilon,1}(s),Y_{\eps,2}(s))- (0,e^{-\lambda _{-}s}y_{2})|^{2}ds \\
&\leq C_2\mathbf{E}\sup_{t\leq \hat{\tau} _{\epsilon }}|\sgm_1
(Y_{\epsilon,1 }(s),Y_{\epsilon,2 }(s))-\sgm_1 (0,e^{-\lambda _{-}s}y_{2})|^{2}.
\label{eq:Gaussian_approx}
\end{align}%
From Lemma~\ref{lemma: y_convergence} and the definition 
of~$\tauh$, it follows that 
\begin{equation}
\sup_{t\leq\hat\tau_\eps}\left |(Y_{\eps,1}(t),Y_{\eps,2}(t))-(0,e^{-\lambda_- t}y_2)\right|=O_{\Pp}(\eps^{\alpha p} ).
\label{eq:approx_along_stable_manifold}
\end{equation}
The desired convergence follows now from \eqref{eq:Gaussian_approx}, 
\eqref{eq:approx_along_stable_manifold}, and the boundedness and Lipschitzness of $\sgm_1$.
\end{proof}

We are now in position to give the first rough asymptotics for the time~$\tauh$. From now on we restrict ourselves to the 
 event $\{\tau_\eps^U > \tauh\}$ since due to Corollary~\ref{cor: tau_eps<infty} its probability is arbitrarily high.
\begin{lemma}
\label{prop: tau_bar_convergence} As $\eps\to0$,
\[
\Pp  \left \{  \tauh > -\frac{\alpha}{\lambda_+} \log \eps \right \} \to 0.
\]
\end{lemma}
\begin{proof}
Let $u_\eps$ be the solution to the following SDE:
\begin{align*}
du_\eps(t) &= \lambda_+ u_\eps(t)dt + \eps \sgm_1 (Y_\eps(t))dW(t),\\
u_\eps(0)&=\eps^\alpha \chi_{\eps,1}.
\end{align*}
Let us take $\delta_0\in(0,1)$ to be specified later and consider the following stopping time 
\[
\tauw=\inf \left \{ t: |u_\eps(t)|= \eps^{\alpha \delta_0} \right \}.
\] 
Duhamel's principle for $u_\eps$ writes as 
\begin{align*}
u_\eps(t) &= \eps^\alpha e^{\lambda_+ t}\chi_{\eps,1} + \eps e^{\lambda_+ t} {N}_\eps^+(t) \\
&= \eps^\alpha e^{\lambda_+ t } \wt{\eta}_\eps (t),
\end{align*}
with 
\begin{equation}
\wt{\eta}_\eps (t)= \chi_{\eps,1} + \eps^{1-\alpha} {N}_\eps^+ (t).
\label{eq:eta-tilde} 
\end{equation}
Hence, the definition of $\tauw$ implies $\eps^{\alpha \delta_0} = \eps^{\alpha} e^{\lambda_+ \tauw} |\wt{\eta}_\eps (\tauw)|$, so that 
\[
\tauw = -\frac{\alpha}{\lambda_+} (1-\delta_0 ) \log\eps -\frac{1}{\lambda_+} \log |\wt{\eta}_\eps (\tauw)|.
\]
Due to~\eqref{eq:eta-tilde} and Lemma~\ref{lemma: ito_convergence}, the distributions of $\frac{1}{\lambda_+} \log |\wt{\eta}_\eps (\tauw)|$ form a tight family. Therefore,
\begin{equation}
\label{eq:tauh_grows}
\lim_{\eps \to 0} \Pp \left \{  \tauw > -(1-\delta_0^2)\frac{\alpha}{\lambda_+} \log \eps \right \}=0.
\end{equation}
This fact allows us to use Lemma~\ref{thm: estimate_non_linear} to estimate $Y_\eps$
up to $\tauh \wedge \tauw$.
From~\eqref{eqn: x_duhamel}, the difference $\Delta_\eps=Y_{\eps,1}-u_\eps$ is given by
\[
\Delta_\eps (t) = e^{\lambda_+ t} \int_0^t e^{ -\lambda_+ s}  H_1(Y_\eps(s),\eps ) ds. 
\]
We can use~\eqref{eq:tauh_grows} to justify the application of Lemma~\ref{thm: estimate_non_linear} up to time $\tauh \wedge \tauw$. Then, we combine Theorem~\ref{Lemma: Def_H1&H2}, Lemma~\ref{thm: estimate_non_linear}, and the definition of $\tauh$ to see that 
\begin{align*}
\sup_{t\leq \tauh \wedge \tauw} e^{- \lambda_+ t} |H_1(Y_{\epsilon}(t), \eps)|&\leq K_1\sup_{t\leq \tauh \wedge \tauw} \left( \left(e^{-\lambda_+ t} |Y_{\eps,1}(t)|\right) |Y_{\eps,1}(t)|\cdot|Y_{\eps,2}(t)|\right) + K_2 \eps^2 \\
&=O_\Pp \left( \eps^{\alpha + \alpha p }  \right)\\
\end{align*}
and 
\[
e^{\lambda_+ \tauh \wedge \tauw } = O_\Pp \left ( \eps^{ -\alpha (1- \delta_0^2)  } \right).
\]
These two estimates together with~\eqref{eq:tauh_grows} imply 
\[
\sup_{t\leq \tauh \wedge \tauw } | \Delta_\eps (t) | = O_\Pp \left (  \eps^{ \alpha (p+\delta_0^2)  } |\log\eps|\right).
\]
On one hand,~\eqref{eq:tauh_grows} implies
\[ 
\Pp\left(\left \{  \tauh > -\frac{\alpha}{\lambda_+} \log 
\eps \right \} \cap \{\tauh \leq \tauw \}\right)\to 0.
\] 
On the other hand, if  $ \tauh > \tauw$ then
\[
|Y_{\eps,1} (\tauw)| = \left| \eps^{ \alpha \delta_0 } + O_\Pp ( \eps ^{ \alpha ( p + \delta _0^2 )} |\log\eps|) \right|,
\]
and
\[
 |Y_{\eps,1}(\tauw)|< \eps^{\alpha p}. 
\]
These relations contradict each other for sufficiently small $\eps$ if we choose $\delta_0 < p$.
So, this choice of $\delta_0$ guarantees that
 $\Pp \left \{  \tauh > \tauw \right \} \to 0$ implying the result. 
\end{proof}

\bigskip

\begin{proof}[Proof of Lemma~\ref{prop: eta_convergence}] 
Recall that we work on the high probability event $\{  \tauh < \tau_\eps^U \}$.  Hence, for each $\epsilon >0$, we have
the identity 
\begin{equation*}
\epsilon ^{\alpha p}=\epsilon ^{\alpha }e^{\lambda _{+}\hat{\tau}_{\epsilon }%
}|\eta _{\epsilon }^{+}|.
\end{equation*}%
Solving for $\tauh$ and then plugging it back into $%
Y_{\epsilon,1 }$, we get%
\begin{align}
\hat{\tau}_{\epsilon } &=-\frac{\alpha }{\lambda _{+}}(1-p)\log \epsilon -%
\frac{1}{\lambda _{+}}\log |\eta _{\epsilon }^{+}|,
\label{eqn: tau_explicit} \\
Y_{\epsilon,1 }(\hat\tau _{\epsilon }) &=\epsilon ^{\alpha p} \sgn(\eta
_{\epsilon }^{+}).  \notag
\end{align}
Using this information we are in position to get the asymptotic behavior of the random variables $\eta
_{\epsilon }^{\pm }$. First, from relation \eqref{eqn: x_duhamel} we get
\begin{equation}
\eta _{\epsilon }^{+}=\chi_{\epsilon,1 }+\epsilon ^{-\alpha }\int_{0}^{
{\tauh}}e^{-\lambda _{+}s}H_{1}(Y_{\epsilon
}(s),\eps)ds+\epsilon ^{1-\alpha }{N}_{\epsilon }^{+}(\hat{\tau _{\epsilon }}%
). \label{eqn: eta+}
\end{equation}
Using \eqref{eqn: tau_explicit} in \eqref{eqn: y_duhamel} we get%
\begin{align}
\eta _{\epsilon }^{-} &=|\eta _{\epsilon }^{+}|^{\lambda _{-}/\lambda
_{+}}(y_{2}+\epsilon ^{\alpha }\chi _{\epsilon,2 })+|\eta _{\epsilon
}^{+}|^{\lambda _{-}/\lambda _{+}}\int_{0}^{\hat{\tau} _{\epsilon }%
}e^{\lambda _{-}s} H_{2}(Y_{\epsilon }(s),\eps)ds  \notag \\
&+\epsilon ^{1-\alpha (1-p)\lambda _{-}/\lambda _{+}}{N}_{\epsilon
}^{-}(\hat{\tau} _{\epsilon }).  \label{eqn: eta-}
\end{align}%

The main part of the proof is based on representations \eqref{eqn: tau_explicit}--\eqref{eqn: eta-}.

Lemma~\ref{prop: tau_bar_convergence} allows us to use the estimates established in
Lemma~\ref{thm: estimate_non_linear} up to time~$\tauh$. In particular, now we can
conclude that the family
\begin{equation}
\left( \eps^{-\alpha} \sup_{t \leq \tauh} |Y_{\eps,2} (t) - e^{-\lambda_- t} y_2| \right)_{\eps>0}
\label{eq:y_slowly_growing} 
\end{equation}
is slowly growing thus improving Lemma~\ref{lemma: y_convergence}.

To obtain the desired convergence for $\eta_\eps^+$, we analyze the r.h.s.\ of~\eqref{eqn: eta+}
term by term. The convergence of the first term was one of our assumptions. For the second one,
we need to estimate $H_1(Y_\eps,\eps)$. 
 Using Lemma~\ref{thm: estimate_non_linear}, the boundness of~$Y_{\eps,2}$ and the definition of  
$\tauh$, we see that
\begin{equation}
\sup_{t\leq \hat{\tau}_\eps} e^{-\lambda_+ t} Y_{\eps,1}^2 (t) |Y_{\eps,2} (t)| = O_\Pp (\eps^{\alpha+\alpha p} ). \label{eqn: eta+_stoch}
\end{equation}
This estimate and Theorem~\ref{Lemma: Def_H1&H2} imply that
\begin{align}
\eps^{-\alpha} \int_0^{\tauh} e^{-\lambda_+ s}
H_1(Y_\eps(s),\eps)ds &\leq K_1 \eps^{-\alpha} \int_0^{\tauh} e^{-\lambda_+ s} Y_{\eps,1}^2 (s) |Y_{\eps,2}(s)|ds 
 + \frac{K_2}{\lambda_+} \eps^{2-\alpha} \notag \\
& =O_\Pp ( \eps^{\alpha p} |\log\eps| ).\notag
\end{align}
Let us estimate the third term in \eqref{eqn: eta+}. We can use the last estimate along with \eqref{eqn: eta+} and Lemma~\ref{lemma: ito_convergence} to conclude that 
the distributions of positive part of $\lambda_+^{-1} \log | \eta_\eps^+|$ form a tight family.
Therefore, \eqref{eqn: tau_explicit} implies that 
\[
\tauh  \overset{\Pp}{\to} \infty,\quad \eps \to 0. 
\]
Combined with It\^o isometry  and Lemma~\ref{lemma: ito_convergence}, this implies
\[
{N}_{\epsilon }^{+}(\hat\tau _{\epsilon })\overset{L^{2}}{%
\longrightarrow }N^+,\quad\epsilon \rightarrow 0,
\]
which completes the analysis of $\eta_\eps^+$ and, due to~\eqref{eqn: tau_explicit}, of $\tauh$.

\medskip

To obtain the convergence of $\eta_\eps^-$, we study~\eqref{eqn: eta-}.  
Combining \eqref{eq:y_slowly_growing}, the inequality 
\[
|Y_{\eps,1}(t)| Y_{\eps,2}^2 (t) \leq 2 |Y_{\eps,1} (t)| \left( |Y_{\eps,2}(t)-e^{-\lambda_- t}y_2|^2 + e^{ -2 \lambda_- t} y_2^2 \right ),
\] and the definition of $\tauh$ we see that for any $q\in(0,\alpha p)$,
\[
\sup_{t \leq \hat\tau_\eps }e^{\lambda_- t } |Y_{\eps,1}(t)| Y_{\epsilon,2 }^2(t) =O_\Pp \left( \eps^{\alpha p + \alpha - q} e^{\lambda_- \tauh }+\epsilon ^{\alpha p } \right ).
\]
Hence, as a consequence of Theorem~\ref{Lemma: Def_H1&H2} and~\eqref{eqn: tau_explicit} we have  
\begin{align*}
\int_{0}^{\tauh}e^{\lambda _{-}s}  H_{2}(Y_{\epsilon }(s),\eps)ds &=O_\Pp \left(\left(\epsilon ^{\alpha p -q + \alpha }e^{\lambda _{-}\tauh} + \epsilon ^{\alpha p }\right)|\log \epsilon |\right) \\
&=O_{\mathbf{P}} \left(\left(\epsilon ^{\alpha (1-(1-p)\lambda _{-}/\lambda _{+}) + (\alpha p - q)} + \epsilon ^{\alpha p}\right)|\log \epsilon | \right).
\end{align*}%
Combining this and Lemma~\ref{lemma: stoch_est} in \eqref{eqn: eta-} we obtain
\begin{align*}
\eta _{\epsilon }^{-} &=|\eta _{\epsilon }^{+}|^{\lambda _{-}/\lambda
_{+}}y_2+O_{\mathbf{P}}(\eps^\alpha)+O_{\mathbf{P}} \left(\left(\epsilon ^{\alpha (1-(1-p)\lambda _{-}/\lambda _{+}) + (\alpha p - q)}+\epsilon ^{\alpha p}\right)|\log \epsilon | \right) \\
& + O_{\mathbf{P}}\left(\epsilon ^{1-\alpha (1-p)\lambda _{-}/\lambda _{+} - q}\right)
\end{align*}
which finishes the proof of Lemma~\ref{prop: eta_convergence} by choosing $q$ small enough.
\end{proof}

\section{Proof of Lemma~\ref{Thm:after_tauh}}\label{sec:unstable_manifold}

Consider the solution to system~\eqref{eq:SDE_changed_coord1}--\eqref{eq:SDE_changed_coord2}
equipped with initial conditions~\eqref{eq:restart_at_tauh} satisfying~\eqref{eq:condition_for_theorem_along_unstable}.
Let us restrict the analysis to the arbitrary high probability event 
\[
\{ |\eta _{\epsilon }^{\pm }|\leq K_{\pm } \},
\]
for some constants $K_\pm>0$.

\begin{lemma}
\label{lemma: y_after_global} Let $p\in (0,1)$ satisfy~\eqref{eqn: p_prop}, and
let $(t_\eps)_{\eps>0} $ be a slowly growing family of stopping times. Consider $t_\eps'=t_\eps \wedge \tau_\eps^U$, then for any $\gamma >0$,
\begin{equation*}
\lim_{\epsilon \rightarrow 0}\mathbf{P}\left\{\sup_{t\leq t_\eps' }|Y_{\epsilon,2
}(t)|\leq (K_{-}+\gamma )\epsilon ^{\alpha (1-p)\lambda _{-}/\lambda
_{+}}\right\}=1.
\end{equation*}
\end{lemma}

\begin{proof}
Let $\gamma >0$. We recall that  $N_\eps^-$ is defined in~\eqref{eqn: def_N-} and introduce the process
\begin{equation}
M_{\epsilon }(t) ={N}_{\epsilon }^{-}(t)+\epsilon
\int_{0}^{t}e^{-\lambda_- (t-s)}\Psi _{2}(Y_{\epsilon }(s))ds, \label{eqn: M_def}
\end{equation}
where $\Psi_2$ was introduced in Theorem~\ref{Lemma: Def_H1&H2},
and the stopping time 
\begin{equation*}
\beta _{\epsilon }=\inf \left\{t:|Y_{\epsilon,2 }(t)|>(K_{-}+\gamma )\epsilon
^{\alpha (1-p)\lambda _{-}/\lambda _{+}}\right\}.
\end{equation*}%
Using the fact that $Y_{\epsilon,1 }$ is bounded, it is easy to see that
there is a constant $%
K_{\lambda _{-}}$ independent of $t$, so that for any $t\leq \beta_\eps\wedge t_\eps'$,
we have
\begin{equation*}
\int_{0}^{t}e^{-\lambda _{-}(t-s)}|Y_{\epsilon,1 }(s)|Y_{\epsilon,2}^2(s)ds\leq K_{\lambda _{-}}\epsilon ^{ 2\alpha
(1-p)\lambda _{-}/\lambda _{+}}.
\end{equation*}%
This estimate, along with Duhamel's principle and Theorem~\ref{Lemma: Def_H1&H2} implies that for some constant $C>0$ and any $t \leq \beta_\epsilon \wedge t_\eps'$,%
\begin{align*}
|Y_{\epsilon,2 }(t)| &\leq \epsilon ^{\alpha (1-p)\lambda _{-}/\lambda
_{+}}|\eta^-_{\epsilon }|+K_1\int_{0}^{t}e^{-\lambda _{-}(t-s)}|Y_{\eps,1}(s)|Y_{\epsilon,2 }^2 (s)ds+\epsilon \sup_{t \leq \beta_\eps }|M_{\epsilon }(t)| \\
&\leq \epsilon ^{\alpha (1-p)\lambda _{-}/\lambda _{+}}K_{-}+C \epsilon ^{2\alpha (1-p)\lambda _{-}/\lambda _{+}}+\epsilon
\sup_{t\leq \beta _{\epsilon }}|M_{\epsilon }(t)|.
\end{align*}%
Hence, using Lemma \ref{lemma: stoch_est} to estimate $M_\eps$, we obtain that
\begin{align*}
\mathbf{P}\{\beta _{\epsilon } < t_\eps' \} &=\mathbf{P}\left\{\sup_{t \leq \beta_{\epsilon } \wedge t_\eps' } |Y_{\epsilon,2 }(t)| \geq (K_{-}+\gamma )\epsilon ^{\alpha
(1-p)\lambda _{-}/\lambda _{+}} \right\} \\
&\leq \mathbf{P}\left\{C\epsilon ^{\alpha (1-p)\lambda_{-}/ \lambda _{+}}+\epsilon ^{1-\alpha (1-p)\lambda _{-}/\lambda
_{+}}\sup_{t\leq \beta _{\epsilon }}|M_{\epsilon }(t)|\geq \gamma
 \right\}
\end{align*}%
converges to $0$ as $\eps \rightarrow 0$ proving the lemma.
\end{proof}

\begin{lemma}
\label{lemma: x_after_global}Under the assumptions of lemma \ref{lemma:
y_after_global}, for any $\rho \in (0,\frac{\alpha p}{%
\lambda _{+}}] $, $\gamma >0$, and $C>0$, define $\rho_\eps~=(-\rho \log \epsilon+C )\wedge \tau_\eps^U $. Then,  we have
\begin{equation*}
\lim_{\epsilon \rightarrow 0}\mathbf{P}\left\{\sup_{t\leq\rho_\eps}|Y_{\epsilon,1 }(t)|e^{-\lambda _{+}t}\leq (1+\gamma )\epsilon ^{\alpha
p}\right\}=1.
\end{equation*}
\end{lemma}

\begin{proof}
Define the stopping time 
\begin{equation*}
\beta _{\epsilon }=\inf \left\{t:|Y_{\epsilon,1 }(t)|e^{-\lambda _{+}t}\geq
(1+\gamma )\epsilon ^{\alpha p}\right\}.
\end{equation*}%
As a consequence of Duhamel's principle and Theorem~\ref{Lemma: Def_H1&H2} we get the bound 
\begin{align*}
\sup_{t\leq \beta _{\epsilon }\wedge \rho_\eps }  |Y_{\epsilon,1}(t)|e^{-\lambda _{+}t} 
\leq &\epsilon ^{\alpha p}+K_{1}\int_{0}^{\beta _{\epsilon }\wedge \rho_\eps}e^{-\lambda _{+}s}Y_{\epsilon,1}^{2}(s)|Y_{\epsilon,2}(s)|ds\\
&\quad +\epsilon ^{2}K_{2}\lambda _{+}^{-1} +\epsilon S_{\epsilon }^{+}(\beta_\eps ). 
\end{align*}
This estimate together with Lemma~\ref{lemma: y_after_global}, Lemma~~\ref{lemma: stoch_est} and the definition of $\rho_\eps$ implies that for any small $\delta>0$ we can find a constant $K>0$, so that with probability bigger than $1-\delta$, the inequalities
\begin{align*}
\sup_{t\leq \beta _{\epsilon }\wedge \rho_\eps }  |Y_{\epsilon,1}(t)|e^{-\lambda _{+}t}
&\leq \epsilon ^{\alpha p}+K\epsilon ^{\alpha p+\alpha (1-p)\lambda
_{-}/\lambda _{+}}(\beta_\eps\wedge \rho_\eps)+K\eps \\
&\leq \epsilon ^{\alpha p}(1+2K\rho \epsilon ^{\alpha (1-p)\lambda
_{-}/\lambda _{+}}|\log \epsilon |+K \epsilon ^{1-\alpha p}),
\end{align*}
hold  for all $\epsilon >0$ small enough. 
Hence, for any small enough $\eps>0$,
\begin{align*}
\mathbf{P}\left\{\beta _{\epsilon } <\rho_\eps\right\} &=\mathbf{P}%
\left\{\sup_{t\leq \beta _{\epsilon }\wedge \rho_\eps}|Y_{\epsilon,1
}(t)|e^{-\lambda _{+}t}\geq (1+\gamma )\epsilon ^{\alpha p}\right\} \\
&\leq \mathbf{P}\left\{K\rho \epsilon ^{\alpha (1-p)\lambda _{-}/\lambda _{+}}|\log
\epsilon |+K\epsilon ^{1-\alpha p}\geq \gamma \right\} + \delta,
\end{align*}%
which implies the result.
\end{proof}

\medskip

The following is an important consequence of Lemma~\ref{lemma: y_after_global}:
\begin{corollary} With $\tau_\eps$ as in~\eqref{eqn: tau_def} it holds that
\[
\lim_{\eps \to 0} \Pp \{ \tau_\eps^U < \tau_\eps \}=0.
\]
In particular,~\eqref{eq:exit_through_unstabale} holds.
\end{corollary}

From now on, we restrict our analysis to the high probability event $\{ \tau_\eps^U \geq \tau_\eps \}$.

Let $\theta _{\epsilon }^{+}=\epsilon ^{-\alpha p}e^{-\lambda _{+} \tau _\eps}Y _{\epsilon,1 }(\tau _{\epsilon })$.
Then, \eqref{eqn: tau_def} implies  
\begin{equation}
\tau _{\epsilon } =-\frac{\alpha p}{\lambda _{+}}\log \epsilon +\frac{1}{%
\lambda _{+}}\log \frac{\delta }{|\theta _{\epsilon }^{+}|},
\label{eqn: sigma_equal} 
\end{equation}
and
\[
Y _{\epsilon,1 }(\tau_{\epsilon }) =\delta \sgn\theta _{\epsilon }^{+}.
\]
Our analysis of these expressions will be based on the next formula which directly follows 
from Duhamel's principle: 
\begin{equation}
\theta _{\epsilon }^{+}=\sgn\eta_{\epsilon }^{+}+\epsilon ^{-\alpha
p}\int_{0}^{\tau _{\epsilon }}e^{-\lambda _{+}s} H_{1}(Y_\eps(s),\eps)ds+\epsilon ^{1-\alpha p}N_{\epsilon
}^{+}(\tau _{\epsilon }).  \label{eqn: theta+}
\end{equation}%

The main term in the r.h.s.\ of~\eqref{eqn: theta+} is $\sgn\eta_{\epsilon }^{+}$. We need to estimate the other two terms. Lemma~\ref{lemma: stoch_est} implies that $\epsilon ^{1-\alpha p} N_{\epsilon}^{+}(\tau _{\epsilon })$ converges to $0$ in probability as $\eps\to 0$. 
Let us now estimate the integral term. 
Relations~\eqref{eqn: sigma_equal} and~\eqref{eqn: theta+} imply that $(\tau_\eps)_{\eps>0}$ is slowly growing, and we can use 
Lemma~\ref{lemma: y_after_global} to derive
\begin{equation}
\sup_{t\leq \tau_{\epsilon }}|Y _{\epsilon,2 }(t)|=O_{\mathbf{P}}(\epsilon
^{\alpha (1-p)\lambda _{-}/\lambda _{+}}). 
\label{eqn:bound_on_nu_in_probability}
\end{equation}
We can now use Theorem~\ref{Lemma: Def_H1&H2} to conclude that
\begin{equation*}
\epsilon^{-\alpha p} \sup_{t \leq \tau_\epsilon } | H_1 (Y_\epsilon(t),\eps)|=O_{\mathbf{P}} ( \epsilon^{ \alpha(1-p)\lambda _{-}/\lambda _{+} -\alpha p}+\eps^{2-\alpha p}),
\end{equation*}
and (\ref{eqn: p_prop}) implies that the r.h.s.\ converges to $0$. Therefore,

\[
\epsilon^{-\alpha p} \int_{0}^{\tau _{\epsilon }}e^{-\lambda _{+}s} H_{1}(Y_{\epsilon}(s),\eps)ds
\stackrel{\Pp}{\longrightarrow}0.
\]
The above analysis of equation~\eqref{eqn: theta+} implies that if we define $\theta _{0}^{+}=\sgn\eta _{0}^{+}$, then
\begin{align}
 \theta_\eps^+&\stackrel{\mathop{Law}}{\longrightarrow} \theta_0^+,\label{eqn:convergence_of_theta}
\end{align}
which implies~\eqref{eq:tau_eps_convergence} due to~\eqref{eqn: sigma_equal}.
It remains to prove~\eqref{eq:asymptotics_for_Y_2_tau}.

Duhamel's principle along with (\ref{eqn: sigma_equal}) yields
\begin{equation}
Y_{\epsilon,2 }(\tau _{\epsilon })=\left( \frac{|\theta _{\epsilon }^{+}|}{
\delta }\right) ^{\lambda _{-}/\lambda _{+}}\epsilon ^{\alpha \lambda
_{-}/\lambda _{+}}
\eta _{\epsilon }^{-}+\int_{0}^{\tau _{\epsilon
}}e^{-\lambda _{-}(\tau_\eps-s)} H_{2}(Y _{\epsilon}(s),\eps)ds+\epsilon 
N_{\epsilon }^{-}(\tau _{\epsilon }).
\label{eqn:duhamel-at-tau-eps}
\end{equation}
In order to study the convergence of $N_\eps^-(\tau_\eps)$ we first give a preliminary result.
\begin{lemma}
\label{lemma: x_approx}
\begin{equation*}
\sup_{t\leq \tau _{\epsilon }}|Y_{\epsilon,1 }(t)-\epsilon ^{\alpha p}e^{\lambda _{+} t} \sgn\eta _{\epsilon }^{+}| \stackrel{\Pp}{\longrightarrow} 0, \quad \eps \to 0.
\end{equation*}%

\end{lemma}

\begin{proof}
The lemma follows from Duhamel's principle and Lemma~\ref{lemma: x_after_global}.
\end{proof}

The following result is essentially Lemma 8.9 from~\cite{nhn}. It holds true in our
setting since its proof is based only on the conclusion of Lemma~\ref{lemma: x_approx}. 

\begin{lemma}
\label{lemma: stoch_convergence}As $\epsilon \rightarrow 0$, 
\begin{equation*}
N_{\epsilon }^{-}(\tau_{\epsilon })\overset{Law}{%
\longrightarrow }N,
\end{equation*}
where $N$ is the Gaussian random variable in~\eqref{eq:theta}.
\end{lemma}

We finish the proof of Lemma~\ref{Thm:after_tauh}. Recall that the process $M_\eps$ was defined in~\eqref{eqn: M_def} and introduce the stochastic processes
\begin{equation}
R_{\epsilon }(t) = \int_{0}^{t}e^{-\lambda _{-}(t-s)} \hat H_2 ( Y_\eps (s) )ds.  \label{eqn: R_def}
\end{equation}%
Note that~\eqref{eqn:duhamel-at-tau-eps} and \eqref{eqn: sigma_equal}  imply
\begin{align}
 Y_{\eps,2}(\tau_\eps)&=e^{-\lambda_-\tau_\eps}Y_{\eps,2}(0)+\int_0^{\tau_\eps} e^{-\lambda_-(\tau_\eps-s)} H_2(Y_{\epsilon}(s),\eps)ds+\eps N^-_\eps(\tau_\eps)
\notag\\
&=e^{-\lambda _{-}\tau _{\epsilon
}}\epsilon ^{\alpha (1-p)\lambda
_{-}/\lambda _{+}}\eta _{\epsilon }^{-}+\epsilon M_{\epsilon }(\tau_{\epsilon })+R_{\epsilon }(\tau_{\epsilon })  \notag \\
&=\eta _{\epsilon }^{-}\left( \frac{|\theta _{\epsilon }^{+}|}{\delta }%
\right) ^{\lambda _{-}/\lambda _{+}}\epsilon ^{\alpha \lambda _{-}/\lambda
_{+}}+\epsilon M_{\epsilon }(\tau_{\epsilon })+R_{\epsilon }(\tau_{\epsilon }).  \label{eqn: y_order}
\end{align}%
Relations \eqref{eq:condition_for_theorem_along_unstable} and~\eqref{eqn:convergence_of_theta} imply
\begin{equation}
\eta _{\epsilon }^{-}\left( \frac{|\theta _{\epsilon }^{+}|}{\delta }%
\right) ^{\lambda _{-}/\lambda _{+}}\overset{Law}{\longrightarrow }%
\left( \frac{|\eta _{0}^{+}|}{\delta }\right) ^{\lambda _{-}/\lambda
_{+}}y_{2}.
\label{eqn:term_from_initial_cond}
\end{equation}
Lemma~\ref{lemma: stoch_convergence} and estimate~\eqref{eqn:bound_on_nu_in_probability} imply
\begin{equation}
M_{\epsilon }(\tau _{\epsilon })\overset{Law}{\longrightarrow }{N},
\quad\epsilon \rightarrow 0.
\label{eqn:gaussian_term}
\end{equation}
Equations~\eqref{eqn:term_from_initial_cond} and~\eqref{eqn:gaussian_term}
describe the behavior of first two terms in (\ref{eqn: y_order})
and the proof of the lemma will be complete as soon as we show that
\begin{equation}
\epsilon ^{-\beta}R_{\epsilon }(\tau _{\epsilon })\overset{%
\mathbf{P}}{\longrightarrow }0,\quad \epsilon \rightarrow 0.
\label{eqn:R_conv_to_0_in_prob}
\end{equation}%

We can write the following rough estimate based on \eqref{eqn:bound_on_nu_in_probability} and Theorem~\ref{Lemma: Def_H1&H2}:
\begin{equation}
\sup_{t\leq \tau_{\epsilon }}|R_{\epsilon }(t)|=O_{\mathbf{P}%
}(\epsilon ^{2\alpha (1-p)\lambda _{-}/\lambda _{+}}).  \label{eqn: R_est_1}
\end{equation}
This is not sufficient for our purposes. We shall need a more detailed analysis instead.
First, note that 
\begin{equation*}
\sup_{t \leq \tau_\eps }|Y_{\epsilon,2 }(t)-\epsilon M_{\epsilon }(t)-R%
_{\epsilon }(t)|e^{\lambda _{-}t}=\epsilon ^{\alpha (1-p)\lambda
_{-}/\lambda _{+}}|\eta _{\epsilon }^{-}|=O_{\mathbf{P}}(\epsilon ^{\alpha
(1-p)\lambda _{-}/\lambda _{+}}).
\end{equation*}
Hence, for any $\gamma >0$ there is a $K_{\gamma }>0$ such that the event 
\begin{equation*}
D_{\epsilon }=\left\{\sup_{t\leq \tau_{\epsilon }}|Y _{\epsilon,2 }(t)-\epsilon M%
_{\epsilon }(t)-R_{\epsilon }(t)|e^{\lambda _{-}t}<K_{\gamma
}\epsilon ^{\alpha (1-p)\lambda _{-}/\lambda _{+}}\right\}
\end{equation*}%
has probability $\mathbf{P}(D_{\epsilon })> 1-\gamma $ for $%
\epsilon >0$ small enough. Moreover, using Theorem~\ref{Lemma: Def_H1&H2} we see that for some constant $K_\beta>0$,%
\begin{equation*}
|R_{\epsilon }(t)|\leq K_\beta\int_{0}^{t}e^{-\lambda
_{-}(t-s)}Y _{\epsilon,2 }^{2}(s)ds.
\end{equation*}%
Then, using the inequality $(a-b)^{2}\leq 2a^{2}+2b^{2}$ and defining 
$K_{\beta ,\gamma }=K_{\beta }K_{\gamma },$ we see that on  $D_{\epsilon }$ for each $t\leq \tau_{\epsilon }$,  
\begin{align} \notag
|R_{\epsilon }(t)| &
\le K_{\beta }e^{-\lambda_{-} t}\int_{0}^{t}(e^{\lambda
_{-}s}Y _{\epsilon,2 }(s))^2 e^{-\lambda_- s}ds
\\ \notag
& \leq 2K_{\beta ,\gamma }e^{-\lambda
_{-}t}\int_{0}^{t}e^{-\lambda _{-}s}\epsilon ^{2\alpha (1-p)\lambda
_{-}/\lambda _{+}}ds+2K_{\beta }\int_{0}^{t}e^{-\lambda _{-}(t-s)}|\epsilon 
M_{\epsilon }(s)+R_{\epsilon }(s)|^{2}ds \\
&\leq 2\frac{K_{\beta ,\gamma }}{\lambda _{-}}\epsilon ^{2\alpha
(1-p)\lambda _{-}/\lambda _{+}}e^{-\lambda _{-}t}+4\frac{K_{\beta }}{\lambda
_{-}}\epsilon ^{2} M_{\epsilon ,\infty }^{2}+4K_{\beta }e^{-\lambda
_{-}t}\int_{0}^{t}e^{\lambda _{-}s}R_{\epsilon }(s)^{2}ds, \label{eqn:R_eps}
\end{align}%
where  
$M_{\epsilon ,\infty }=\sup_{t\leq \tau _{\epsilon }}|M_{\epsilon
}(t)|,
$
so that (according to Lemma \ref{lemma: stoch_est}) $ M_{\epsilon ,\infty }$ is slowly growing.
Due to \eqref{eqn: R_est_1}  we can find a constant $K_{\gamma }^{\prime }>0$ (independent of $%
\epsilon >0$ and $t>0$) so that the event 
\begin{equation*}
D_{\epsilon }^{\prime }=D_{\epsilon }\cap \left\{\sup_{t\leq
\tau _{\epsilon }}|R_{\epsilon }(t)|\leq K_{\gamma }^{\prime
}\epsilon ^{\alpha (1-p)\lambda _{-}/\lambda _{+}}\right\}
\end{equation*}%
has probability $\mathbf{P}(D_{\epsilon }^{\prime })>1-\gamma $
for all $\epsilon >0$ small enough. Hence, multiplying both sides of~\eqref{eqn:R_eps} by
$e^{\lambda _{-}t}$, we see that for some constant $C_{\gamma }>0$ and all $t\le \tau_\eps$,
\begin{equation*}
e^{\lambda _{-}t}|R%
_{\epsilon }(t)|\mathbf{1}_{\mathcal{D}_{\epsilon }^{\prime }}\leq \alpha (t)+C_{\gamma }\epsilon ^{\alpha (1-p)\lambda
_{-}/\lambda _{+}}\int_{0}^{t}e^{\lambda _{-}s}|R_{\epsilon }(s)|%
\mathbf{1}_{\mathcal{D}_{\epsilon }^{\prime }}ds,
\end{equation*}%
where
\begin{equation}
\alpha (t)=C_{\gamma }\epsilon ^{2\alpha (1-p)\lambda _{-}/\lambda
_{+}}+C_{\gamma }\epsilon ^{2}M_{\epsilon ,\infty }^{2}e^{\lambda _{-}t}. \label{eqn: alpha_def}
\end{equation}
Using Gronwall's lemma and~\eqref{eqn: alpha_def} we get
\begin{eqnarray*}
\mathbf{1}_{\mathcal{D}_{\epsilon }^{\prime }}e^{\lambda _{-}t}|R%
_{\epsilon }(t)| &\leq &\alpha (t)+C_{\gamma }\epsilon ^{\alpha (1-p)\lambda
_{-}/\lambda _{+}}\int_{0}^{t}\alpha (s)e^{C_{\gamma }\epsilon ^{\alpha
(1-p)\lambda _{-}/\lambda _{+}}(t-s)}ds \\
&\leq &\alpha (t)+C_{\gamma }^{2}\epsilon ^{3\alpha (1-p)\lambda
_{-}/\lambda _{+}}t e^{C_{\gamma }\epsilon ^{\alpha (1-p)\lambda _{-}/\lambda
_{+}}t} \\
&&+\frac{C_{\gamma }^{2}}{\lambda _{-}}\epsilon ^{2+\alpha (1-p)\lambda
_{-}/\lambda _{+}}M_{\epsilon ,\infty }^{2} t e^{\lambda _{-}t+C_{\gamma
}\epsilon ^{\alpha (1-p)\lambda _{-}/\lambda _{+}}}.
\end{eqnarray*}%
Hence, 
\begin{eqnarray*}
\mathbf{1}_{\mathcal{D}_{\epsilon }^{\prime }}|R_{\epsilon }(t)|
&\leq &C_{\gamma }\epsilon ^{2\alpha (1-p)\lambda _{-}/\lambda
_{+}}e^{-\lambda _{-}t}(1+C_{\gamma }\epsilon ^{\alpha (1-p)\lambda
_{-}/\lambda _{+}}t e^{C_{\gamma }\epsilon ^{2\alpha (1-p)\lambda _{-}/\lambda
_{+}}t}) \\
&&+C_{\gamma }\epsilon ^{2}M_{\epsilon ,\infty }^{2}(1+\frac{C_{\gamma }}{%
\lambda _{-}}\epsilon ^{\alpha (1-p)\lambda _{-}/\lambda _{+}}t e^{C_{\gamma
}\epsilon ^{\alpha (1-p)\lambda _{-}/\lambda _{+}}}).
\end{eqnarray*}%
Using \eqref{eqn: sigma_equal}, we get that for any $q>0$,
\begin{align*}
\mathbf{1}_{\mathcal{D}_{\epsilon }^{\prime }}|R_{\epsilon
}(\tau _{\epsilon })|
&= O_\Pp \left(\epsilon ^{2\alpha (1-p)\lambda _{-}/\lambda
_{+}}e^{-\lambda _{-}\tau_\eps}+\epsilon ^{2}M_{\epsilon
,\infty }^{2}\right)
\\&
=O_\Pp\left(\epsilon ^{\alpha \lambda _{-}/\lambda _{+}+\alpha (1-p)\lambda
_{-}/\lambda _{+}}+\epsilon ^{2-q}\right),
\end{align*}%
so that~\eqref{eqn:R_conv_to_0_in_prob} follows, and the proof is complete by choosing $q$ small enough.

\chapter{Levinson Case} \label{ch: levinson}
In this chapter we study Levinson case as presented in Section~\ref{sec: unstable} of Chapter~\ref{ch: Intro}. We then apply the results obtained for this case to the $1$-dimensional diffusion conditioned on rare events as explained in Section~\ref{sec: 1d-time-reversal} of Chapter~\ref{ch: Intro}.

The chapter is organized as follows. In Section~\ref{sec:main-results} we state the main theorem for the Levinson case, postponing its proof to Section~\ref{sec:proof-levinson}. A approximation to the diffusion by the deterministic flow in finite time is presented in Section~\ref{sec:finite-levinson}. This  approximation is a key ingredient in all the arguments of Section ~\ref{sec:proof-levinson}.
In Section~\ref{sec:1-d-diffusion} we state the result on the diffusion conditioned on a rare event and derive it
from the main theorem and some auxiliary statements proven in Section~\ref{sec:aux-1-d}.

\section{Introduction}
\label{Sec: Notation}
In this section we consider the dynamics in $d$ dimensions. That is, we consider a $C^2$-smooth bounded vector field~$b$ in $\R^d$. The unperturbed dynamics is given by
 the deterministic flow $S=(S^t)_{t\in\R}$ generated by $b$.

The model has slight modifications from the classical exit problem. For this chapter, we introduce three components of perturbations of this deterministic flow. They all depend on
a small parameter $\eps>0$. 

The first component is white noise perturbation generated by the matrix $\eps \sigma$, where
$\sigma:\R^d \to \R^{d \times d}$ is a $C^2$-smooth bounded matrix valued function.

The second one is $\eps^{\alpha_1}\Psi_\eps$, where $\Psi_\eps$ is
a deterministic Lipschitz vector field on $\R^d$ for each $\eps$,  converging uniformly to 
a limiting Lipschitz vector field $\Psi_0$, and $\alpha_1$ is a positive scaling exponent.
These conditions ensure that the stochastic It\^o
equation  
\begin{equation}
dX_\eps(t) = \left( b( X_\eps(t) )+\eps^{\alpha_1} \Psi_\eps(X_\eps(t)) \right)dt + \eps \sigma(X_\eps(t))dW \label{eq: lev-PrincipalEquation}
\end{equation}
w.r.t.\ a standard $d$-dimensional Wiener process $W$ %
has a unique strong solution for any $\eps>0$ and all initial conditions.  

The last component of the perturbation is the initial condition satisfying
\begin{equation}
X_{\epsilon }(0)=x_{0}+\epsilon ^{\alpha_2 }\xi _{\epsilon },\quad\eps>0.
\label{eqn: initial_condition}
\end{equation}%
Here  $\alpha_2>0$, and $(\xi _{\epsilon })_{\epsilon >0}$ is
a family of random variables independent of $W$, such that for some random
variable $\xi _{0}$, $\xi _{\epsilon }\rightarrow \xi _{0}$ as $\epsilon
\rightarrow 0$ in distribution.

Let $M$ be a $C^2$-smooth hypersurface in $\R^d$.  If 
\begin{equation*}
\tau _{\epsilon }=\inf \left \{t\ge 0:X_{\epsilon
}(t)\in M \right \},
\end{equation*}
then on $\{\tau_\eps<\infty\}$ we have $X_\eps(\tau_\eps)\in M$. We are going to study
the exit problem from $M$ under the assumptions above. We use $M$ instead of $D$, since $M$ is assumed to be an hypersurface and we want to stick to the standard notation.
In this setting, we state the main theorem in the next section.

\section{Main result}\label{sec:main-results}
In this section we state the main theorem and its hypothesis. Let us start with the assumptions on the joint geometry of the vector field $b$ and the surface $M$. 
First we define
\[
T=\inf \left \{t>0: S^tx_0 \in M \right \},
\]
and assume that $0<T<\infty$. Secondly, we denote $z=S^Tx_0\in M$ and assume that 
$b(z)$ does not belong to the tangent
hyperplane $T_zM$. In other words, we assume that the positive orbit of $x_0$ intersects $M$ and
the crossing is transversal. The reader can check that this is equivalent to Condition~\ref{cond: levinson} in Section~\ref{sec: unstable} of Chapter~\ref{ch: Intro}.

In the case of $\xi_\eps\equiv 0$ and $\Psi \equiv 0$, Levinson's theorem states (see \cite{Levinson:MR0033433}, 
\cite[Chapter 2]{Freidlin--Wentzell-book}, and \cite[Chapter 2]{Freidlin-lectures:MR1399081}) that 
$X_\eps(\tau_\eps)\to z$ in probability as $\eps \to 0$. Levinson worked in the PDE context and 
showed how to obtain an expansion for the solution of the corresponding elliptic PDE depending 
on the small parameter~$\eps$.  
The main result of this note describes the limiting behavior of the correction $(\tau_\eps-T,X_\eps(\tau_\eps)-z)$ and extends  \cite[Theorem 2.3]{Freidlin--Wentzell-book} to the situation with generic
perturbation parameters $\xi_0,\Psi,\alpha_1$, and $\alpha_2$. This extension is essential since, as the analysis in~\cite{nhn}
shows, in the sequential study of entrance-exit distributions for multiple domains one has to consider nontrivial scaling laws for the initial conditions; also, considering nontrivial deterministic perturbations will allow us to study rare events, see 
Section~\ref{sec:1-d-diffusion}.

We need more notation. Due to the smoothness of $b$,
\begin{equation}
b(x)=b(y)+Db(y)(x-y)+Q_1(y,x-y),\quad x,y\in\R^d, \label{eqn: b_Taylor}
\end{equation} 
where 
\begin{equation}
|Q_1(u,v)|\leq K  |v|^2 \label{eqn: bound_Q1},
\end{equation}
 for some constant $K>0$ and any $u,v \in \R^d$.
We denote by $\Phi_x(t)$ the linearization of $S$ along the
orbit of $x$:
\begin{equation}
\frac{d}{dt}\Phi_x(t)=A(t)\Phi_x(t)\text{, \ }\Phi_x(0)=I, \label{eqn: Phi_def}
\end{equation}
where $A(t)=Db(S^tx)$ and $I$ is the identity matrix. 
 
Finally, for any vector $v\in \R^d$, we define $\pi_b v\in \R$ and $\pi_M v\in T_zM$ by
\[
 v=\pi_b v \cdot b(z)+ \pi_M v,
\]
i.e., $\pi_b$ is the (algebraic) projection onto $\Span(b(z))$ along $T_zM$ and 
$\pi_M$ is the (geometric) projection onto $T_zM$ along $\Span(b(z))$.

 \begin{theorem}
 \label{thm: Main-Levinson}  Let $\alpha=\alpha_1 \wedge \alpha_2\wedge 1$, and
\begin{align} \notag
\phi_0 (t) &= \mathbf{1}_{\{\alpha_2= \alpha\}}  \Phi_{x_0}(t) \xi_0 + \mathbf{1}_{\{\alpha_1= \alpha\}}\Phi_{x_0}(t) \int_0^t \Phi_{x_0} (s)^{-1} \Psi_0( S^s x) ds \\
& \quad +\mathbf{1}_{\{1=\alpha\}} \Phi_{x_0}(t)\int_{0}^{t}\Phi_{x_0}^{-1}(s)\sigma(S^s x_0)dW(s), \quad t>0.
\label{eqn: phi_0_def}
\end{align}
Then, in the setting introduced above, 
\begin{equation}
 \eps^{-\alpha}(\tau_\eps-T, X_\eps(\tau_\eps)-z){\to} (-\pi_b \phi_0(T), \pi_M \phi_0(T)). \label{eq:main_convergence_statement}
\end{equation}
in distribution.
If additionally we require that $\xi_\eps\to\xi_0$ in probability or that $\alpha_2>\alpha$, then the convergence in \eqref{eq:main_convergence_statement}
is also in probability.
\end{theorem}

\begin{remark}\label{rem:weakening-conditions-by-localization}
\rm The conditions of Theorem~\ref{thm: Main-Levinson} can be relaxed using the standard localization procedure. In fact, one needs to
require uniform convergence of $\Psi_\eps\to\Psi_0$ and regularity properties of $b$ and $\sigma$
only in some neighborhood of the set $\{S^tx_0:\ 0\le t\le T(x_0)\}$.
\end{remark}

\begin{remark}\rm In applications (see~\cite{nhn},\cite{nhn-ds}), the parameters $\alpha_1$ and $\alpha_2$ can be chosen so that
the r.h.s.\ of~\eqref{eq:main_convergence_statement} is nondegenerate. 
\end{remark}

\begin{remark}\rm
In the case where $d=1$, the hypersurface $M$ is just a point. Therefore, $\pi_M$ is identical zero and the only 
contentful information Theorem~\ref{thm: Main-Levinson} provides is the asymptotics of the exit time. 
\end{remark}

\section{A finite time approximation result}\label{sec:finite-levinson}

With high probability, at time $T$ the process $X_\eps$ is close to $z$ and the hitting
time~$\tau_\eps$ is close to $T$. The idea behind the proof of Theorem~\ref{thm: Main-Levinson} is that while the diffusion is close to~$z$, the process
may be approximated very well by motion with constant velocity~$b(z)$. 

In this section we prove the main ingredient to ensure this approximation. 
\begin{lemma} 
\label{lemma: Bakhtin-Levinson}
Let $X_\eps$ be the solution of the SDE~\eqref {eq: lev-PrincipalEquation} with initial condition~\eqref{eqn: initial_condition}. 
Let
\begin{align}
\Theta_\eps (t) &=\eps^{\alpha_2-\alpha}\Phi_{x_0}(t) \xi_\eps+  \eps^{\alpha_1-\alpha}\Phi_{x_0}(t) \int_0^t \Phi_{x_0} (s)^{-1} \Psi_0( S^s x_0) ds \notag \\
&\quad + \eps^{1-\alpha} \Phi_{x_0}(t) \int_0^t \Phi_{x_0} (s)^{-1}\sigma( S^s x_0) dW(s).  \label{eqn: SDE_Theta} 
\end{align} Then, 
\[
X_\eps(t)=S^tx_0 + \eps^\alpha \phi_\eps (t)
\]
holds almost surely for every $t>0$,
where $\phi_\eps(t)=\Theta_\eps(t)+ r_\eps(t)$, and $r_\eps$ converges to 0 uniformly over compact time intervals in probability.

If $\xi_\eps\to\xi_0$ in distribution, then for any $T>0$, 
$\phi_\eps \to \phi_0$  in distribution in $C[0,T]$ equipped with uniform norm,
 where $\phi_0$ is the stochastic process defined in~\eqref{eqn: phi_0_def}.

If $\xi_\eps\to\xi_0$ in probability or $\alpha_2>\alpha$, then  the uniform convergence for $\phi_\eps$
also holds in probability.
\end{lemma}

\begin{remark}\rm 
This lemma gives the first-order approximation for $X_\eps(t)$. Higher-order approximations in the spirit of~\cite{Blagoveschenskii:MR0139204}
are also possible. They can be used to refine Theorem~\ref{thm: Main-Levinson}.
\end{remark}

\begin{proof}
Let $\Delta_\eps^t= X_\eps (t) -S^t x_0$ and note that it satisfies the equation
\[
d\Delta_\eps^t= \left(  \left( b(X_\eps (t))-b(S^t x_0) \right) + \eps^{\alpha_1} \Psi_\eps(X_\eps (t)) \right)dt + \eps \sigma (X_\eps (t) ) dW(t),
\]
with initial condition $\Delta_\eps^0 = \eps^{\alpha_2} \xi_\eps$. We want to study the properties of this equation. We start with the difference in $b$.
Since $b$ is a $C^2$ vector field, we may write 
\begin{align} 
b(X_\eps (t))-b(S^t x_0) &=Db(S^t x_0) \Delta_\eps^t + Q_1(S^t x_0, \Delta_\eps^t).
 \label{eqn: expansionb}
\end{align}
Also, we can write 
\begin{equation} \label{eqn: expansion_Psi}
\Psi_\eps ( X_\eps (t) )=\Psi_0 ( S^t x_0)+Q_2 (S^t x_0,\Delta_\eps^t)+R_\eps(S^t x_0),
\end{equation} and
\begin{equation} \label{eqn: expansionSigma}
\sigma ( X_\eps (t) )=\sigma ( S^t x_0)+Q_3(S^t x_0,\Delta_\eps^t),
\end{equation}
where 
\[
 R_\eps(x)=\Psi_\eps(x)-\Psi_0(x)=o(1),\quad \eps\to 0,
\]
uniformly in $x$;
$Q_i:\R^d \times \R^d \to \R^d$, $i=2,3$ satisfies
\begin{equation}
|Q_i (u,v)|\leq K   |v|,\quad u,v\in\R^d  \label{eqn: bound_Q2}.
\end{equation}
We can assume that the constant $K>0$ in~\eqref{eqn: bound_Q1} and~\eqref{eqn: bound_Q2} 
is the same for simplicity of notation. 

Let $Q=Q_1+ \eps^{\alpha_1} Q_2+\eps^{\alpha_1}R_\eps$. Combine~\eqref{eqn: expansionb}, \eqref{eqn: expansion_Psi}, and \eqref{eqn: expansionSigma} to get 
\begin{align} \notag
d\Delta_\eps^t = &\left(  A(t)\Delta_\eps^t + \eps^{\alpha_1} \Psi_0 ( S^t x_0) + Q(S^t x_0,  \Delta_\eps^t) \right)dt \\
&\quad + \eps \left(  \sigma ( S^t x_0 ) +Q_3(S^t x_0,\Delta_\eps^t) \right) dW(t), \label{eqn: SDE_Delta} \\
\Delta_\eps^0 = &\eps^{\alpha_2} \xi_\eps.
\end{align}
Hence, applying Duhamel's principle to ~\eqref{eqn: SDE_Delta} and using~\eqref{eqn: SDE_Theta},  we get
\begin{align} \notag
\Delta_\eps^t & = \eps^\alpha \Theta_\eps (t)+ \Phi_{x_0} (t) \int_0^t \Phi_{x_0} (s)^{-1}Q(S^s x_0,  \Delta_\eps^s )ds\\ \notag
&\quad + \eps \Phi_{x_0} (t) \int_0^t \Phi_{x_0} (s)^{-1}Q_3(S^s x_0, \Delta_\eps^s )dW(s) \\ \label{eqn: expansion_X}
&= \eps^\alpha \Theta_\eps (t) + \Theta'_\eps (t),
\end{align} where $\Theta'_\eps$ is defined by~\eqref{eqn: expansion_X}.
A simple inspection of~\eqref{eqn: SDE_Theta} shows that $\left( \Theta_\eps \right)_{\eps>0}$ 
 converges in  distribution in $C(0,T)$ to the process $\phi_0 (t)$. This convergence is in probability if $\alpha_2 > {\alpha}$
or $\xi_\eps\to\xi_0$ in probability.
 Therefore, the lemma will follow with $\phi_\eps = \Theta_\eps + \eps^{-\alpha} \Theta'_\eps$ if we show that 
\begin{equation} \label{eqn: estimate_Theta'}
\eps^{-\alpha} \sup_{t \leq T} | \Theta'_\eps (t) | \overset{\Pp}{\longrightarrow} 0, \quad \eps \to 0.
\end{equation}
 For any $\delta \in (1/2,1)$, we introduce the stopping time
\[
l_\eps ( \delta )=\inf \left\{ t>0: |\Delta_\eps^t| \geq \eps^{\alpha \delta } \right \}.
\]
Now, $\Theta'_\eps=\Theta'_{\eps,1}+\eps \Theta'_{\eps,2}$, where
\[
\Theta'_{\eps,1} (t) = \Phi_{x_0} (t) \int_0^t \Phi_{x_0} (s)^{-1}Q(S^s x_0, \Delta_\eps^s)ds,
\] and
\[
\Theta'_{\eps,2} (t)=\eps \Phi_{x_0} (t) \int_0^t \Phi_{x_0} (s)^{-1}Q_3(S^s x_0,  \Delta_\eps^s )dW(s).
\]
Bounds~\eqref{eqn: bound_Q1}, and \eqref{eqn: bound_Q2} imply
\begin{equation}
\sup_{t\leq T\wedge l_\eps(\delta) }|\Theta'_{\eps,1} (t) | = O( \eps^{2 \alpha \delta}  + \eps^ {{\alpha_1}+\alpha \delta})+o(\eps^{\alpha_1})=o(\eps^{\alpha}). \label{eqn: Theta_prime_1}
\end{equation}
Likewise,~\eqref{eqn: bound_Q2} for $Q_3$ and BDG inequality imply that for any $\kappa>0$ there is a constant $K_\kappa$ such that
\begin{equation}
\Pp \left \{ \sup_{t\leq T\wedge l_\eps(\delta) }|\Theta'_{\eps,2} (t) | > K_\kappa \eps^ {1+\alpha\delta} \right \} < \kappa \label{eqn: Theta_prime_2}
\end{equation} for all $\eps>0$ small enough. Then, this together with~\eqref{eqn: Theta_prime_1} imply that 
\begin{equation}
\eps^{-\alpha\delta} \sup_{t \leq T \wedge l_\eps (\delta)} | \Theta'_\eps (t) | \overset{\Pp}{\longrightarrow} 0, \quad \eps \to 0.
\label{eq:theta-prime-small}
\end{equation}
Then, if $l_\eps (\delta) < T$ we use~\eqref{eqn: expansion_X} to get
\begin{align*}
1 &=\eps^{-\alpha\delta} \sup_{t\leq T \wedge l_\eps (\delta)} | \Delta_\eps^t | \\
& \leq \eps^{\alpha(1-\delta)} \sup_{t\leq  T \wedge l_\eps (\delta)} | \Theta_\eps (t) | +\eps^{-\alpha\delta} 
\sup_{t\leq  T \wedge l_\eps (\delta)}  | \Theta'_\eps (t) |.
\end{align*} The r.h.s. converges to 0 in probability due to~\eqref{eq:theta-prime-small} and the tightness of distributions of $\Theta_\eps$. Hence, $\Pp \{ l_\eps (\delta) < T \} \to~0$ as $\eps~\to~0$. Using $T$ instead of $T \wedge l_\eps ( \delta) $ in~\eqref{eqn: Theta_prime_1} and~\eqref{eqn: Theta_prime_2}, we see that with the choice of $\delta > 1/2$, \eqref{eqn: estimate_Theta'} follows and the proof is finished.
\end{proof}

\section{Proof of Theorem~\ref{thm: Main-Levinson}} \label{sec:proof-levinson}

As we said before, the core idea of the proof is to approximate the behavior of the process $\x$ with that of the deterministic flow in a small neighborhood of $z$. Let us start analyzing the process $X_\eps(t)-z$ for $t$ close to $T$. Let us first estimate the deviation of the flow $S$
from the motion with costant velocity~$b(z)$. Let 
\begin{equation}
r_\pm (t,x)= S^{\pm t}x-\left( x\pm tb(z) \right),\quad t>0,\ x\in\R^d. \label{eqn: def_r_pm}
\end{equation}
\begin{lemma}
\label{lemma: rectification}
There are constants $C_1$ and $C_2$ so that for any $t>0$ and $x\in\R^d$
\[
\sup_{s\leq t}\left|  r_\pm(s,x) \right| \leq C_1 e^{C_2 t }( t |x-z | + t^2). 
\]
\end{lemma}

\begin{proof}
We prove the result for $r_+$. The analysis of $r_-$ is similar since $S^{-t}x$ is the solution to the ODE
\[
\frac{d}{dt} S^{-t}x = -b (S^{-t}x ).
\]
Let $L>0$ be the Lipschitz constant of $b$. The proof follows from the inequalities:
\begin{align*}
\left| r_+(t,x) \right| &\leq \int_0^t \left| b(S^sx)-b(z)\right|ds \\
& \leq L \int_0^t \left| S^sx-z\right|ds \\
& \leq L\int_0^t \left|r_+ (s,x) \right|ds + L\int_0^t \left | x + sb(z)-z\right| ds \\
& \leq L\int_0^t \left|r_+ (s,x) \right|ds + L\int_0^t |x-z|ds + L\int_0^t s|b(z)| ds \\
&\leq  L\int_0^t \left| r_+(s,x) \right|ds + L t |x-z| + t^2 L|b(z)|/2.
\end{align*}
The result follows as an application of Gronwall's lemma.
\end{proof}

\begin{lemma} \label{lemma: X_before_after}
Let $\gamma \in (\alpha/2, \alpha)$. Then, there are two a.s.-continuous stochastic processes $\Gamma_{\eps,\pm}$ such that
\[
\sup_{t\in[0,\eps^\gamma]} |\Gamma_{\eps,\pm} (t)| \stackrel{\Pp}{\longrightarrow} 0, \quad \eps \to 0,
\]
and almost surely for any $t\in[0,\eps^\gamma]$
\begin{equation}
X_\eps (T-t)=z-tb(z)+\eps^\alpha \left( \phi_\eps (T-t) +\Gamma_{\eps,-}(t) \right) \label{eqn: X_before}
\end{equation} and 
\begin{equation}
X_\eps (T+t)=z+tb(z)+\eps^\alpha \left( \Phi_z (t) \phi_\eps (T) +\Gamma_{\eps,+}(t) \right). \label{eqn: X_after}
\end{equation}
\end{lemma}
\begin{proof}
Due to Lemma~\ref{lemma: Bakhtin-Levinson},  the flow property, and~\eqref{eqn: def_r_pm} we have 
\begin{align*}
X_\eps (T-t) &= S^{T-t}x_0 + \eps^\alpha \phi_\eps (T-t) \\
&= S^{-t}z + \eps^\alpha \phi_\eps (T-t) \\
&=z-tb(z) + r_-(t,z)  + \eps^\alpha \phi_\eps (T-t).
\end{align*}
The first estimate with $\Gamma_{\eps,-} (t) = \eps^{-\alpha} r_-(t,z)$ follows from Lemma~\ref{lemma: rectification} for $x=z$.

Due to Strong Markov property and Lemma~\ref{thm: Main-Levinson} the process $\tilde X_\eps (t) = X_\eps (t + T)$ is a solution of the initial value problem
\begin{align*}
d\tilde X_\eps(t) &= (b( \tilde X_\eps(t))+\eps^{\alpha_1}\Psi_\eps(\tilde X_\eps(t)))dt + \eps \sigma (\tilde X_\eps(t)) d\tilde W, \\
\tilde X_\eps (0) &= X_\eps (T)=z+\eps^\alpha \phi_\eps (T),
\end{align*}
with respect to the Brownian Motion $\tilde W(t)=W(t+T)-W(T)$. So, again, applying Lemma~\ref{thm: Main-Levinson} to this shifted equation, we obtain $\tilde X_\eps (t)=S^tz+\eps^\alpha \hat \phi_\eps (t)$, where, for $t>0$
\begin{align*}
\hat \phi_\eps (t) = \Phi_z (t) \phi_\eps (T) + \theta_\eps (t),
\end{align*}
and 
\[
\theta_\eps (t)= \eps^{1-\alpha}   \Phi_z (t) \int_0^t \Phi_z (s)^{-1}\sigma (S^s z)d\tilde W(s)  
+ \eps^{\alpha_1-\alpha}\Phi_z(t) \int_0^t \Phi_z (s)^{-1} \Psi_0( S^s z) ds + \tilde r_\eps (t),
\] where $\tilde r_\eps$ converges to 0 uniformly over compact time intervals in probability.
Then due to~\eqref{eqn: def_r_pm},
\begin{align*}
\tilde X_\eps (t) &=S^tz+\eps^\alpha (\Phi_z (t) \phi_\eps (t) + \theta_\eps (t) ) \\
&= z+tb(z) + r_+(t,z)+\eps^\alpha (\Phi_z (t) \phi_\eps (t) + \theta_\eps (t) ).
\end{align*}
Hence, with $\Gamma_{\eps,+} (t)=\theta_\eps (t) + \eps^{-\alpha} r_+(t,z)$ the result is a consequence of Lemma~\ref{lemma: rectification}.
\end{proof}

\bigskip

Let us now parametrize, locally around $z$, the hypersurface $M$ as a graph of a $C^2$-function $F$ over $T_zM$, i.e.,
$y\mapsto z+y+F(y)\cdot b(z)$ gives a $C^2$-parametrization of a neighborhood of $z$ in $M$ by a neighborhood of $0$ in $T_zM$. Moreover, $DF(0)=0$ so that
$|F(y)|=O(|y|^2)$, $y\to 0$.
With this definition, it is clear that, for $w\in \R^d$ with $w-z$ small enough, $w\in M$ if and only if $\pi_b(w-z)=F(\pi_M(w-z))$.

Let us define
\begin{align*} 
\Omega_{1,\eps}&=\bigl\{\tau_\eps = \inf \{  t\geq 0: \pi_b \left(X_\eps(t)-z \right) = F\left( \pi_M \left(X_\eps(t)-z \right)  \right)\}\bigr\},\\
\Omega_{2,\eps}&= \left \{  |\tau_\eps - T| \leq \eps^\gamma \right \},\\
\Omega_\eps&=\Omega_{1,\eps}\cap \Omega_{2,\eps}.
\end{align*} 

\begin{lemma}\label{lm:omega_eps_to_1} $\Pp(\Omega_{\eps})\to 1$ as $\eps\to 0$.
\end{lemma}

\begin{proof}
The definition of $F$ and Lemma~\ref{thm: Main-Levinson} imply that  as $\eps\to 0$, $\Pp(\Omega_{1,\eps})\to 1$.

We use~\eqref{eqn: X_after} to conclude that
\[
\pi_b \left( X_\eps (T+\eps^\gamma) -z \right)=\eps^\gamma  \left( 1+ \eps^{\alpha-\gamma} \pi_b \left( \Phi_z(\eps^\gamma) \phi_\eps(T) + \Gamma_{\eps,+} (\eps^\gamma) \right) \right), 
\]
and 
\[
F \left( \pi_M \left( X_\eps (T+\eps^\gamma) -z \right) \right)=F \left( \eps^\alpha \pi_M \left(  \Phi_z(\eps^\gamma) \phi_\eps(T) + \Gamma_{\eps,+} (\eps^\gamma) \right) \right).
\]
Since  $|F(x)|= O(|x|^2)$, these estimates imply that
\begin{multline*}
\limsup_{\eps \to 0} \Pp\left( \left\{  \tau_\eps > T+\eps^\gamma \right \}\cap \Omega_{1,\eps} \right)\\ \le \limsup_{\eps \to 0} \Pp \left\{  \pi_b \left( X_\eps (T+\eps^\gamma) -z \right) 
\le F \left( \pi_M \left( X_\eps (T+\eps^\gamma) -z \right) \right) \right \} =0.
\end{multline*}
It remains to prove
\begin{equation}
\lim_{\eps \to 0} \Pp \left\{  \tau_\eps <T-\eps^\gamma \right \}=0.
\label{eq:rough_lower_estimate_exit_time}
\end{equation}
Let us denote the Hausdorff distance between sets by $d(\cdot,\cdot)$. Then
an obvious estimate
\[
d(\{S^tx_0: 0\le t\le T-\delta\},M)\ge c\delta 
\]
 holds true for some $c>0$ and all sufficiently small $\delta>0$. Now~\eqref{eq:rough_lower_estimate_exit_time} follows from Lemma~\ref{thm: Main-Levinson},
and the proof is complete
\end{proof}

\begin{lemma} \label{prop: tau_conv}
Define $\tau_\eps^\prime = \tau_\eps - T $. Then,
\[
\eps^{-\alpha} \tau_\eps^\prime +  \pi_b \phi_\eps (T) \stackrel{\Pp}{\longrightarrow} 0, \quad \eps \to 0.
\]
\end{lemma}
\begin{proof}
 Let us define $A_\eps = \left \{  0 \leq \tau_\eps^\prime  \leq \eps^\gamma \right \}\cap\Omega_{1,\eps}$ and 
$B_\eps = \left \{  -\eps^\gamma \leq \tau_\eps^\prime  < 0 \right \}\cap\Omega_{1,\eps}$, so that 
$\Omega_\eps = A_\eps \cup B_\eps$.
We can use~\eqref{eqn: X_after} and the definition of $\Omega_{1,\eps}$ to get
\begin{align*}
\one_{A_\eps} \tau_\eps^\prime + \one_{A_\eps} \eps^\alpha  \pi_b \left( \Phi_z (\tau_\eps^\prime) \phi_\eps (T) +\Gamma_{\eps,+}(\tau_\eps^\prime) \right) = 
\one_{A_\eps}F\left( \eps^\alpha \pi_M \left( \Phi_z (\tau_\eps^\prime) \phi_\eps (T) +\Gamma_{\eps,+}(\tau_\eps^\prime) \right) \right).
\end{align*}
This implies
\begin{align} 
\notag
\one_{A_\eps} \eps^{-\alpha} \tau_\eps^\prime &=\eps^{-\alpha} \one_{A_\eps} F\left( \eps^\alpha \pi_M \left( \Phi_z (\tau_\eps^\prime) \phi_\eps (T) +\Gamma_{\eps,+}(\tau_\eps^\prime) \right) \right)
\\ \notag &\quad-\one_{A_\eps} \pi_b \left( \Phi_z (\tau_\eps^\prime) \phi_\eps (T) +\Gamma_{\eps,+}(\tau_\eps^\prime) \right) \\
\notag &= -\one_{A_\eps} \pi_b \left( \Phi_z (\tau_\eps^\prime) \phi_\eps (T)  \right) + r_{\eps,1} \\
&= -\one_{A_\eps} \pi_b  \phi_\eps (T)  +\one_{A_\eps} \pi_b \left( (I-\Phi_z (\tau_\eps^\prime) )\phi_\eps (T)  \right)+r_{\eps,1}, \label{eqn: tau_on_A}
\end{align}
where $r_{\eps,1}$ is a random variable that converges to $0$ in probability as $\eps \to 0$.

Likewise, since $\tau_\eps = T - (-\tau_\eps^\prime)$ and $\one_{B_\eps} \tau_\eps^\prime \le 0$, we can use ~\eqref{eqn: X_before} and the definition of $\Omega_{1,\eps}$ to see that 
\[
\one_{B_\eps}\tau_\eps^\prime + \one_{B_\eps}\eps^\alpha \pi_b \left ( \phi_\eps (T+\tau_\eps^\prime) + \Gamma_{\eps,-} (-\tau_\eps^\prime) \right ) = \one_{B_\eps} F \left (\eps^\alpha\left ( \phi_\eps (T+\tau_\eps^\prime) + \Gamma_{\eps,-} (-\tau_\eps^\prime) \right ) \right ).
\]
Hence, proceeding as before, we see that 
\begin{align*}
\one_{B_\eps} \eps^{-\alpha} \tau_\eps^\prime &= -\one_{B_\eps} \pi_b \phi_\eps (T+\tau_\eps^\prime) + r_{\eps,2} \\
&= -\one_{B_\eps} \pi_b \phi_\eps (T)+\one_{B_\eps} \pi_b \left( \phi_\eps (T)-\phi_\eps(T+\tau_\eps^\prime) \right)+r_{\eps,2}
\end{align*}
for some random variable $r_{\eps,2}$ such that $r_{\eps,2} \to 0$ in probability as $\eps \to 0$. Adding this identity and~\eqref{eqn: tau_on_A}, we see that on $\Omega_{\eps}$
\[
 \eps^{-\alpha} \tau_\eps^\prime=-\pi_b \phi_\eps (T)+\one_{A_\eps} \pi_b \left( ( I-\Phi_z (\tau_\eps^\prime) ) \phi_\eps (T)  \right)+\one_{B_\eps} \pi_b \left( \phi_\eps (T)-\phi_\eps(T+\tau_\eps^\prime) \right)+r_{\eps,1}+r_{\eps,2}.
\]
Due to Lemma~\ref{lm:omega_eps_to_1}, to finish the proof it is sufficient to notice that as $\eps \to 0$
\begin{equation}
\label{eqn: aux_1}
\sup_{0\leq t \leq \eps^\gamma } \left| (I-\Phi_z (t) ) \phi_\eps (T)  \right| \stackrel{\Pp}{\longrightarrow} 0,
\end{equation} and
\begin{equation} \label{eqn: aux_2}
\sup_{0\leq t \leq \eps^\gamma } \left| \phi_\eps (T)-\phi_\eps(T+t) \right|\stackrel{\Pp}{\longrightarrow} 0.
\end{equation}
\end{proof}

Lemma~\ref{prop: tau_conv} takes care of the time component in Theorem~\ref{thm: Main-Levinson}. We shall consider the spatial component now.

 Let $A_\eps$ and $B_\eps$ be as in the proof of Lemma~\ref{prop: tau_conv}. Then, \eqref{eqn: X_after} implies
\begin{multline}
\one_{A_\eps} \left (  X_\eps (\tau_\eps) - z \right)\eps^{-\alpha} = \one_{A_\eps}\eps^{-\alpha} \tau_\eps^\prime b(z) + \one_{A_\eps} \left(  \Phi_z(\tau_\eps^\prime)\phi_\eps(T) + \Gamma_{\eps,+}(\tau_\eps^\prime) \right) \\
= \one_{A_\eps} \left( \eps^{-\alpha}\tau_\eps^\prime b(z) + \phi_\eps (T) \right) + \one_{A_\eps} \left[  (\Phi_z(\tau_\eps^\prime)- I )\phi_\eps(T) + \Gamma_{\eps,+}(\tau_\eps^\prime) \right] \label{eqn: spacial_after}
\end{multline}
Likewise, from~\eqref{eqn: X_before} we get that  
\begin{multline}
\one_{B_\eps} \left (  X_\eps (\tau_\eps) - z \right)\eps^{-\alpha} = \one_{B_\eps}\eps^{-\alpha} \tau_\eps^\prime b(z) + \one_{B_\eps} \left(  \phi_\eps(T+\tau_\eps^\prime) + \Gamma_{\eps,-}(-\tau_\eps^\prime) \right) \\
= \one_{B_\eps} \left(\eps^{-\alpha} \tau_\eps^\prime b(z) + \phi_\eps (T) \right) + \one_{B_\eps} \left[  (\phi_\eps(T+\tau_\eps^\prime)-\phi_\eps(T) ) + \Gamma_{\eps,-}(-\tau_\eps^\prime) \right] \label{eqn: spacial_before}.
\end{multline}
Adding~\eqref{eqn: spacial_after} and~\eqref{eqn: spacial_before} and proceding as in the proof of Lemma~\ref{prop: tau_conv} we see that
\[
\left (  X_\eps (\tau_\eps) - z \right)\eps^{-\alpha} - \pi_M \phi_\eps (T)= \left(\eps^{-\alpha} \tau_\eps^\prime +\pi_b \phi_\eps (T) \right)b(z) + \rho_\eps,
\] where, due to~\eqref{eqn: aux_1},~\eqref{eqn: aux_2} and Lemma~\ref{lemma: X_before_after}, $\rho_\eps \to 0$ in probability as $\eps \to 0$.  From this expression and Lemma~\ref{prop: tau_conv} we get that 
\[
\left (  X_\eps (\tau_\eps) - z \right)\eps^{-\alpha} - \pi_M \phi_\eps (T) \overset{\Pp}{\longrightarrow} 0, \quad \eps \to 0.
\]
Then, using this and the convergence in Lemma~\ref{prop: tau_conv}
\[
\eps^{-\alpha}(\tau_\eps-T, X_\eps(\tau_\eps)-z)=R_\eps + G(\phi_\eps (T) ),
\]
where $R_\eps$ is a random variable such that $R_\eps \to 0$ in probability as $\eps \to 0$. $G$ is the continuous function $x \mapsto (-\pi_b x, \pi_M x )$. Hence, Theorem~\ref{thm: Main-Levinson} follows from the convergence in~Lemma~\ref{thm: Main-Levinson}.

\section{Conditioned diffusions in 1 dimension}\label{sec:1-d-diffusion}

In this section we apply Theorem~\ref{thm: Main-Levinson} to the analysis of the exit time of conditioned diffusions in 1-dimensional situation for the large deviation case.

Suppose, for each $\eps>0$, $X_\eps$ is a weak solution of the following SDE:
\begin{align*}
 dX_\eps(t)&=b(X_\eps(t))dt+\eps\sigma(X_\eps(t)) dW(t),\\
 X_\eps(0)&=x_0,
\end{align*}
where $b$ and $\sigma$ are $C^1$ functions on $\R$, such that $b(x)<0$ and $\sigma(x)\ne 0$ for all $x$ in an interval $[a_1,a_2]$ 
containing $x_0$. We introduce 
\[
 \tau_\eps=\inf\{t\ge0:\ X_\eps(t)=a_1\ \text{\rm or}\ a_2\}
\]
and  $B_\eps=\{X_\eps(\tau_\eps)=a_2\}$. Since $b<0$, $B_\eps$ is a rare event since $\lim_{\eps\to0}\Pp(B_\eps)=0$. More precise estimates
on the asymptotic behavior of $\Pp(B_\eps)$ can be obtained in terms of large deviations. However, here we study the diffusion
$X_\eps$ conditioned on the rare event $B_\eps$.

Let $T(x_0)$ denote the time it takes for the solution of $\dot x=-b(x)$ starting at $x_0$ to reach $a_2$. Given that $b<0$ on the hole interval $[a_1,a_2]$, a simple calculation shows that 
\[
T(x_0)=-\int_{x_0}^{a_2}\frac{1}{b(x)}dx.
\]

\begin{theorem} \label{thm: conditioned} Conditioned on $B_\eps$, the distribution of $\eps^{-1}(\tau_\eps-T(x_0))$ converges weakly to a centered 
Gaussian
distribution with variance 
\[
- \int_{x_0}^{a_2} \frac{ \sigma^2 (y) }{ b^3 (y) } dy.
\]
\end{theorem}

To prove this theorem, we will need two auxiliary statements. Their proofs are given in Section~\ref{sec:aux-1-d}.

\begin{lemma} \label{lemma: conditioned_drift}Conditioned on $B_\eps$, the process $X_\eps$ is a diffusion with the same
diffusion coefficient as the unconditioned process, and with the
drift coefficient given by 
\[
b_\eps(x)= b(x)+\eps^2\sigma^2(x)\frac{h_\eps(x)}{\int_{a_1}^x h_\eps(y)dy},
\]
where
\begin{equation}
 \label{eq:h_eps}
 h_\eps(x)=\exp\left\{-\frac{2}{\eps^2}\int_{a_1}^x \frac{b(y)}{\sigma^2(y)}dy\right\}.
\end{equation}
\end{lemma}

Further analysis requires understanding the limiting behavior of $b_\eps$. This is the purpose of the next lemma:
\begin{lemma} \label{lemma: b_approx} There is $\delta>0$ such that
\[
\limsup_{\eps \to 0}  \eps^{-2}\left(\sup_{x \in [x_0-\delta, a_2+\delta]} | b_\eps(x) + b (x) | \right) < \infty.
\]
\end{lemma}

\begin{remark} \rm Although we need the condition that $b(x)<0$ for all $x \in [a_1,a_2]$ for Theorem~\ref{thm: conditioned}
 to hold, Lemmas~\ref{lemma: conditioned_drift} and~\ref{lemma: b_approx} hold independently of the sign properties of~$b$. 
\end{remark}

\begin{proof}[Proof of Theorem~\ref{thm: conditioned}]
Let us fix $\beta \in (1,2)$. Lemmas~\ref{lemma: conditioned_drift} and~\ref{lemma: b_approx} imply that
$X_\eps$ conditioned on $B_\eps$, up to $\tau_\eps$ satisfies an SDE of the form 
\[
d X_\eps (t)= \left( -b(X_\eps(t)) + \eps^{\beta} \Psi_{\eps,\beta} ( X_\eps (t) ) \right) dt + \eps \sigma ( X_\eps (t) ) d\tilde W(t),
\]
for some Brownian Motion $\tilde W$ and with $\Psi_{\eps, \beta} \to 0$ uniformly as $\eps \to 0$. We can assume that after time $\tau_\eps$,
this process still follows the same equation at least up to the time it hits $x_0 -\delta$ or $a_1+\delta$.

So, having the dynamics from $\dot x=-b(x)$ as the underperturbed dynamics, we can apply Theorem~\ref{thm: Main-Levinson}  (taking into account Remark~\ref{rem:weakening-conditions-by-localization}) to see that 
\begin{equation} \label{eqn: conv_example}
\eps^{-1} ( \tau_\eps - T(x_0) ) \stackrel{\Pp}{\longrightarrow} 
 - \frac{1}{b(a_2)}
\Phi_{x_0}(T( x_0) )\int_{0}^{T (x_0)}\Phi_{x_0}^{-1}(s)\sigma(S^s x_0)d\tilde W(s), \quad \eps\to0,
\end{equation}
where $S^t x_0$ is the flow generated by the vector field $-b$, the time $T(x_0)$ solves $S^{T(x_0)}x_0=a_2$, and
$\Phi_{x_0}$ is the linearization of $S$ near the orbit of $x_0$.
The limit is clearly a centered Gaussian random variable. To compute its variance we must first solve
\[
\frac{d}{dt}\Phi_{x_0} (t) = - b^\prime (S^t x_0 ) \Phi_{x_0} (t), \quad \Phi_{x_0} (0)=1.
\]
 The solution to this linear ODE is 
\[
\Phi_{x_0} (t)= \exp \left \{- \int_0^t b^\prime (S^s x_0 )ds \right \},
\]
so that after the change of variables $u=S^s x_0$ in the integral, we get
\[
\Phi_{x_0} (t)= \frac{b(S^t x_0)}{b(x_0)}.
\]
Using this expression and It\^o isometry for the limiting random variable in~\eqref{eqn: conv_example}, we get that the variance of such random variable is
\[
\int_0^{T(x_0)} \frac{ \sigma^2 ( S^tx_0 ) }{ b^2 ( S^t x_0 ) } dt.
\]
We can now use the change of variable $u=S^s x_0$ to get the expression in Theorem~\ref{thm: conditioned}.
\end{proof}

\subsection{Proof of Lemmas~\ref{lemma: conditioned_drift} and~\ref{lemma: b_approx}}\label{sec:aux-1-d}

\begin{proof}[Proof of Lemma~\ref{lemma: conditioned_drift}] Let us find the generator of the conditioned diffusion. To that end we denote the generator of the original diffusion
by $L_\eps$:
\begin{equation}
\label{eq:generator1}
 L_\eps f(x)=b(x)f'(x)+\frac{\eps^2}{2}\sigma^2(x)f''(x)=\lim_{t\to 0}\frac{\E_x f(X_\eps)-f(x)}{t},
\end{equation}
where $f$ is any bounded $C^2$-function with bounded first two derivatives and $\E_x$ denotes expectation with respect to the
measure $\Pp_x$,
the element of the Markov family describing the Markov process emitted from point $x$.

Let us denote $u_\eps(x)=\Pp_x(B_\eps)$. This function solves the following boundary-value problem for the 
backward Kolmogorov equation:
\begin{align*}
L_\eps u_\eps(x)=0,\quad u_\eps(a_1)=0,\quad
u_\eps(a_2)=1. 
\end{align*}
Using~\eqref{eq:generator1}, it is easy to check that a unique solution is given by
\[
 u_\eps(x)=\frac{\int_{a_1}^xh_\eps(y)dy}{\int_{a_1}^{a_2}h_\eps(y)dy},
\]
where $h_\eps$ is defined in~\eqref{eq:h_eps}.

Now we can compute the generator $\bar L_\eps$ of the conditioned flow. For any smooth and bounded function $f\in C^2$ with
bounded first two derivatives, we can write
\begin{align*}
\E_x[f(X_\eps)|B_\eps]&=u^{-1}_\eps(x)\E_x f(X_\eps(t))\one_{B_\eps}
\\&=u_\eps^{-1}(x)\E_x f(X_\eps(t))\one_{B_\eps}\one_{\{\tau_\eps\ge t\}}+R_\eps
\\&=u_\eps^{-1}(x)\E_x\E_x [f(X_\eps(t))\one_{B_\eps}\one_{\{\tau_\eps\ge t\}}|\Fc_t]+R_\eps
\\&=u_\eps^{-1}(x)\E_x f(X_\eps(t)) \Pp_{X_\eps(t)}(B_\eps)+R_\eps
\\&=u_\eps^{-1}(x)\E_x f(X_\eps(t)) u_\eps(X_\eps(t))+R_\eps,
\end{align*}
where
\[
|R_\eps|=u_\eps^{-1}(x)|\E_x f(X_\eps)\one_{B_\eps}\one_{\{\tau_\eps<t\}}|\le C(x)\Pp\{\tau_\eps<t\}=o(t)
\]
for some $C(x)>0$. Therefore, we
obtain
\begin{align*}
\bar L_\eps f(x)&=\lim_{t\to 0} \frac{\E_x[f(X_\eps(t))|B_\eps]-f(x)}{t}
\\ &=\lim_{t\to 0}\frac{u_\eps^{-1}(x)\E_x f(X_\eps(t)) u_\eps(X_\eps(t))-f(x)}{t}
\\ &=\frac{1}{u_\eps(x)}\lim_{t\to 0}\frac{\E_x f(X_\eps(t)) u_\eps(X_\eps(t))-f(x)u_\eps(x)}{t}
\\ &= \frac{1}{u_\eps(x)} L_\eps(fu_\eps) (x)
\\ &= \left(b(x)+\eps^2\sigma^2(x)\frac{u'_\eps(x)}{u_\eps(x)}\right)f'(x)+\eps^2\frac{\sigma^2(x)}{2}f''(x).
\\ &= \left(b(x)+\eps^2\sigma^2(x)\frac{h_\eps(x)}{\int_{a_1}^x h_\eps(y)dy}\right)f'(x)+\eps^2\frac{\sigma^2(x)}{2}f''(x),
\end{align*}
completing the proof.
\end{proof}

\begin{proof}[Proof of Lemma~\ref{lemma: b_approx}] The proof is a variation of Laplace's method. 
Let 
\begin{equation}
\Phi (x)=2 \int_{a_1}^x \frac{ b(y)}{ \sigma^2 (y) } dy,\quad x\ge a_1, \label{eqn: def_Phi}
\end{equation}
so that $h_\eps (x) = e^{ -\Phi (x) / \eps^2 }$.  We take any $\beta \in (1,2)$ and break the integral of $h_\eps$ in two parts:
\[
\int_{a_1}^x e^{ -\Phi (y) /\eps^2 } dy = I_{\eps,1} (x) + I_{\eps,2} (x),
\]
where
\begin{equation} \label{eqn: I_eps,1}
I_{\eps,1} (x) =\int_{a_1}^{ x-\eps^\beta} e^{ -\Phi (y) /\eps^2 } dy,
\end{equation}
and 
\begin{equation} \label{eqn: I_eps,2}
I_{\eps,2} (x) =  \int_{ x-\eps^\beta} ^ x e^{ -\Phi (y) /\eps^2 } dy.
\end{equation}
The idea is to prove that $I_{\eps,1}$ is exponentially smaller than $I_{\eps,2}$ and then estimate~$I_{\eps,2}$. 

We start with some preliminaries for the function $\Phi$. Since both $b$ and $\sigma$ are $C^1$ and $\sigma \neq 0 $ in $[a_1, a_2 ]$ we conclude that $\Phi$ is a $C^2$ function so that  we can find a function $R:\R \times \R \to \R$ and a number $\delta_0>0$ such that for every $x,y \in [a_1,a_2+\delta_0]$, we have the expansion
\begin{equation} \label{eqn: Taylor_phi}
\Phi (y)= \Phi (x) + \Phi ^ \prime (x) (y -x) + R(x, y-x ),
\end{equation}
and 
\begin{equation}
|R(x,v)| \leq K_1 |v|^2,\quad x \in [a_1, a_2+\delta_0], v \in \R, \label{eqn: bound_R}
\end{equation}
for some $K_1>0$.

To estimate $I_{\eps,1}$, we introduce 
\[J_{\eps,1} (x) = \frac{e^{\Phi(x) / \eps^2}}{\eps^2\sigma^2(x)} I_{\eps, 1 } (x),\quad x \in [a_1, a_2+\delta_0].\]
Since $\Phi$ is decreasing, we have that for some constant $K_2>0$ independent of $x \in [a_1, a_2+\delta_0]$,
\begin{align} 
J_{\eps,1} (x) \leq \frac{K_2}{\eps^2}e^ { ( \Phi(x)-\Phi(x-\eps^\beta) )  / \eps^2}. \label{eqn: estimate_2}
\end{align}
Since $\beta<2$ and $\Phi^\prime$ is negative and bounded away from zero, we conclude that there is
$\alpha(\eps)$ such that $\alpha(\eps)=o (\eps ^ 2 )$ as $\eps\to 0$ and 
\begin{equation}
\label{eq:sup_estimate-on-J_1}
 \sup_{x \in [a_1, a_2+\delta_0]} J_{\eps,1} (x) \leq \alpha(\eps).
\end{equation}

We now estimate $I_{\eps,2}$. Using expansion~\eqref{eqn: Taylor_phi} and the change of variables $u=- \Phi(x) (y-x)/\eps^2$, we get
\begin{align} \notag
I_{\eps,2}(x)&= e^{ -\Phi (x) /\eps^2 } \int_{ x- \eps^\beta}^x e^{ -\Phi ^ \prime (x) (y -x)/\eps^2 - R(x, y-x ) /\eps^2 } dy \\ \notag
&=-\frac{  \eps^2  } {\Phi^\prime (x) } e^{ -\Phi (x) /\eps^2 }  \int_{\Phi^\prime(x)/\eps^{2-\beta}}^0 e^{ u - R(x,-\eps^2 u/ \Phi^\prime(x) ) /\eps^2 } du \\
&=-\frac{  \eps^2 \sigma^2 (x) } {2 b(x) } e^{ -\Phi (x) /\eps^2 } J_{\eps,2} (x), \label{eqn: Laplace_int}
\end{align}
where we use~\eqref{eqn: def_Phi} to compute the derivative of $\Phi$, and we define $J_{\eps,2}$ by~\eqref{eqn: Laplace_int}.
Hence, combining ~\eqref{eqn: estimate_2} with the definition of $b_\eps$ and~\eqref{eqn: Laplace_int}, we get 
\begin{align*}
b_\eps (x) = b(x) + \frac{1}{J_{\eps,1}(x) - \frac{1}{2 b(x)} J_{\eps,2} (x) } .
\end{align*}
Due to \eqref{eq:sup_estimate-on-J_1}, the proof will be complete once we prove that for sufficiently small $\delta>0$,
\[
\limsup_{\eps \to 0}  \eps^{-2}\left(\sup_{x \in [x_0-\delta, a_2+\delta]} | J_{\eps,2} (x) -1 | \right) < \infty.
\]

Note that for any $\delta\in(0,x_0-a_1)$, some constant $K_3=K_3(\delta)>0$ and all $x \in [x_0-\delta,a_2+\delta]$,
\begin{align} \notag
|J_{\eps,2} (x)-1| &=\Bigl| \int_{\Phi^\prime (x)/\eps^{2-\beta} }^0 e^u (1 - e^{ - R(x,-\eps^2 u/ \Phi^\prime(x) ) /\eps^2 } )du  \\ \notag
& \quad +\int_{-\infty}^{\Phi^\prime (x)/\eps^{2-\beta} } e^u du\Bigr| \\
&\leq  \int_{\Phi^\prime (x)/\eps^{2-\beta}}^0 e^u | 1 - e^{ - R(x,-\eps^2 u/ \Phi^\prime(x) ) /\eps^2 } |du  + e^{- K_3/\eps^{2-\beta} }. \label{eqn: approx_int}
\end{align}
Using~\eqref{eqn: bound_R} we see that for some constant $K_4>0$ independent of $x \in [x_0-\delta,a_2+\delta]$ and $u \in \R$, 
\[
| R(x,-\eps^2 u/ \Phi^\prime(x) ) |/\eps^2 \leq K_4 \eps^2 u^2.
\]
In particular, 
\[
\sup_{x \in [x_0-\delta,a_2+\delta]}\sup_{ u \in [\Phi^\prime (x)/\eps^{2-\beta} , 0] } | R(x,-\eps^2 u/ \Phi^\prime(x) ) |/\eps^2 \leq K_4 \eps^{2 ( \beta -1) }.
\]
Since $\beta>1$, the r.h.s.\ converges to $0$ and we can apply a basic Taylor estimate which implies that for all $\eps>0$ small enough,
\[
 \sup_{x \in [x_0-\delta,a_2+\delta]} \sup_{ u \in [\Phi^\prime (x)/\eps^{2-\beta} , 0] } | 1 - e^{ - R(x,-\eps^2 u/ \Phi^\prime(x) ) /\eps^2 } | \leq K_5 \eps^2 u^2,
\]
for some $K_5>0$. Using this fact in the integral of~\eqref{eqn: approx_int}, we can find a constant $K_6=K_6(\delta)>0$ such that
\[
\sup_{x \in [x_0-\delta, a_2+\delta]}|J_{\eps,2}(x)-1| \leq K_6 \eps^2 + e^{-K_3/\eps^{2-\beta}}, 
\]
which finishes the proof.
\end{proof}

\chapter{Conclusion} \label{ch: conclusion}

This chapter is devoted to give further discussion of the topics covered in this text. In Section~\ref{sec: conc-normalform} we made some comments related to the application of normal forms. In Section~\ref{sec: conc-saddle} we comment about the exit problem in the case where the deterministic flow has a unique saddle point.
In Section~\ref{sec: conc_scal} we present an open problem related to scaling limit and show a possible relation with Chapter~\ref{ch: levinson}.

\section{Normal Forms}~\label{sec: conc-normalform}

In this thesis, a transformation (normal form transformation) is used to conjugate the original equation for $\x$ into a non-linear perturbation of the linearized equation. This is done so we can avoid an approximation step between our original equation and its linearization. There is evidence that a similar methodology has been in the mind of the researchers since the publication of~\cite{DaySPA}. A concrete conjugation of the original equation into the linearized system was used in~\cite{nhn}. Although this result was successful, it required certain assumptions that are removed in this work (in the $2$-dimensional setting) by conjugating to a non-linear system instead. As far as the author knows, it is the first time a program of this nature has been successful.

Inspired by~\cite{GentzBook}, in~\cite{Forgoston1} a normal form transformation was applied to an epidemiological model. Contrastingly to our case, this is a specific equation and not an abstract setting. In~\cite{NormalFormsArnold} the normal form theory is presented for stochastic differential equations in the abstract setting. Although the transformation is presented, no estimates are computed as in this work. Further the development of normal form theory in~\cite{NormalFormsArnold} is not complete. For example, it excludes the one-resonant case.

Although normal form theory has proved to be a powerful tool in dynamical systems, in probability is still not clear how powerful the theory really is.  In this text we use the very explicit shape of the nonlinearity in the normal form to obtain specific estimates that successfully lead to a complete solution of the problem in Chapter~\ref{ch: Saddle}. As far as the author knows, this work is the first time in which normal form theory is applied in an abstract setting and is used to obtain tight estimates that lead to a solution of a probabilistic problem. The approach presented has some further generalizations in which the application of normal forms may be useful. We give a brief presentation about the possible complications that may be found. 
 
\section{Escape from a Saddle: further generalizations.} \label{sec: conc-saddle}
 
In this work we have studied the exit problem for small noise diffusions. In particular, we have shown the existence of possible asymmetries in the case in which the flow generated by the drift admits a saddle point. The proof is restricted to the $2$-dimensional setting. Let us discuss about this particular restriction.

Our method of proof was to transform the original equation into a very specific non-linear equation known as normal form. Then, we obtained several estimates that intensively uses the smallness of the noise and the specific form of the nonlinearity in the normal form. Let us recall the form of the nonlinearity.

In our case, the nonlinearity in the normal form is given by a finite sum of resonant monomials (see Section~\ref{sec: Normal_Forms} ) of the form $(x_1^{ \alpha_1^+}x_2^{ \alpha_2^+},x_1^{ \alpha_1^-}x_2^{ \alpha_2^-} )$, where $(\alpha_1^\pm,\alpha_2^\pm) \in \Z^2$ satisfy the resonance relations
\[
\alpha_1^\pm \lambda_+ - \alpha_2^\pm \lambda_- = \pm \lambda_\pm,
\]
of some order $r=\alpha_1^\pm + \alpha_2^\pm \geq 2$.
If we were to generalize the argument in Chapter~\ref{ch: Saddle} to the $d$-dimensional case, we would need to take into account the particular form that the nonlinearity would have in the normal form. Indeed, there are two points to consider:
\begin{enumerate}
\item \label{item: conc-res-form} The resonant monomials of order $r \geq 2$ are of the form 
\[
(x_1^{ \alpha_{1,1} } \cdots x_d^{ \alpha_{1,d} }, ... ,x_1^{ \alpha_{d,1} } \cdots x_d^{ \alpha_{d,d} } ),
\]
where the vector $\alpha_i = ( \alpha_{i,1},...,\alpha_{i,d} ) \in \Z^d$ satisfies
\begin{align*}
\alpha_{i,1} \lambda_1 + ... + \alpha_{i,d} \lambda_d &= \lambda_i \\
\alpha_{i,1} + ... + \alpha_{i,d} &=r ,  
\end{align*}
for each $i=1,...,d$. Here $\lambda_1,...,\lambda_d$ are the eigenvalues of the matrix $\nabla b (0 )$.

\item \label{item: conc-nonl} According to ~\cite[Theorem~3,Section~2]{IIlyashenko}, the nonlinearity $N$, after being transformed by a normal form transformation of degree $R>1$, will be of the form 
\[
N(x)=P(x)+Q(x),
\]
where $P$ is a finite sum of resonant monomials, and $Q$ is a correction of order $|x|^{R+1}$ (as $|x|\to0$) when the vector of eigenvalues $\lambda=(\lambda_1,...,\lambda_d)$ is not one-resonant and identically $0$ when $\lambda$ is one-resonant.
\end{enumerate}

The first point implies that to obtain the exponents $\alpha_{i,j}$ more combinatorial work than the one put in Section~\ref{sec: saddle-changevar} is needed. Still this is not the biggest difficulty. The biggest difficulty relies on the lack of structure of the correction $Q$ in the case $\lambda$ is not one-resonant. Indeed, to have any hope that our techniques in Chapter~\ref{ch: Saddle} work, we require at least that whenever $\alpha_{i,k} \neq 0$ for some $k < i$, then $\alpha_{i,j} \neq 0$ for some $k < j\leq d$ (this is in the case we order the eigenvalues as usual: $\Re \lambda_1 \geq ... \geq \Re \lambda_d$). There is no guarantee that a condition of this form holds in the not one-resonant case. In conclusion, a higher dimensional analogue for the saddle case can be obtained using the techniques presented in this theses only in the one-resonant case. This is so, unless the particular structure that the eigenvalues $\lambda_1,...,\lambda_d$ have in this case implies that a normal form transformation can be chosen so that the non-linearity of the transformed drift is a finite sum of resonant monomials with no correction. As far as the author knows, this is an unsolved issue in normal form theory.

There are still some results to be filled in order to complete the case in which the deterministic flow has a saddle, and hence the case in which it admits an heteroclinic network. This result is worthwhile pursuing since the implications of the asymmetry found in Chapter~\ref{ch: Saddle} has very interesting analogues in higher dimensions as famous chaotic systems (such as the Lorentz system) in higher dimensions exhibit homoclinic behavior (see~\cite[Chapters 27,30 and 31]{Wiggins} for further examples).

\subsection{A non-smooth transformation alternative}

As discussed in Section~\ref{sec: saddle-changevar}, it is possible to conjugate a nonlinear equation to a linear one. The restriction for It\^o equations is that this transformation has to be at least $C^2$. Recent results have extended It\^o's formula for functions with less smoothness. The first result of this nature is the well known Tanaka's formula~\cite[Chapter IV]{Protter}, which relies on the existence of local time for one dimensional semimartingales to extend the range of applicability of It\^o's formula to convex functions. For higher dimensional semimartingales there is no local time, so there was no immediate high dimensional analogue for Tanaka's formula. For a long time Tanaka's formula remain the more general change of variables (in terms of smoothness requirements) known. Recent studies have established change of variables for higher dimensional semimartingales with less smoothness~\cite{RussoC1},~\cite{RussoChapter}, ~\cite{FollmerProtter}, ~\cite{Eisenbaum}. Let us give a brief (and informal) comment about this setting.

Consider $f:\R^d \to \R$ to be a continuously differentiable function. Let $Z$ be a semimartingale in $\R^d$ with $Z(t)=V(t)+M(t)$, $M$ being a martingale and $V$ a stochastic process with bounded variation paths. Then~\cite{RussoC1},~\cite{RussoChapter}, ~\cite{FollmerProtter}, ~\cite{Eisenbaum} agree that if the quadratic covariation $[f(Z),Z^j]$ is well defined for every $1 \leq j \leq d$, then It\^o's formula holds:
\[
f(Z(t))=f(Z(0)) + \int_0^t \nabla f (Z(s))dZ(s) + \frac{1}{2} \sum_{i,j=1}^d [\partial_{x_j} f(Z),Z^j] (t).
\]
We recall the definition of quadratic covariation (see~\cite[Section V.5]{Protter}):
\begin{definition}
Let $H$ and $J$ be two continuous stochastic processes in $\R$. The quadratic covariation of $H$ and $J$, denoted as $[H,J]$, is, when it exist, the continuous process of finite variation over compacts, such that for any sequence $\sigma_n$ of random partitions tending to the identity,
\begin{equation} \label{eqn: cocn-quad}
[H,J]=H(0)J(0)+\lim_{n \to \infty } \sum_i ( H^{T^n_{i+1}}-H^{T^n_{i}} )( H^{T^n_{i+1}}-H^{T^n_{i}} ),
\end{equation}
uniformly over compacts in probability. Here, for any random $S>0$, the process $H^S$ is short for $t \mapsto H(t \wedge S)$, and $\sigma_n$ is the sequence $0=T^n_0 \leq ... \leq T^n_{k_n}$, where $\sup_{i} (T^n_{i+1} -T^n_i ) \to 0$, $k_n \to \infty$, and, $T^n_{k_n} \to \infty$ as $n\to \infty$.
\end{definition}

For our diffusion process $\x$, there are several problems to consider. One is to show that $[\partial_{x_j} f(\x),\x^i]$ is well defined. The other, is to prove that 
\begin{equation}\label{eqn: C1est}
\eps^{-1} [\partial_{x_j} f(\x),\x^i] \stackrel{\Pp}{\longrightarrow} 0, \quad \eps \to 0.
\end{equation}
Once this is established, the result in Chapter~\ref{ch: Saddle} follows immediately. 

In order to show that $[\partial_{x_j} f(\x),\x^i]$ is well defined, the proposal in~\cite{RussoC1},~\cite{RussoChapter}, ~\cite{FollmerProtter} is to use the theory of reversible diffusions proposed in~\cite{NualartTimeReversal}. Indeed, assume for a moment that we know that, for a fixed $T>0$, $\hat{X}_\eps (t)=\x(T-t)$ is a diffusion. Then, observe that~\eqref{eqn: cocn-quad} for $H=\partial_{x_j} f(\x)$ and $J=\x^i$ can be written as 
\begin{equation} \label{quad}
[\partial_{x_j} f(\x),\x^i](t)=-\int_0^t \partial_{x_j} f(\x(s))d \x^i (s) - \int_{T-t}^T \partial_{x_j} f(\hat{X} (s) )d \hat{X}^i_\eps (s),
\end{equation}
where both integrals are It\^o integrals with respect to different filtrations. In order to use this formula to prove~\eqref{eqn: C1est} the first attempt may be to get the generator of $\hat{X}_\eps$. Under several assumptions (the most important one being the ellipticity of the noise) in~\cite{NualartTimeReversal} it is proved that, for a fixed time $T>0$, $X(T-\cdot)$ is also a diffusion with the same diffusion matrix and with drift $\hat{b}=-b(x)+ \nabla \log p_{T-t,T}(x,\x(T))$, where $p_{T-t,T}$ is the transition density of the Markov process $\x$ (which existence is proved in~\cite{NualartTimeReversal}).  Hence in order to establish~\eqref{eqn: C1est} we first would need to have a bound in $\eps$ of $\nabla \log p_{T-t,T}$. This quantity is of interest in control theory~\cite{FlemingBook}, but there is, as far as the author knows,no reference to an estimate in $\eps>0$. Another option that avoids this estimate is to extend the filtration $\mathcal{F}^W$ to the minimum complete filtration that includes $\mathcal{F}^W$ such that $\x(T)$ is measurable and write Doob-Meyer decomposition for the process $\x$ with respect to this filtration. 

This is still undergoing work, that is promising not only because it allows to prove the results included in this thesis, but also because it uses several tools of modern stochastic analysis. 

\section{Scaling limits} \label{sec: conc_scal}

In this thesis we proved a scaling limit for the exit problem for two cases, the case in which the flow $S$ has a unique saddle and the Levinson case. The idea will be to prove scaling limits for more general systems. In particular, recall that if the quasipotential has a unique minimizer $z$, then the exit point $\x(\tau_\eps)$ converges to it in probability as $\eps \to 0$. By a scaling limit, we mean find an $\alpha>0$ such that the distribution of $\eps^{-\alpha } ( z - \x( \tau_\eps^D ) )$ is tight. 

Let $V:D \times \partial D \to [0,\infty)$ be the quasipotential given by~\eqref{eqn: quasipotential}:
\[
V(x,y) = \inf_{T>0} \left \{ I_T^x ( \varphi) : \varphi (T)=y, \varphi( [0,T] ) \subset D \cup \partial D \right \}
\]
In order to state the claim, given $x_0 \in D$, let $\mathcal{M}_{x_0} \subset \partial D$ be the set of minimizers of $y \mapsto V(x_0,y)$. The claim is the following:

\begin{claim}
Suppose $\mathcal{M}_{x_0}$ is finite $\mathcal{M}_{x_0}=\{e_1,...,e_q \}$. There is a probability distribution $\nu$ over $\mathcal{M}$, a random number $\alpha \in (0,1]$ and a family of random variables $(\xi_\eps)_{\eps>0}$ such that the exit can be written as
\[
\x ( \tau_\eps )= \nu_1 e_1 + ... + \nu_q e_q + \eps^\alpha \xi_\eps.
\]
Further, there is a random variable $\xi_0$ so that $\xi_\eps \to \xi_0$ in distribution as $\eps \to 0$. 
\end{claim}

The results of this thesis imply this claim in the case the flow $S$ admits an heteroclinic network (see Section~\ref{sec: Intro-PlanarHetero}). The proof was done by solving two simple cases (saddle point and Levinson case) and then using a Poincar\'e distributional map argument for each critical point in the network. 

Here we shall proceed similarly: start from simple cases with random initial conditions so that a Poincar\'e argument can be applied. The proposal is to choose as the base case the well developed stable case~\cite{Freidlin--Wentzell-book}: $0 \in D$ and $D$ is contained in the basin of attraction of $0$. It is known that if the domain $D$ is attracted to the origin and $M_{x_0}=\{e\}$ then $\x(\tau_\eps) \to e$ in probability. Moreover, if there is a unique extreme trajectory $\varphi_0$ (the one that realizes the minimum in $V$) then for every $\delta>0$,
\[
\lim_{ \eps \to 0} \Pp_{x_0} \left \{ \sup_{ \theta_\eps \leq t \leq \tau_\eps } | \x(t) - \varphi_0 ( t - \theta_\eps + \theta_0 ) |< \delta  \right \}=0,
\]
where $\theta_\eps$ ($\theta_0$) is the last time $\x$ ($\varphi_0$) hits a ball of arbitrary small (but fixed) radius around the origin. From the perspective introduced in Chapter~\ref{ch: levinson}, consider the process conditioned on exit close to $e$. This process is a semimartingale with the same diffusion matrix as the original process, but with a drift of the form $b_\eps = -b(x) + \eps^2 \varphi_\eps (t,x)$, where $\varphi_\eps$ is uniformly bounded. Hence, our results in the Levinson case of Chapter~\ref{ch: levinson} apply. 

Once this result is established, we can follow the same pattern as in~\cite{Freidlin--Wentzell-book} to study possible asymmetric behavior in metastable process. With this development we can show that the idea of random Poincar\'e maps apply to a general dynamical system.

\appendix
\chapter{Large Deviations} \label{ch: appendix}

Large deviation theory is a mixture of probability theory, analysis, variational calculus, point set topology among others. This theory has been used for different purposes. In Section~\ref{sec: intro_FW} we discussed  the role played by large deviation theory in the development of Freidlin-Wentzell theory. The purpose of this chapter is to provide a quick reference to large deviation theory as needed to understand Section~\ref{sec: intro_FW}.

We present the general theory of large deviations. The theory was first formulated in the right degree of abstraction by Varadhan~\cite{Varadhan66}, we follow~\cite{Vares} in this exposition. In Section~\ref{sec: app-LDP} we begin with the basic definitions. In Section~\ref{app: FW-LDP} we present the large deviation results related to diffusion processes.

\section{Large Deviations Principle (LDP) } \label{sec: app-LDP}

Let $\mathcal{X}$ be a Polish metric space with metric function $d:\mathcal{X} \times \mathcal{X} \to [0,\infty)$. By a probability measure on $\mathcal{X}$, we mean a probability measure on the Borel sigma algebra on $\mathcal{X}$. We will give the general definition of large deviation principle for a family of probability measures on $\mathcal{X}$. First, recall the following definition.
\begin{definition} The function $f: \mathcal{X} \to [-\infty,\infty]$ is lower semi-continuous if it satisfies any of the following equivalent properties:
\begin{enumerate}
\item $\liminf_{ n \to \infty } f (x_n) \geq f(x)$ for all sequences $(x_n)_{n\in \N} \subset \mathcal{X}$ and all points $x \in \mathcal{X}$ such that $x_n \to x$ in $\mathcal{X}$.
\item For all $x \in \mathcal{X}$, $\lim_{ \delta \to 0} \inf_{ y \in B_\delta (x) } f(y)=f(x)$, where $B_\delta(x)=\{ y\in \mathcal{X}: d(x,y)<\delta \}$.
\item $f$ has closed level sets, that is, $f^{-1}( [ -\infty,c] )=\{ x \in \mathcal{X}: f(x) \leq c  \}$ is closed for all $c \in \R$.
\end{enumerate}
\end{definition}

Here are the key definitions of large deviation theory:
\begin{definition}
The function $I:\mathcal{X}\to[0,\infty]$ is called a rate function if 
\begin{enumerate}
\item $I \not \equiv \infty$,
\item $I$ is lower semi-continuous,
\item $I$ has compact level sets.
\end{enumerate}
\end{definition}

\begin{definition} \label{def: app-LDP}
A family of probability measures $(\Pp)_{\eps>0}$ on $\mathcal{X}$ is said to satisfy , as $\eps \to 0$,the large deviation principle (LDP) with rate $\alpha_\eps \to 0$ and rate function $I$ if
\begin{enumerate}
\item $I$ is a rate function,
\item $\limsup_{\eps \to 0 } \alpha_\eps \log \Pp_\eps ( C ) \leq -I ( C )$, for every $C \subset \mathcal{X}$ closed,
\item $\liminf_{\eps \to 0 } \alpha_\eps \log \Pp_\eps ( O ) \geq -I ( O )$, for every $O \subset \mathcal{X}$ open.
\end{enumerate}
Here the bounds are in terms of the set function defined by
\[
I(S)=\inf_{ s \in S } I(x), \quad S \subset \mathcal{X}.
\]
\end{definition}

The goal of large deviation theory is to build up an arsenal of theorems based on these two definitions. We will not describe most of this theorems, since they are out of the scope for the present text. The interested reader is invited to consult the standard monographs on the subject~\cite[Chapter 4]{DemboZeitouni},~\cite[Chapter III]{DenHollanderLDP},~\cite[Chapter 2]{Vares}. The only theorem that we cite is the so called contraction principle. First, we give some remarks
\begin{remark}
\begin{enumerate}
\item It is a standard exercise to show that once the large deviation principle is satisfied, the rate function $I$ is unique.
\item In Definition~\ref{def: app-LDP} it is crucial to make a difference between open and closed sets. Naively, one might try to replace the second and third conditions with the stronger requirement that
\[
\lim_{ \eps \to 0} \alpha_\eps \Pp_{\eps} (S) = -I(S), \quad S \subset \mathcal{X}.
\]
However, there are examples that show that this would be far too restrictive. 
\end{enumerate}
\end{remark}

We now present the contraction principle:
\begin{theorem} \label{them: app-cp}
Let $( \Pp )_{\eps >0} $ be a family of probability measures on $\mathcal{X}$ that satisfies the LDP, as $\eps \to 0$, with rate function $\alpha_\eps$ and with rate function $I$. Let $\mathcal{Y}$ be a Polish space, $T:\mathcal{X} \to \mathcal{Y}$ a continuous map, and $\mathbb{Q}_\eps=\Pp_\eps \circ T^{-1}$ an image probability measure. Then, the family $(\mathbb{Q}_\eps)_{ \eps>0 }$ satisfies the LDP on $\mathcal{Y}$ with rate $\alpha_\eps$ and with rate function $J$ given by 
\[
J(y)=\inf_{x\in \mathcal{X}: T(x)=y } I(x),
\]
with the convention $\inf_\emptyset I = \infty$.
\end{theorem}

\section{Freidlin-Wentzell LDP} \label{app: FW-LDP}

In this section we present the large deviation results that Freidlin-Wentzell theory is based on.

Given $T>0$, let $W(t), t\in[0,T]$, be a standard Brownian motion in $R^d$. Consider the process $W_\eps (t) = \eps W(t)$, and let $\Pp_\eps^W$ be the probability measure induced by $W_\eps$ on $C([0,T];\R^d)$, the space of all continuous functions $\varphi:[0,T] \to \R^d$ equipped with the supremum norm topology. We first state the LDP for $W_\eps$ derived by Schilder~\cite{SchidlerLDP}:
\begin{theorem} \label{thm: app-LDP-W}
The family of probability measures $(\Pp_\eps^W)_{\eps >0}$ on $C( [0,T];\R^d )$ satisfy a LDP with rate $\eps^2$ and with rate function
\[
J_T ( \phi) = \left \{
\begin{tabular}{ l c l }
$\frac{1}{2} \int_0^T |\dot{\phi}(s)|^2 ds$  & , &  $\phi \in H^1$ \\
$\infty$ & , & otherwise \\
\end{tabular} \right.
\]
Here $H^1$ is the space of absolutely continuous functions with square integrable derivative.
\end{theorem}

The simple case in which the process $\x$ is the strong solution of
\[
d \x (t) = b ( \x (t ) ) dt + \eps d W (t)
\]
is a consequence of Theorem~\ref{thm: app-LDP-W} and the Contraction Principle~\ref{them: app-cp}. Indeed, let $F:C([0,T];\R^d) \to C([0,T];\R^d)$ be the map defined by $f=F(g)$, where $f$ is the unique solution of
\[
f(t)=\int_0^T b (f(s) ) ds + g (t).
\]
Then, after noticing that $F$ is continuous and some calculation, Theorem~\ref{thm: app-LDP-W} implies the following result:
\begin{corollary}
The law of $\x$ on $C( [0,T];\R^d)$ satisfies a LDP with rate $\eps^2$ and with rate function
\[
J_T^\prime =\left \{
\begin{tabular}{ l r l }
$\frac{1}{2} \int_0^T |\dot{\phi}(s)-b(s)|^2 ds$  & , & $\phi \in H^1$ \\
$\infty$ & , & otherwise \\
\end{tabular} \right. .
\]
\end{corollary}
Now, consider $\x$ to be the solution of our typical SDE
\[
d\x (t) = b( \x (t) ) dt + \eps \sigma ( \x (t ) ) dW(t).
\]
As said on Section~\ref{sec: intro_FW} a LDP for this process is the base of Freidlin-Wentzell theory. It turns out that to obtain a LDP for the law of $\x$ the contraction principle does not apply. Instead, raw approximations have to be made. We state the theorem without a proof (see ~\cite[Section 4.2]{DemboZeitouni} or ~\cite[chapter 3]{Freidlin--Wentzell-book} for a proof).

\begin{theorem} [Freidlin-Wentzell~\cite{Freidlin--Wentzell-book} ]
Let $H_{0,T}^1$ be the space of all absolutely continuous functions from $[0,T]$ to $\R^d$ with square integrable derivatives. Define the functional $I_T^x$ by
\begin{equation}  \label{eqn: LDP_I}
I_T^x (\varphi)= \frac{1}{2} \int_0^T \langle \stackrel{\cdot}{\varphi} (s)-b(\varphi(s) ), a^{-1} (\varphi (s) ) ( \stackrel{\cdot}{\varphi} (s)-b(\varphi(s) ) ) \rangle ds,
\end{equation}
if $\varphi \in H_{0,T}^1$ and $\varphi(0)=x$, and $\infty$ otherwise. Here $b$ is the drift in~\eqref{eqn: Ito_additive} and $a=\sigma^T \sigma$, with $\sigma$ the diffusion matrix in~\eqref{eqn: Ito_additive}.

Then for each $x \in \R^d$ and $T>0$ the family $(\Pp_x^\eps)_{\eps>0}$ satisfies a Large Deviation Principle on $C([0,T];\R^d)$ equipped with uniform norm at rate $\eps^2$ with good rate function $I_T^x$.
\end{theorem}

\chapter{Appendix to Section~\ref{subsec: sketch} } \label{ch: app-sec}

The purpose of this appendix is to present any technical material left out in Chapter~\ref{ch: Intro} Section~\ref{subsec: sketch}. This appendix (in contrast with Appendix~\ref{ch: appendix}) contains original material in the simple case discussed on Section~\ref{ch: app-sec}. Let us recall the setting from that section.

Given two positive numbers $\lambda_\pm >0$, consider the diffusion $\x=(x_\eps^1,x_\eps^2)$
\begin{align*}
d\x (t) &= { \rm{diag} }( \lambda_+, - \lambda_- ) \x (t) dt + \eps dW(t).
\end{align*}
Let $\delta>0$ and $D=(-\delta,\delta) \times (-\delta,\delta)\subset \R^2$. We study the exit problem of $\x$ from $D$. We start the diffusion $\x$ inside $D$: $\x (0 ) = (0,x_0) \in D$.

Recall that in Section~\ref{subsec: sketch}, we used It\^o's formula in each coordinate to write Duhamel principle for $x_\eps^1$ and $x_\eps^2$. Here we rewrite identities~\eqref{eqn: Duh_Intro1} and~\eqref{eqn: Duh_Intro2} for easier reference:
\begin{align} \label{eqn: app-Duh_Intro1}
x_\eps^1 (t) &= \eps e^{ \lambda_+ t } \int_0^t e^{-\lambda_+ s} dW_1 (s), \\ \label{eqn: app-Duh_Intro2}
x_\eps^2 (t) &= e^{-\lambda_- t } x_0 +  \eps  \int_0^t e^{-\lambda_- ( t- s) } dW_2 (s).
\end{align}
Let $\mathcal{N}(t)$ denote the stochastic integral in~\eqref{eqn: app-Duh_Intro1}.

Recall that $\tau_\eps^\delta$ is defined as
\[
\tau_\eps^\delta = \inf \left\{ t>0: |x_\eps^1 (t)| \geq \delta \right \}.
\]
First we prove that $\tau_\eps^\delta$ is finite with probability $1$. This is a general fact that can be found in the literature, for example in~\cite[Proposition 1.8.2]{BassBook}, but we chose to prove it directly from Duhamel principle. We do this in order to stress the importance of such a representation in our setting. Without any further discussion, we go into the results.

\begin{lemma}
For every $\delta>0$ and $\eps >0$, $\tau_\eps^\delta < \infty$ $\Pp-$a.s.
\end{lemma}
\begin{proof}
Let $n \in \N$, it is enough to show that $\Pp \{ \tau_\eps^\delta > n \} \to 0$ as $n \to \infty$. Observe that~\eqref{eqn: app-Duh_Intro1} implies that
\begin{align*}
\Pp \{ \tau_\eps^\delta > n \} &= \Pp \left\{ \sup_{ t \in [0,n] } \eps e^{ \lambda_+ t } | \mathcal{N} (t) | < \delta \right \} \\
&\leq \Pp \left\{ \sup_{ t \in [n/2,n] } \eps e^{ \lambda_+ t } | \mathcal{N} (t) | < \delta   \right \} \\
&\leq \Pp \left\{ \eps e^{ \lambda_+ n / 2 } \sup_{ t \in [n/2,n] }  | \mathcal{N} (t) | < \delta   \right \}.
\end{align*}
Here the last two inequalities follow from the properties of the supremum and the exponential function respectively. Take $n_0 \in \N$ such that $ \eps^{-1} \delta< e^{\lambda_+ n /4 }$, for every $n \geq n_0$. Then, for every $n \geq n_0$,
\begin{align*}
\Pp \left\{ \eps e^{ \lambda_+ n / 2 } \sup_{ t \in [n/2,n] }  | \mathcal{N} (t) | < \delta   \right \} & \leq
\Pp \left\{ \eps e^{ \lambda_+ n / 2 } \sup_{ t \in [n/2,n] }  | \mathcal{N} (t) | < \delta , e^{ \lambda_+ n / 4 } \sup_{ t \in [n/2,n] }  | \mathcal{N} (t) | \geq 1  \right \} \\
& \quad + \Pp \left\{ e^{ \lambda_+ n / 4 } \sup_{ t \in [n/2,n] }  | \mathcal{N} (t) | < 1  \right \} \\
& \leq \Pp \left\{ e^{ \lambda_+ n / 4 } \sup_{ t \in [n/2,n] }  | \mathcal{N} (t) | < 1  \right \}.
\end{align*}
The proof will be finished as soon as we can show that the last probability converges to $0$. To see this, note that, for every $t>0$, the random variable $\mathcal{N}(t)$ is a zero mean gaussian random variable with variance
$ \int_0^t e^{ - 2 \lambda_+  t} ds =\frac{1 - e^{ - 2 \lambda_+ t } } { 2 \lambda_+ }. $
Denote
\[
\alpha_n = e^{- \lambda_+ n / 4 } \sqrt { \frac{ 2 \lambda_+ }{1 - e^{ - 2 \lambda_+ n } } }.
\]
The result follows since, $\alpha_n \to 0$, as $n \to \infty$, and
\begin{align*}
\Pp \left\{ e^{ \lambda_+ n / 4 } \sup_{ t \in [n/2,n] }  | \mathcal{N} (t) | < 1  \right \}
& \leq \Pp \left\{ e^{ \lambda_+ n / 4 }  | \mathcal{N} (n)| < 1  \right \} \\
& =\frac{1}{ \sqrt{2 \pi} } \int_{ - \alpha_n }^{\alpha_n} e^{ - r^2/2 } dr.
\end{align*}
\end{proof}

\begin{lemma} \label{lemma: app-tauinfty}
For any $\delta > 0$,
\[
\tau_\eps^\delta \stackrel{\Pp}{\longrightarrow} \infty,
\]
as $\eps \to 0$.
\end{lemma}
\begin{proof}
Let $\delta>0$. It is enough to prove that $\Pp \{ \tau_\eps^\delta > T \} \to 0$, $\eps \to 0$, for any $T>0$. Use Duhamel principle~\eqref{eqn: app-Duh_Intro1} to get
\begin{align*}
\Pp \{ \tau_\eps^\delta > T \} &= \Pp \left \{ \sup_{ t \leq T \wedge \tau_\eps^\delta } | x_\eps^1 (t) | > \delta \right \} \\
&\leq  \Pp \left \{ \eps e^{ \lambda_+ T } \sup_{ t \leq T \wedge \tau_\eps^\delta } \left | \int_0^t e^{ - \lambda_+ s} ds \right | > \delta \right \}.
\end{align*}
The last inequality, Chebyshev inequality~\cite[Section 3.2]{ChungBook}, BDG inequality~\cite[Proposition 3.3.28]{Karatzas--Shreve} and It\^o isometry~\cite[Proposition 2.10]{Karatzas--Shreve} imply that for some constant $C_1$,
\begin{align*}
\Pp \{ \tau_\eps^\delta > T \} & \leq C_1 e^{2 \lambda_+ T} \eps^2 \mathbf{E} \int_0^{T \wedge \tau_\eps^\delta}  e^{ -2 \lambda_+ s} dW(s) \\
& \leq \frac{ C_1 } { \lambda_+ } e^{2 \lambda_+ T} \eps^2.
\end{align*}
This proves our result.
\end{proof}

The last technical step in this appendix is about the convergence of the random variable $\mathcal{N}_\eps$. Recall that $\mathcal{N}_\eps = \mathcal{N} ( \tau_\eps^\delta ) $ and
\[
\mathcal{N} = \int_0^\infty e^{ - \lambda_+ s } dW(s).
\]
\begin{lemma}
As $\eps \to 0$, $\mathcal{N}_\eps \to \mathcal{N}$ in probability.
\end{lemma}
\begin{proof}
The lemma is a consequence of It\^o isometry and Lemma~\ref{lemma: app-tauinfty}. Let $\gamma>0$ and $T_\gamma=-( 2 \lambda_+ )^{-1} \log ( \gamma \lambda_+ )>0$. Due to Lemma~\ref{lemma: app-tauinfty} we can find $\eps_0>0$ such that
\begin{equation} \label{eqn: app-gamma-big}
\Pp \{ \tau_\eps^\delta > T_\gamma \} \leq   \gamma \lambda_+,
\end{equation}
for every $\eps \in (0, \eps_0 )$. Use It\^o isometry and~\eqref{eqn: app-gamma-big} to obtain
\begin{align*}
\mathbf{E} | \mathcal{N}_\eps - \mathcal{N} |^2 &= \mathbf{E} \int_{\tau_\eps^\delta}^\infty e^{- 2 \lambda_+ s} ds \\
& \leq \frac{ e^{- 2 \lambda_+ T_\gamma} } {2 \lambda_+ } + \frac{1}{ 2 \lambda_+ } \Pp \{ \tau_\eps^\delta < T_\gamma \} \\
& \leq \gamma,
\end{align*}
for every $\eps \in (0, \eps_0)$. The result follows since $\gamma>0$ is arbitrary and $L^2$ convergence implies convergence in probability.
\end{proof}

\bibliography{happydle}{}
\bibliographystyle{plain}

\end{document}